\documentclass[10pt]{article}
\usepackage[english]{babel}
\usepackage{amsmath}
\usepackage{amsfonts}
\usepackage{amssymb}
\usepackage{amsthm}
\usepackage{amscd}
\usepackage{fullpage}
\usepackage{xcolor}
\usepackage{bbm}
\usepackage{enumerate}
\usepackage{hyperref}
\usepackage{caption}
\usepackage[normalem]{ulem}

\usepackage[pdftex]{graphicx} 

\newcommand{\eqd}{\overset{(d)}{=}}
\newcommand{\Z}{\mathbb{Z}}
\newcommand{\D}{\mathbb{D}}
\newcommand{\R}{\mathbb{R}}
\newcommand{\N}{\mathbb{N}}

\newcommand{\C}{\mathbb{C}}

\newcommand{\E}{\mathbb{E}}

\newcommand{\mc}{\mathcal}
\newcommand{\mb}{\mathbb}

\newcommand{\eps}{\varepsilon}

\renewcommand{\P}{\mathbb{P}}

\DeclareMathOperator{\Var}{Var}

\newcommand{\Pro}{\mathbb{P}}

\newcommand{\norme}[1]{\left \| #1 \right \|}

\newcommand{\ol}[1]{\overline{#1}}

\newcommand{\vertiii}[1]{{\left\vert\kern-0.25ex\left\vert\kern-0.25ex\left\vert #1 
    \right\vert\kern-0.25ex\right\vert\kern-0.25ex\right\vert}}

\DeclareMathOperator{\Cov}{Cov}

\DeclareMathOperator{\diam}{diam}

\def\cH{\mathcal{H}}

\def\cB{\mathcal{B}}

\def\alb#1\ale{\begin{align*}#1\end{align*}}
\def\allb#1\alle{\begin{align}#1\end{align}}

\newcommand{\aryb}{\begin{eqnarray*}}
\newcommand{\arye}{\end{eqnarray*}}
\def\alb#1\ale{\begin{align*}#1\end{align*}}
\newcommand{\eqb}{\begin{equation}}
\newcommand{\eqe}{\end{equation}}
\newcommand{\eqbn}{\begin{equation*}}
\newcommand{\eqen}{\end{equation*}}

\newcommand{\wt}{\widetilde}
\newcommand{\wh}{\widehat}

\newtheorem{thm}{Theorem}[section]

\newtheorem{Prop}[thm]{Proposition}
\newtheorem{Lem}[thm]{Lemma}
\newtheorem{Cor}[thm]{Corollary}

\newtheorem{Rem}[thm]{Remark}

\def\Xint#1{\mathchoice
{\XXint\displaystyle\textstyle{#1}}%
{\XXint\textstyle\scriptstyle{#1}}%
{\XXint\scriptstyle\scriptscriptstyle{#1}}%
{\XXint\scriptscriptstyle\scriptscriptstyle{#1}}%
\!\int}
\def\XXint#1#2#3{{\setbox0=\hbox{$#1{#2#3}{\int}$ }
\vcenter{\hbox{$#2#3$ }}\kern-.6\wd0}}

\def\dashint{\Xint-}

\title{Volume of metric balls in Liouville quantum gravity}

\author{Morris Ang \thanks{Department of Mathematics, Massachusetts Institute of Technology, Cambridge, MA 02139, USA.} \and Hugo Falconet \thanks{Department of Mathematics, Columbia University, New York, NY 10027, USA.} \and  Xin Sun \thanks{Department of Mathematics, University of Pennsylvania, Philadelphia, PA 19104, USA.}}

\begin{document}

\maketitle

\abstract{We study the volume of metric balls in Liouville quantum gravity (LQG).
	For $\gamma \in (0,2)$, it has been known since the early work of Kahane (1985) and Molchan (1996) that the LQG volume of Euclidean balls has finite moments exactly for $p \in (-\infty, 4/\gamma^2)$. Here, we prove that the LQG volume of LQG metric balls admits all finite moments. This answers a question of Gwynne and Miller and generalizes a result obtained by Le Gall for the Brownian map, namely, the $\gamma = \sqrt{8/3}$ case. We use this moment bound to show that on a compact set the volume of metric balls of size $r$ is  given by $r^{d_{\gamma}+o_r(1)}$, where $d_{\gamma}$ is the dimension of the LQG metric space. Using similar techniques, we prove analogous results for  the first exit time of Liouville Brownian motion from a metric ball.
Gwynne-Miller-Sheffield (2020) prove that the metric measure space structure of $\gamma$-LQG a.s. determines its conformal structure when $\gamma =\sqrt{8/3}$; their argument and our estimate yield the result for all $\gamma \in (0,2)$.}

\tableofcontents

\section{Introduction}
Liouville quantum gravity (LQG) was introduced in the physics literature by Polyakov \cite{P81} as a canonical model of two-dimensional random geometry, and has also been shown to be the scaling limit of random planar maps in various topologies (see e.g.~\cite{Notice-Gwynne,MOT-Survey} and references therein).
Let $h$ be an instance of the Gaussian free field (GFF) on the plane $\C$, and fix $\gamma \in (0,2)$. Formally, the $\gamma$-LQG surface described by $(\C,h)$ is the Riemannian manifold with metric tensor given by ``$e^{\gamma h} (dx^2 + dy^2)$''.  The conformal factor $e^{\gamma h}$ only makes sense formally since the GFF $h$ does not admit pointwise values. Nevertheless, one can make rigorous sense of the $\gamma$-LQG volume measure $\mu_h$ through the following regularization and renormalization procedure by Duplantier and Sheffield~\cite{DS11}
\[\mu_h = \lim_{\eps \to 0} \eps^{\gamma^2/2} e^{\gamma h_\eps(z)} dz, \]
where $h_\eps(z)$ is 
 the average of $h$ on the radius $\eps$ circle centered at $z$. This falls into the  framework of Gaussian multiplicative chaos, see \cite{Kahane85,  RV14, Shamov16, Berestycki17}. The circle average mollification can be replaced by other alternatives.

We now explain the recent construction of the LQG metric.
For $\eps > 0$, let
\[D^\eps_h (z,w) = \inf_{P: z \to w}\int_0^1 e^{(\gamma/d_\gamma) h^*_\eps(P(t))} |P'(t)|\ dt, \]
where $h^*_\eps$ is a particular mollified version of $h$ obtained by integrating against the heat kernel, $d_\gamma$ is the dimension of $\gamma$-LQG \cite{DG18,gwynne2019kpz}, and the infimum is taken over all piecewise continuously differentiable paths from $z$ to $w$. Ding, Dub\'edat, Dunlap and Falconet \cite{DDDF19} proved that for all $\gamma \in (0,2)$ the laws of the suitably rescaled metrics $D^\eps_h$ are tight, so subsequential limits  exist as $\eps \to 0$ (see also the earlier tightness results \cite{DD16,DF18,DD18}). Building on this and several other works \cite{DFGPS, GM19confluence, GM19local}, Gwynne and Miller \cite{GM19uniqueness} showed that all subsequential limits agree and satisfy a natural list of axioms uniquely characterizing the LQG metric. So it makes sense to speak of \emph{the} LQG metric $D_h$. 

Now, we have the metric-measure space corresponding to $\gamma$-LQG. The main result of our paper is the following theorem concerning the volume of metric balls, which answers a question of \cite{GM19uniqueness} and generalizes estimates obtained by Le Gall~\cite{LeGall10} for the Brownian map. 
\begin{thm}\label{thm-main}
Fix $\gamma \in (0,2)$ and let $h$ be a whole-plane GFF normalized to have average zero
 on the unit circle. Let $\cB_s(z; D_h)$ be the $D_h$-ball of radius $s$ centered at $z$. Then
\eqb \label{eq-main}
\E \left[ \mu_h(\cB_1(0; D_h))^p \right] < \infty \quad \text{ for all }p \in \R. 
\eqe
Moreover, for any compact set $K\subset \C$ and $\eps > 0$, we have almost surely that
\begin{equation}
\label{eq:UniformVolumeEst}
\sup_{s \in (0,1)} \sup_{z \in K} \frac{\mu_h(\cB_s(z;D_h))}{s^{d_\gamma - \eps}} < \infty \quad \text{and} \quad \inf_{s\in (0,1)} \inf_{z \in K} \frac{\mu_h(\cB_s(z;D_h))}{s^{d_\gamma + \eps}} > 0.
\end{equation}
Consequently, the Minkowski dimension of $\gamma$-LQG is $d_\gamma$ almost surely. 
\end{thm}

This result is in stark contrast to the LQG volume of a deterministic bounded open set, which only has finite moments for $p \in (-\infty, 4/\gamma^2)$. 
Roughly speaking, $\mu_h(\cB_1(0; D_h))$ has finite positive moments because the metric ball $\cB_1(0; D_h)$ in some sense \emph{avoids} regions where $h$ (and thus $\mu_h$) is large. Our arguments also show~\eqref{eq-main} when we replace $h$ by $h + \alpha \log |\cdot|^{-1}$ for $\alpha < \frac\gamma2 + \frac2\gamma$ (see Propositions~\ref{Prop:MomentsWholePlane} and~\ref{Prop:NegativeMoments}).

Similar arguments allow us to prove an analogous result for the first exit time of the Liouville Brownian motion (LBM) from metric balls. Classically, Brownian motion is well defined on smooth manifolds and on some random fractals. Formally, LBM is Brownian motion associated to the metric tensor ``$e^{\gamma h }(dx^2 + dy^2)$'', and can be rigorously constructed via regularization and renormalization \cite{GRVbm,B15}. It is a time-change of an ordinary Brownian motion independent of $h$. For a set $X \subset \C$ and $z \in \C$, denote by $\tau_h(z;X)$ the first exit time of the Liouville Brownian motion started at $z$ from the set $X$. When $X$ is a deterministic bounded open set, $\tau_h(z;X)$ has finite moments for $p \in (-\infty,4/\gamma^2)$. Here, we study the case where $X$ is given by a metric ball.

\begin{thm}\label{thm-exit}
Fix $\gamma \in (0,2)$ and let $h$ be a whole-plane GFF normalized to have average zero on the unit circle. Then
\[
\E \left[\tau_h(0;\cB_1(0; D_h))^p \right] < \infty \quad \text{ for all }p \in \R.
\]
 Moreover, for any compact set $K \subset \C$ and $\eps > 0$, we have at a rate uniform in $z \in K$ that
\[\lim_{s\to 0}  \P [\tau_h(z;\cB_s(z; D_h)) \in (s^{d_\gamma + \eps}, s^{d_\gamma - \eps}) ]  = 1. \]
\end{thm}

As an application of Theorem~\ref{thm-main}, we can extend results of \cite{GMS18} to the case of general $\gamma \in (0,2)$. The following theorem resolves another question of \cite{GM19uniqueness}.
\begin{thm}\label{thm-recover}
Let  $\gamma \in (0,2)$ and $h$ be a whole-plane GFF $h$ normalized to have average zero on the unit circle.
Then the field $h$ up to rotation and scaling of the complex plane is almost surely determined by (i.e.\ measurable with respect to) the random pointed metric measure space $(\C, 0, D_h, \mu_h)$.
\end{thm}
We emphasize that the input is $(\C, 0, D_h, \mu_h)$ as a \emph{pointed metric measure space}, so in particular we forget the exact parametrization in the complex plane of $D_h$ and $\mu_h$. More precisely we view it as an element in the space of pointed metric measure spaces endowed with the local Gromov-Hausdorff-Prokhorov topology (local here refers to metric balls about the point). For the special case $\gamma = \sqrt{8/3}$, \cite{GMS18} proves an analogous theorem for the quantum disk (see also~\cite{MS16b}). Their results depend on the correspondence between the Brownian map and $\sqrt{8/3}$-LQG \cite{MS15b, MS16a, MS16b, MS16}, and rely on the estimates obtained by Le Gall~\cite{LeGall10} for the Brownian map. Theorem~\ref{thm-main} provides the estimates needed to generalize the results of \cite{GMS18} to \emph{all} $\gamma \in (0,2)$, yielding Theorem~\ref{thm-recover} and a statement of the convergence of the simple random walk on a Poisson-Voronoi tessellation of $\gamma$-LQG to Brownian motion (viewed as curves modulo time-parametrization) in the quenched sense; see Section~\ref{sec-recover}.

\paragraph*{Paper outline.} In Section~\ref{sec-background}, we discuss preliminary material about LQG. We prove the finiteness of moments statement of Theorem~\ref{thm-main} in Sections~\ref{section-positive} and~\ref{sec-negative}, which bound the positive and negative moments of the unit LQG ball volume respectively. In Section~\ref{sec-UVE}, we complete the proof of Theorem~\ref{thm-main}. Section~\ref{sec-exit} addresses Theorem~\ref{thm-exit}. Finally Section~\ref{sec-recover} discusses Theorem~\ref{thm-recover}. In the appendix, we recollect some ingredients of the proof by Le Gall for the Brownian map case as a comparison.

\paragraph*{Acknowledgements.}
We thank Julien Dub\'edat, Ewain Gwynne and Scott Sheffield for insightful discussions, and two anonymous referees for helpful comments on the first version of this paper. We thank also the organizers of the ``Probability and quantum field theory'' conference in Porquerolles, France; this is where the work was initiated. M.A. was partially supported by NSF grant DMS-1712862. 
The research of X.S.\ is supported by the Simons Foundation as a Junior Fellow at the Simons Society of Fellows, by NSF grant DMS-1811092, and by Minerva fund at the Department of Mathematics at Columbia University.

\section{Background and preliminaries}\label{sec-background}

\subsection{Notation}
For each $\gamma \in (0,2)$, we write $d_\gamma$ for the Hausdorff dimension of $\gamma$-LQG \cite{gwynne2019kpz} (this was originally introduced in the literature as the ``fractal dimension'' of $\gamma$-LQG, a scaling exponent associated with models expected to converge to $\gamma$-LQG; see \cite{DG18, ding2019heat, GHS16}). We also set the $\gamma$-dependent constants
\begin{equation}\label{eq:xi-Q}
Q = \frac\gamma2 + \frac2\gamma \quad \textrm{and}\quad \xi = \frac\gamma{d_\gamma}. 
\end{equation} 
  
We write $\N = \{ 1, 2, 3 \dots \}$ and $\N_0=\N\cup \{0\}$. For $x \in \R$, $\lfloor x \rfloor$ and $\lceil x \rceil$ denote the floor and ceiling functions evaluated at $x$. We write $|E|$ for the cardinality of a finite set $E$. If $f$ is a function from a set $X$ to $\R^n$ for some $n \geq 1$, we denote the supremum norm of $f$ by $\norme{f}_{X} := \sup_{x \in X} |f(x)|$.

In our arguments, it is natural to consider both Euclidean balls and metric balls. We use the notation $B_r(z)$ to denote the \emph{Euclidean} ball of radius $r$ centered at $z$, and $\mc B_r(z; D_h)$ to denote the \emph{metric} ball of radius $r$ centered at $z$ (i.e. the ball with respect to the metric $D_h$). We also distinguish the unit disk $\D := B_1(0)$. We denote by $\ol{X}$ the closure of a set $X$. For any $r > 0$ and $z \in \C$, let $A_r(z)$ stand for the annulus $B_{r}(z) \setminus \ol{B_{r/2}(z)}$. Furthermore, for $0 < s < r$, we set $A_{s,r}(z) := B_r(z) \setminus \ol{B_s(z)}$. 

The LQG metric $D_h$ is almost surely a length metric, i.e. $D_h(z,w)$ is the infimum of the $D_h$-lengths of continuous paths between $z,w$. For an open set $U \subset \C$, the internal metric $D^U_h$ on $U$ is given by the infimum of the $D_h$-lengths of continuous paths in $U$.

We write $\dashint_C f$ for the average of $f$ over the circle $C$. For a GFF $h$, we write $h_r(z)$ for the average of $h$ on the circle $\partial B_r(z)$. 

We write $X \sim \mc{N}(m,\sigma^2)$ to express that the random variable $X$ is distributed according to a Gaussian probability measure with mean $m$ and variance $\sigma^2$.

We say that an event $E_{\eps}$, depending on $\eps$, occurs with superpolynomially high probability if for every fixed $p >0$, for all $\eps$ small enough,  $\P [E_{\eps}] \geq 1 - \eps^{p}$.
We similarly define events which occur with superpolynomially high probability as a parameter tends to $\infty$.

\subsection{The whole-plane Gaussian free field}
We give here a brief introduction to the whole-plane GFF. For more details see \cite{ig4}.

Let $H$ be the Hilbert space closure of smooth compactly supported functions $f$ on $\C$, equipped with the Dirichlet inner product
\[(f,g)_\nabla = (2\pi)^{-1} \int_\C \nabla f(z) \cdot \nabla g(z) \ dz. \]
Let $\{f_n\}$ be any orthonormal basis of $H$, and consider the equivalence relation on the space of distributions given by $T_1 \sim T_2$ when $T_1 -  T_2$ is a constant. The \emph{whole-plane GFF modulo additive constant} $h$ is a random equivalence class of distributions, a representative of which is given by $\sum \alpha_n f_n$ where $\{\alpha_n\}$ is a sequence of i.i.d. $\mc N(0,1)$ random variables. The law of $h$ does not depend on the choice of $\{f_n\}$. 

For any complex affine transformation of the complex plane $A$, it is easy to verify that $(f \circ A, g\circ A)_\nabla = (f,g)_\nabla$. Consequently, $h$ has a law that is invariant under affine transformations: for each $r,z \in \C$ we have $h \stackrel{d}{=} h(r\cdot + z)$. 

Write $\wt H \subset H$ for the subspace of functions $f$ with $\int_{\C} f = 0$. Although we cannot define $\langle h, f \rangle$ for general $f \in H$, the distributional pairing makes sense for $f \in \wt H$ (the choice of additive constant does not matter). Explicitly, for $f \in \wt H$ the pairing $\langle h, f\rangle$ is a centered Gaussian with variance
\eqb\label{eq-zero-mean}
\Var (\langle h, f \rangle ) = \iint_{\C^2} f(w) f(z)  \log |w-z|^{-1} \ dwdz.
\eqe
It is easy to check that~\eqref{eq-zero-mean} in fact \emph{defines} the whole-plane GFF modulo additive constant. 

We will often fix the additive constant of $h$, i.e. choose an equivalence class representative. This can be done by specifying the value of $\langle h, f\rangle$ for some $f \in H$ with $\int_{\C} f \neq 0$, or the average of $h$ on a circle (see \cite[Section 3]{DS11} for details on the circle averages of $h$). Recalling that $h_r(z)$ means the circle average of $h$ on $\partial B_r(z)$, we will typically work with a whole-plane GFF $h$ normalized so $h_1(0) = 0$ (this is a distribution \emph{not} modulo additive constant).

Let $\cH_1 \subset H$ (resp. $\cH_2 \subset H$) be the Hilbert space completion of compactly supported functions which are constant (resp. have mean zero) on $\partial B_r(0)$ for all $r >0$. It is easy to verify the orthogonal decomposition $H = \cH_1 \oplus \cH_2$. This allows us to write the whole-plane GFF $h$ with $h_1(0) = 0$ as the sum of independent fields $h^1$ and $h^2$; these are respectively the projections of $h$ to $\cH_1$ and $\cH_2$. Moreover, we can explicitly describe the law of $h^1$: Writing $X_t = h_{e^{-t}}(0)$, the processes $(X_t)_{t \geq 0}$ and $(X_{-t})_{t\geq 0}$ are independent Brownian motions started at zero. The strong Markov property tells us that for any stopping time $T$ of $(X_t)_{t \geq 0}$, the random process $(X_{s +T} - X_T)_{s \geq 0}$ is independent from $X_T$. Also, by the scale invariance of the whole-plane GFF, the law of $h^2$ is scale invariant. These observations (with the independence of $h^1, h^2$) give us the following.
\begin{Lem}\label{lem-markovish}
Let $h$ be a whole-plane GFF with $h_1(0) = 0$, and let $T \geq 0$ be a stopping time of the circle average process $(h_{e^{-t}}(0))_{t \geq 0}$. Then we have, as fields on $\D$,  
\[h(e^{-T} \cdot)|_{\D} - h_{e^{-T}}(0) \stackrel{d}{=}h|_\D.\] 
Moreover, $h(e^{-T} \cdot)|_{\D} - h_{e^{-T}}(0)$ is independent of $h_{e^{-T}}(0)$. 
\end{Lem}

We note that there exist variants of the GFF on bounded domains $D \subset \C$, such as the zero boundary GFF and the Neumann GFF; we do not go into further detail, but remark that their LQG measures (Section~\ref{sec-GMC}) are well defined. 

Finally, we present a version of the Markov property for the whole-plane GFF, taken from \cite[Lemma 2.2]{harmonic-CRT}. It essentially follows from the orthogonal decomposition $H = \cH_\D \oplus \cH_\mathrm{harm}$ where $\cH_\D$ (resp. $\cH_\mathrm{harm}$) is the Hilbert space completion of functions which are compactly supported (resp. harmonic) in $\D$. 
\begin{Lem}[Markov property of GFF]\label{lem-markov}
Let $h$ be a whole-plane GFF normalized so $h_1(0) = 0$. For each open set $U \subset \C$ with harmonically non-trivial boundary and $U \cap \partial \D = \emptyset$, we have the decomposition 
\[h = \mathfrak h + \mathring h \]
where $\mathfrak h$ is a random distribution which is harmonic on $U$, and $\mathring h$ is independent from $\mathfrak h$ and has the law of a zero-boundary GFF on $U$ (in particular, $\mathring h|_{U^c} \equiv 0$). 
\end{Lem}

\subsection{Liouville quantum gravity and Gaussian multiplicative chaos }

\label{sec-GMC}

Fix $\gamma \in (0,2)$ and let $h$ be a GFF plus a random continuous function on a domain $D \subset \C$. We can define the \emph{$\gamma$-LQG volume measure} or \emph{quantum volume measure} via the almost sure limit in the vague topology 
\[\mu_h(dz) =  \lim_{\eps \to 0} \eps^{\gamma^2/2} e^{\gamma h_\eps(z)} \ dz\]
where the limit $\eps \to 0$ is taken along powers of two \cite{DS11}. (The limit was shown to hold in probability without the dyadic constraint \cite{Shamov16, Berestycki17}.)
Two properties are clear from the form of the above limit. Firstly, $\mu_h$ is locally determined by $h$, i.e. for any open set $U$, the volume $\mu_h(U)$ is a.s. determined by $h|_U$. Secondly, we have $\mu_{h+C}(\cdot) = e^{\gamma C} \mu_h(\cdot)$ for any $C \in \R$, and slightly more generally, for any random continuous function $f$ on a compact set $D$, we have almost surely $e^{\gamma \inf_D f} \mu_h(D) \leq \mu_{h+f} (D) \leq e^{\gamma \sup_D f} \mu_h(D)$. 

Liouville quantum gravity is a special case of Gaussian multiplicative chaos (GMC) introduced by Kahane \cite{Kahane85}, which considers more general $\log$-correlated fields. More precisely, if $D \subset \C$ is a bounded domain and $g$ a continuous function on $\ol{D} \times \ol{D}$ such that $K(x,y) = \log |x - y|^{-1} + g(x,y)$ is a nonnegative definite kernel, then one can consider the $\log$-correlated Gaussian field $\phi$ with covariance kernel $K$. Consider then the approximating measures $\mu_{\phi_{\eps}}(dz) = e^{\gamma \phi_{\eps}(z) - \frac{\gamma^2}{2} \Var \phi_{\eps}(z)  } \sigma(dz)$ where $\phi_{\eps}(x)$ denotes the circle average approximation of $\phi$ and $\sigma$ is a Radon measure on $\ol{D}$ of dimension at least two (in this paper we will always take $\sigma$ to be Lebesgue measure). Then, $\mu_{\phi_\eps}$ converges in probability towards a Borel measure $\mu$ on $D$ for the topology of weak convergence of measures on $D$ and the limit is the same for different approximation schemes e.g. when replacing the circle average approximation by an other mollification (\cite{Berestycki17, Shamov16}). The renormalizations for LQG and GMC  are different since we typically have $\eps^{\gamma^2/2} \neq e^{-\frac{\gamma^2}2 \Var \phi_\eps(z)}$. We will work with the LQG one when we use the GFF $h$ and the GMC one when we consider another $\log$-correlated Gaussian field $\phi$.

We refer the reader to \cite{Aru17, BerestyckiGffLqg, RV14} for excellent introductions to the domain.

\subsection{LQG volume of Euclidean balls} \label{sec-euc-vol}

Tails estimates for the LQG volume of Euclidean balls are quite well understood. It has been known since the work of Kahane \cite{Kahane85} and Molchan \cite{Mo96} that it admits finite moments for $p \in (-\infty, 4/\gamma^2)$. This result contrasts a very different behavior between the right tails and the left tails.

\paragraph*{Negative moments} The finiteness of all negative moments goes back to Molchan \cite{Mo96}; moreover it is more generally true that for any base measure of the GMC, the total GMC mass has negative moments of all order \cite{GHSS18}. Duplantier and Sheffield obtained the following more explicit tail behavior \cite[Lemma 4.5]{DS11}: writing $\mu_h$ for the LQG measure corresponding to a zero boundary GFF $h$ on $\D$, they showed that if $U \subset \subset \D$ is an open set, then there exists $C,c > 0$ such that for all $s > 0$,
\begin{equation}
\label{eq:LeftTailsEucli}
\Pro \left[\mu_h(U) \leq e^{-s} \right] \leq C e^{-c s^2}.
\end{equation}
We note that this result is sharp in the sense that 
\[
\Pro \left[ \mu_h(U) \leq e^{-s} \right] \geq c e^{-Cs^2}.
\]
by a simple application of the Cameron-Martin formula.   When $h$ is replaced by $h-\dashint_U hdz$, a sharper tail estimate is obtained in \cite{LRV-Mabuchi}.

\paragraph*{Positive moments}  Recently, Rhodes and Vargas  \cite{RV17} obtained a precise asymptotic result about the upper tails of GMC when $\gamma \in (0,2)$. They obtained a power law and identified the constant. This result has been generalized to a  more general family of Gaussian fields in \cite{W19a}, and extended to the critical case $\gamma = 2$ in \cite{W19b}. 

As already mentioned, the LQG volume of Euclidean balls has finite $p$ moments for $p < 4/\gamma^2$. 
This can be easily seen for integer moments $k < 4/\gamma^2$, which we review below. (This will also serve as a preparation to some of our arguments.) Indeed, due to the logarithmic correlations of the field, the problem is essentially equivalent to the finiteness of
\[
u_k := \int_{\D^k} \frac{d z_1,\dots dz_k}{\prod_{i < j} |z_i-z_j|^{\gamma^2}}.
\]
By introducing
\begin{equation}\label{eq:uv}
u_k(r) := \int_{r \D^k} \frac{d z_1,\dots dz_k}{\prod_{i < j} |z_i-z_j|^{\gamma^2}}   \quad \text{and} \quad v_k(r) = \int_{\D^k} \frac{1_{r/2 \leq \max_{i < j} |z_i-z_j| \leq r} }{\prod_{i < j} |z_i-z_j|^{\gamma^2}} dz_1 \dots dz_k,
\end{equation}
we note that when $u_k < \infty$ then $u_k(r) = r^{2k-\gamma^2 \frac{k(k-1)}{2}} u_k$. Furthermore, the $v_k$'s provide the following inductive inequality, obtained by splitting the points $\{z_1, \dots, z_k\}$ into two well-separated clusters (see Lemma \ref{Lem:Preliminaries} in the Appendix for details):
\[
v_k(r)  \leq C_k r^{-2} \sum_{i=1}^{k-1}  r^{-\gamma^2 i (k-i)}   u_i(4r) u_{k-i}(4r) \leq C_k r^{k \gamma Q - \frac{1}{2}\gamma^2 k^2 - 2} \sum_{i=1}^{k-1}  u_i u_{k-i}.
\]
Finally, we note that
\[
k \gamma Q - \frac{1}{2} \gamma^2 k^2 - 2 = k (2 + \frac{\gamma^2}{2}) - \frac{1}{2} \gamma^2 k^2 - 2 = 2(k-1) - \frac{\gamma^2}{2} k(k-1) > 0 \quad \text{if} \quad 1 < k < 4 /\gamma^2,
\]
and the conclusion follows from $u_k = \sum_{p \geq -1} v_k(2^{-p})$ and an induction on $k$.

Our later arguments in Section~\ref{Sec:InductiveStarScale} follow a similar structure to the above, but also have to account for the random geometry of the metric ball $\cB_1(0; D_h)$.

\subsection{LQG metric}\label{sec-metric}

Recently, a metric for LQG was constructed and characterized in \cite{DDDF19, GM19uniqueness}, relying on \cite{DG18, DFGPS, GM19local, GM19confluence, MQ18}. It is also the limit of an approximation scheme similar to the one of the LQG measure. 
For $\gamma \in (0,2)$, the $\gamma$-LQG metric is the unique metric $D_h$ determined by a field $h$ (a whole-plane GFF plus a possibly random bounded continuous function) which induces the Euclidean topology and satisfies the following.
\begin{enumerate}[I.]
\item \textbf{Length space.} $(\C, D_h)$ is almost surely a length space. That is, the $D_h$-distance between any two points in $\C$ is the infimum of the $D_h$-lengths of continuous paths between the two points.
\item \label{axiom-locality}\textbf{Locality.} Let $U \subset \C$ be a deterministic open set. Then the internal metric $D^U_{h}$ is almost surely determined by $h|_U$. 
\item \label{axiom-weyl}\textbf{Weyl scaling.} Recall $\xi$ in~\eqref{eq:xi-Q}. For each continuous function $f: \C \to \R$, define 
\eqb
(e^{\xi f} \cdot D_h)(z,w)  := \inf_{P:z\to w} \int_0^{\mathrm{len}(P;D_h)} e^{\xi f(P(t))} dt, \quad \text{ for all }z,w \in \C,
\eqe
where we take the infimum over all continuous paths from $z$ to $w$ parametrized by $D_h$-length. Then almost surely $e^{\xi f} \cdot D_h = D_{h+f}$ for every continuous $f: \C \to \R$. 
\item \textbf{Coordinate change for translation and scaling.}  Recall $Q$ in~\eqref{eq:xi-Q}. For fixed deterministic $z \in \C$ and $r>0$ we have almost surely 
\[D_h(ru+z, rv+z) = D_{h(r\cdot + z) + Q \log r} (u,v) \quad \text{ for all }u,v \in \C. \]
\end{enumerate}
To be precise, $D_h$ is unique up to a global multiplicative constant, which can be fixed in some way, e.g. requiring the median of $D_h(0,1)$ to be 1 for $h$ a whole-plane GFF normalized so $h_1(0) = 0$.
We emphasize that the metric $D_h$ depends on the parameter $\gamma \in (0,2)$; to follow previous works and avoid clutter we will omit $\gamma$ in the notation.

\paragraph*{Basic estimates for distances} 

The main quantitative input we need when working with the LQG metric is the following estimate relating the $D_h$-distance between compact sets to circle averages of $h$.
\begin{Prop}[Concentration of side-to-side crossing distance {\cite[Proposition 3.1]{DFGPS}}]\label{prop-dist}
Let $U \subset \C$ be an open set (possibly $U = \C$) and let $K_1, K_2 \subset U$ be disjoint connected compact sets which are not singletons. Then for $r > 0$, it holds with superpolynomially high probability as $A \to \infty$ (at a rate uniform in $r$) that 
\[A^{-1} r^{\xi Q} e^{\xi h_r(0)} \leq D_{h}^{rU}(rK_1, rK_2)\leq Ar^{\xi Q} e^{\xi h_r(0)}.  \]
\end{Prop}
This formulation is slightly different from that of \cite[Proposition 3.1]{DFGPS}, but by \cite[Remark 3.16]{DFGPS} they are equivalent. Note that by taking $r =1$, this includes the superpolynomial tails of side-to-side crossing distances.

\paragraph*{Euclidean balls within LQG balls}

The next lemma is an important input in the proof of the finiteness of the negative moments.

\begin{Prop}[LQG balls contain Euclidean balls of comparable diameter {\cite[Proposition 4.5]{GMS18}}]\label{prop-euc-ball}
Fix $\zeta \in (0,1)$ and compact $K \subset \C$. Let $h$ be a whole-plane GFF normalized so $h_1(0) = 0$. With superpolynomially high probability as $\delta \to 0$, each $D_h$-metric ball $B \subset K$ with $\diam(B)\leq \delta$ contains a Euclidean ball of radius at least $\diam(B)^{1+\zeta}$. 
\end{Prop}
\begin{proof}
\cite[Proposition 4.5]{GMS18} gives this result with $K$ replaced by $\D$ and with the specific choice $\gamma = \sqrt{8/3}$. To get the result for $K$, we simply note that the law of the whole-plane GFF (viewed modulo additive constant) is scale-invariant, and that the set of all $D_h$-metric balls (viewed as subsets of $\C$) does not depend on the choice of additive constant. To generalize to $\gamma \in (0,2)$, we remark that the proof of \cite[Proposition 4.5]{GMS18} uses only the following few inputs for the $\sqrt{8/3}$ LQG metric, which we ascertain hold for general $\gamma$:
\begin{itemize}
\item The scaling relation \cite[Lemma 2.3]{GMS18}. In our setting, this is Axiom~\hyperref[axiom-weyl]{III} (Weyl scaling), plus the following easy consequence of Weyl scaling: for $h$ a whole-plane GFF plus a bounded continuous function and $f: \C \to \R$ a (possibly random) bounded continuous function, almost surely
\[\exp\left( \xi \inf_\C f \right) D_h(z,w) \leq D_{h+f} (z,w) \leq \exp \left( \xi \sup_\C f \right)D_h(z,w) \quad \text{ for all } z,w \in \C. \]
\item With probability tending to 1 as $C \to \infty$, the $D_h$-distance from $S = [0,1]^2$ to $\partial B_{1/2}(S)$ is at least $1/C$ (here, $B_{1/2}(S)$ is the Euclidean $1/2$-neighborhood of $S$). This follows immediately from Proposition~\ref{prop-dist}.
\item Fix $n\geq 1$. With probability tending to 1 as $C \to \infty$, each Euclidean ball of radius $e^{-C n^{2/3}}$ which intersects $[0,1]^2$ has $D_h$-diameter at most $e^{-n^{2/3}}$. This follows from the fact that $D_h$ is a.s. bi-H\"older with respect to the Euclidean metric \cite[Theorem 1.7]{DFGPS}, and that $e^{-C n^{2/3}} \to 0$ as $C \to \infty$. 
\end{itemize}\end{proof}
We point out that this is possible to obtain a more quantative version of this Proposition, with essentially the same arguments as in \cite{GMS18}, which can then be used to obtain more precise lower tail estimates for the volume of LQG metric balls.

\section{Positive moments}\label{section-positive}

The main result of this section is the following.
\begin{Prop}
\label{Prop:MomentsWholePlane}
Let $h$ be a whole-plane GFF such that $h_1(0) = 0$. Then, $\mu_h(\mc{B}_1(0;D_h))$ has finite $k$th moments for all $ k \geq 1$.  Furthermore, this result still holds if we add to the field $h$ an $\alpha$-$\log$ singularity at the origin for $\alpha < Q$, i.e. replace $h$ with $h + \alpha \log |\cdot|^{-1}$.
\end{Prop}

In the following paragraphs, we present heuristic arguments and an outline of the proof. Recall the definition of the annulus $A_1 = B_1(0) \backslash \ol{B_{1/2}(0)}$. The key difficulty to prove this result is in arguing that $\E [ \mu_h(\cB_1(0; D_h) \cap A_1)^k] < \infty$. So we want to prove
\begin{equation}
\label{eq:target}
\E \left[\int_{(A_1)^k} \prod_{i=1}^k  1_{D_h(0, z_i) < 1} \mu_h(dz_1) \dots \mu_h(dz_k) \right] < \infty, 
\end{equation}
and the starting point is to rewrite it via a Cameron-Martin shift, as 
\eqb\label{eq-heuristic-moment}
{\LARGE \int}_{(A_1)^k} \exp\left( \gamma^2 \sum_{i<j} \Cov(h(z_i), h(z_j)) \right) \P \left[ D_{h + \gamma \sum_j \Cov(h(z_j), h(\cdot))} (0, z_i) <1 \text{ for all }i\right] dz_1 \dots dz_k < \infty.
\eqe

\paragraph*{A first heuristic}
We present a heuristic explaining why $\E \left[ \mu_h(\mc{B}_1(0;D_h) \cap A_1)^k \right] < \infty$. As remarked above and since $h$ is $\log$-correlated, the left-hand side of \eqref{eq:target} is bounded from above by
\begin{equation}
\label{eq:integral}
\int_{A_1^k} \frac{P_{z_1,\dots, z_k}}{\prod_{i< j} |z_i -z_j|^{\gamma^2}} dz_1 \dots dz_k
\end{equation}
where
\[
P_{z_1, \dots, z_k} =  \P [ D_{h + \gamma \sum_j \Cov(h(z_j), h(\cdot))} (z_i,\partial B_{1/2}(z_i)) <1 \text{ for all }i ]. \nonumber
\]
The volume of Euclidean balls have infinite $k$th moments when $k$ is large due to the contribution of clusters  at mutual distance $r$ (collection of points in the domain whose pairwise distance are between $c r$ and $C r$). Indeed, for such clusters $\{z_1, \dots, z_k\}$, the singularities contributes as $\prod_{i< j} |z_i -z_j|^{-\gamma^2} \approx r^{- {k \choose 2} \gamma^2}$, on a macroscopic domain, we have $r^{-2}$ possibilities for placing this cluster and the volume associated is $r^{2k}$. The total contribution is then  $r^{-2+2k - {k \choose 2} \gamma^2}$ and the sum over dyadic $r$ is finite if and only if $k <4/\gamma^2$. Now, we explain how this is counterbalanced by the $P_{z_1, \dots, z_k}$ term when $k \geq 4/\gamma^2$. By the annulus crossing distance bound from Proposition \ref{prop-dist}, for any $z \in K = \{z_1, \dots, z_k\}$, the following lower bound  holds
\[
 D_{h+ \gamma \sum_{i\leq k} \log |\cdot - z_i|^{-1}}(z,\partial B_{1/2}(z)) \gtrsim r^{\xi Q} e^{\xi h_r(z)} r^{-\xi k \gamma}.
\]
Indeed, one can use an annulus centered at $z$, separating $z$ from $\partial B_{1/2}(z)$ and at distance $r$ of $z$, whose width is of the same order. Then, we see that the circle average of the $\log$-singularity gives the $r^{-\xi k \gamma}$ term. So, by the condition defining $P_{z_1, \dots, z_k}$, on the associated event, for $z \in \{z_1, \dots, z_k\}$,
\[
1 \gtrsim r^{\xi Q} e^{\xi h_r(z)} r^{-\xi k \gamma}.
\]
By a Gaussian tail estimate, introducing the term $c_k = k \gamma -Q \geq \frac{4}{\gamma^2} \gamma - Q = 2/\gamma - \gamma/2 > 0 $, we have
\[
P_{z_1, \dots, z_k}\lesssim \P \left[ h_r(z) \leq - c_k \log r^{-1}  \right] \approx r^{\frac{1}{2}c_k^2}.
\]
An elementary computation, namely $-2+2k - {k \choose 2} \gamma^2 + \frac{1}{2} c_k^2 = \frac{1}{2}Q^2 -2$, gives then that for such a cluster, the scale $r$ contribution to \eqref{eq:integral} is $r^{\frac{1}{2}Q^2 -2}$, 
which is summable for all $k$ since $Q = \frac{\gamma}{2} + \frac{2}{\gamma} > 2$ for $\gamma \in (0,2)$ and this is essentially the reason of the finiteness of all moment.

\paragraph*{Outline of the proof}

To turn this argument into a proof requires us to take care of all configurations of clusters $K = \{z_1, \dots, z_k \}$. Similarly to the one presented in Section \ref{sec-euc-vol}, our proof works by induction on $k$.  We will partition $K = \{ z_1, \dots, z_k \}$ into two clusters $I$ and $J$ such that the pairwise distance of points between $I$ and $J$ is $\geq r$, since both $\prod_{i<j} |z_i-z_j|^{\gamma^2}$ and $P_{z_1, \dots, z_k}$ have a nice hierarchical clusters structure (see \eqref{Def:SepDis} for the exact splitting procedure partitioning $K=I \cup J$ and the definition of $r$). Indeed, for such a cluster, we can bound from above
\begin{equation}
\label{eq:HierDet}
\prod_{i< j} |z_i -z_j|^{-\gamma^2} \lesssim r^{- |I| |J| \gamma^2} \prod_{I} |z_a -z_b|^{-\gamma^2}  \prod_{J} |z_a -z_b|^{-\gamma^2}.
\end{equation}
Now, we discuss $P_{z_1, \dots, z_k}$. The aforementioned annuli crossing distance bounds from Proposition \ref{prop-dist} imply that for all $z \in K$, $\eps \in (0,1/2)$, 
\begin{equation}
\label{eq:CondCircle}
h_{\eps}(z) + \gamma \sum_{z_a \in K} \dashint_{\partial B_{\eps}(z)} \log | \cdot - z_a|^{-1} +x \leq Q  \log \eps^{-1},
\end{equation}
for $x =0$. From now, denote by $\widehat{P}_{z_1, \dots, z_k}^x$ the circle average variant of $P_{z_1, \dots,z_k}$ associated with \eqref{eq:CondCircle}:  this is the probability that \eqref{eq:CondCircle} holds for every $z \in K = \{z_1, \dots, z_k \}$ and $\eps \in (0,1/2)$, with this extra parameter $x \in \R$, which is necessary to consider when deriving an inductive inequality. Note that when $I$ and $J$ are at distance of order $r$ and the diameters of both $I$ and $J$ are smaller than $O(r)$, for $\eps \in (0,r)$, then $\forall z,z_a \in K$ and $\forall z_i \in I, z_j \in J$, 
\[
\dashint_{\partial B_{r}(z)} \log | \cdot - z_a |^{-1} \approx \log r^{-1} \quad \text{and} \quad  \dashint_{\partial B_{\eps}(z_i)} \log | \cdot - z_j|^{-1} \approx \log r^{-1}.
\]
Therefore, we can rewrite the condition \eqref{eq:CondCircle} for $z \in I$ as follows
\begin{multline*}
\left( h_{\eps}(z)-h_r(z) \right) + \left( \gamma \sum_{z_i \in I} \dashint_{\partial B_{\eps}(z)} \log | \cdot - z_i|^{-1} + |J|  \gamma \log r^{-1}\right) - k \gamma \log r^{-1} \\
 + \left(x + h_r(z) + k \gamma \log r^{-1}  - Q \log r^{-1} \right) \leq Q \log (\eps/r)^{-1}.
\end{multline*}
Hence, after simplification, for $z \in I$, we have
\[
\left( h_{\eps}(z)-h_r(z) \right) +  \gamma \sum_{z_i \in I} \dashint_{\partial B_{\eps}(z)} \log | \cdot/r - z_i/r|^{-1}  \\
 + \left(x + h_r(z) +  c_k \log r^{-1} \right) \leq Q \log (\eps/r)^{-1}
\]
which is a variant of \eqref{eq:CondCircle}, and a similar condition holds for $z \in J$. Furthermore, note that the processes $((h_{\eps}(z) - h_r(z))_{\eps \in (0,r)})_{z \in I}$ and $((h_{\eps}(z) - h_r(z))_{\eps \in (0,r)})_{z \in J}$ are approximately independent and $h_r(z) \approx h_r(w)$ for all $z,w \in K$, which we then denote by $X_r$ (this can thought as their common approximate value; to be rigorous, by monotonicity, one can take their maximum).  From this, and the fact that circle average processes evolve as correlated Brownian motions, it is natural to expect
\begin{equation}
\label{eq:HierRand}
\widehat{P}_K^0 \lesssim \E \left[1_{X_r + c_k \log r^{-1} \leq 0} \widehat{P}_{I/r}^{x + X_r + c_k \log r^{-1}} \widehat{P}_{J/r}^{x + X_r + c_k \log r^{-1}} \right],
\end{equation}
which is the hierarchical structure we were looking for. Altogether, \eqref{eq:HierDet} and \eqref{eq:HierRand} allow to inductively bound from above  the term
\[
\int_{A_1^k} \frac{P_{z_1,\dots, z_k}^x}{\prod_{i< j} |z_i -z_j|^{\gamma^2}} dz_1 \dots dz_k,
\]
by a quantitative estimate in term of $x$. This provides not only  $\E [ \mu_h (\mc{B}_1(0;D_h) \cap A_1)^k] < \infty$ but also a quantitative estimate which allows to get $\E [ \mu_h(\mc{B}_1(0;D_h) \cap A_s^k] < s^{\alpha_k}$ for some $\alpha_k >0$ and all $s \in (0,1)$, via a standard scaling/decoupling argument. An application of H\"older's inequality shows $\E [ \mu_h(\mc{B}_1(0;D_h) \cap \D)^k] < \infty$ and similar techniques concludes that $\E [ \mu_h(\mc{B}_1(0;D_h) \cap \C \setminus \D)^k] < \infty$, yielding the proof of Proposition~\ref{Prop:MomentsWholePlane}.

In our implementation of these ideas, because we have to carry the Euclidean domains associated with the clusters $I$, $J$ and $K$, we use $\star$-scale invariant fields. The short-range correlation of the fine field gives independence between well-separated clusters, and invariance properties of the $\star$-scale invariant field simplifies our multiscale analysis.

In Section~\ref{Sec:InductiveStarScale}, we prove a quantitative variant of~\eqref{eq-heuristic-moment} where the field $h$ is replaced by a $\star$-scale invariant field plus some constant, and the probability in the integrand is replaced by the probability of \emph{coarse-field distance approximations} being less than 1. In Section~\ref{sec-positive-whole-plane}, we use these estimates to first bound $\E [ \mu_h(\cB_1(0; D_h) \cap A_1)^k]$, by using a truncated moment estimate, then extend our arguments to all annuli to deduce the finiteness of the $k$th moment $M_k := \E [ \mu_h(\cB_1(0; D_h)^k]$ for all $k \geq 1$. By keeping track of the $k$ dependence, it turns out that it is possible to bound $M_k$ by $C k^{c k^2}$ for some constants $C,c$ depending only on $\gamma$. To simplify the presentation of our arguments, we omit these precise estimates. 

\subsection{Inductive estimate for the $\star$-scale invariant field}

\label{Sec:InductiveStarScale}

We derive a key estimate for the positive moments (Proposition~\ref{prop-circle-moment}),  which is like a quantitative version of~\eqref{eq-heuristic-moment} where we add a constant to the field.
We will use $\star$-scale invariant fields, which satisfy properties convenient for multiscale analysis. Relevant references are \cite{ARV13, DF18, JSW18}. 
\begin{Prop}[$\star$-scale decomposition of $h$] ~
\label{Prop:StarScaleProp}
The whole plane GFF $h$ normalized so $h_1(0) = 0$ can be written as 
\[
h = g + \phi = g + \phi_1 + \phi_2 + \dots
\]
where the fields $g, \phi_1, \phi_2, \dots$ satisfy the following properties:
\begin{enumerate} 
\item
$g$ and the $\phi_n$'s are continuous centered Gaussian fields.

\item
The law of $\phi_n$ is invariant under Euclidean isometries.

\item
$\phi_n$ has finite range dependence with range of dependence $e^{-n}$, i.e. the restrictions of $\phi_n$ to regions with pairwise distance at least $e^{-n}$ are mutually independent. 

\item
$\left( \phi_n(z) \right)_{z \in \R^2}$ has the law of $\left( \phi_1(z e^{n-1}) \right)_{z \in \R^2}$.

\item
The $\phi_n$'s are mutually independent fields.

\item
The covariance kernel of $\phi$ is $C_{0,\infty}(z,z') = - \log |z - z'| + q(z-z')$ for some smooth function $q$.

\item 
We have $\E[\phi_n(z)^2] = 1$ for all $n,z$.
\end{enumerate}
\end{Prop}
The convergence of this infinite sum is with respect to the weak topology on $\mc{S}'(\mb{R}^2)$.
\begin{proof}
Lemma \ref{Lem:CouplingFields} gives the coupling $h = g + \phi$ with $g$ continuous. The fields $\phi_n$ are defined in Appendix \ref{Sec:Whole-plane-Star-scale}, and are shown to satisfy these properties there. 
\end{proof}
Define also the field $\phi_{a,b}$ from scales $a$ to $b$ via
\begin{equation}\label{eq:phiab}
\phi_{a,b} := \left\{
\begin{array}{ll}
\phi_{a+1} + \cdots + \phi_b  & \mbox{if } a<b \\
0 & \mbox{if } a \geq b
\end{array}
\right. 
\end{equation}
so that $\phi = \phi_{0,\infty}$ and set, for $z, z' \in \C$,
\begin{equation}
\label{Def:kernel}
C_{a,b}(z,z') := \E \left( \phi_{a,b}(z) \phi_{a,b}(z') \right).
\end{equation}

We will construct a hierarchical representation of a set of points $K = \{z_1, \dots, z_k\}\subset \C$. Roughly speaking, starting with $K$, we will iteratively split each cluster into smaller clusters that are well separated. We formalize the splitting procedure below.

\paragraph*{Splitting procedure} Define for any finite set $S$ of points in the plane (with $|S| \geq 2$) the \emph{separation distance} $s(S)$ to be the largest $t \geq 0$ for which we can partition $S = I \cup J$ such that $d(I,J) \geq t$, i.e.
\begin{equation}
\label{Def:SepDis}
s(S) := \max_{S = I \cup J, |I|, |J| \geq 1} d(I,J).
\end{equation}
Define $I_S, J_S \subset S$ to be any partition of $S$ with $d(I, J) = s(S)$. Note that if $\diam S$ denote the diameter of the set $S$, we have the following inequality 
\eqb\label{eq-sep-diam}
\frac{\diam S}{|S|} \leq s(S) \leq \diam S.
\eqe 
For the edge case where $|S| = 1$ define $s(S) = 0$. 
\begin{Lem}\label{lem-separation}
For $|S| \geq 2$, we have $s(I_S), s(J_S) \leq s(S)$.
\end{Lem}
\begin{proof}
It suffices to prove the lemma for $S$ such that all pairwise distances in $S$ are distinct, then continuity yields the result for general $S$. Suppose for the sake of contradiction that $s(J) > s(S)$, then there is a partition $J = J_1 \cup J_2$ satisfying $d(J_1,J_2) > s(S)$. Since distances are pairwise distinct, we must have $d(I, J_i) = s(S)$ and $d(I, J_{3-i}) > s(S)$ for some $i$. Then $d(I \cup J_i, J_{3-i}) = \min( d(I, J_{3-i}) , d(J_i, J_{3-i})) > s(S)$. This contradicts the definition of $s(S)$.
\end{proof}

\paragraph*{Hierarchical structure of $K = \{z_1, \dots, z_k\}$ and definition of $T^{a}_K(\{ \phi \})$}

By iterating the splitting procedure above, we can decompose a set $K = \{z_1, \dots, z_k\} \subset \C$ into a binary tree of clusters. This decomposition into hierarchical clusters is unique for Lebesgue typical points $\{z_1, \dots, z_k\}$. Two vertices in this tree are separated by at least the separation distance of their first common ancestor. See  Figure~\ref{fig-hierarchy} for an illustration.

A \emph{labeled (binary) tree} is a rooted binary tree with $k$ leaves. For each $K = \{ z_1, \dots, z_k\} \subset \C$, collection of fields $\{\phi\} = (\phi_n)_{n \geq 0}$, and nonnegative integer $a \leq \lceil \log s(K)^{-1} \rceil$ we will define a labeled binary tree denoted by $T_K^a(\{\phi\})$.  Each internal vertex of this tree is labeled with a quadruple $(S,m,\psi,\eta)$ with $S \subset K$ and $|S| \geq 2$, an integer $m$, and $\psi,\eta \in \R$, whereas each leaf is labeled with just a singleton $\{z\} \subset K$. The truncated labels $(S,m)$ depend only on the recursive splitting procedure described above: $S$ is one of the clusters associated with this hierarchical cluster decomposition, and $m = \lceil \log s(S)^{-1} \rceil$. The variable $a$ represents an initial scale.

For such a labeled tree $T$ we write $T + (\psi_0, \eta_0)$ to be the tree obtained by replacing each internal vertex label $(S,m,\psi, \eta)$ with $(S,m,\psi + \psi_0, \eta + \eta_0)$. We also write $\mathrm{Left}(S)$ to denote the leftmost point of $S$, viz. $\arg \min_{z \in S} \Re(z)$, where $\Re(z)$ denotes the real part of the complex number $z$.

We explain how the remaining parts $(\psi, \eta)$ of the labels are obtained. For $(K,\{\phi\},a)$ as above, we proceed as follows to complete the definition of the  labeled tree $T^a_K(\{\phi\})$. For $k := | K | = 1$, we simply set $T^a_K(\{\phi\})$ to be the tree with one vertex, labeled with the singleton $K$.
For $k > 1$, setting $m := \lceil \log s(K) ^{-1} \rceil \geq a$, the root vertex of $T^a_K(\{\phi\})$ is labeled $(K,m,\phi_{a,m}(\mathrm{Left}(K)), (m-a) k\gamma)$, and its two child subtrees are given by $T^{m}_{I_K} (\{\phi\}) + (\phi_{a,m}(\mathrm{Left}(K)), (m - a)k\gamma)$ and $T^{m}_{J_K} (\{\phi\}) + (\phi_{a,m}(\mathrm{Left}(K)), (m - a)k\gamma)$.  Essentially, after making the split $K = I \cup J$, we add up the contribution of the coarse field $\phi_{a, m}$ and the contribution of the $\gamma$-log singularities to get the scale $m$ field approximation for the clusters $I$ and $J$.

We note that the tree structure of $T^a_K(\{\phi\})$ is deterministic, and for each internal vertex with label $(S,m,\psi,\eta)$, only $\psi = \psi (\{\phi\})$ is random; the other components are deterministic. Roughly speaking, $S$ is a cluster in our hierarchical decomposition, $m$ is the scale of the cluster (i.e. $s(S) \approx e^{-m}$), $\psi$ (resp. $\eta$) approximates a radius $e^{-m}$ circle average of the field $\phi_{a, m}$ (resp. $\gamma \sum_{z \in K} \log |z - \cdot|^{-1} - \gamma k a$) at the cluster.

\begin{figure}[ht!]
\begin{center}
  \includegraphics[scale=0.50]{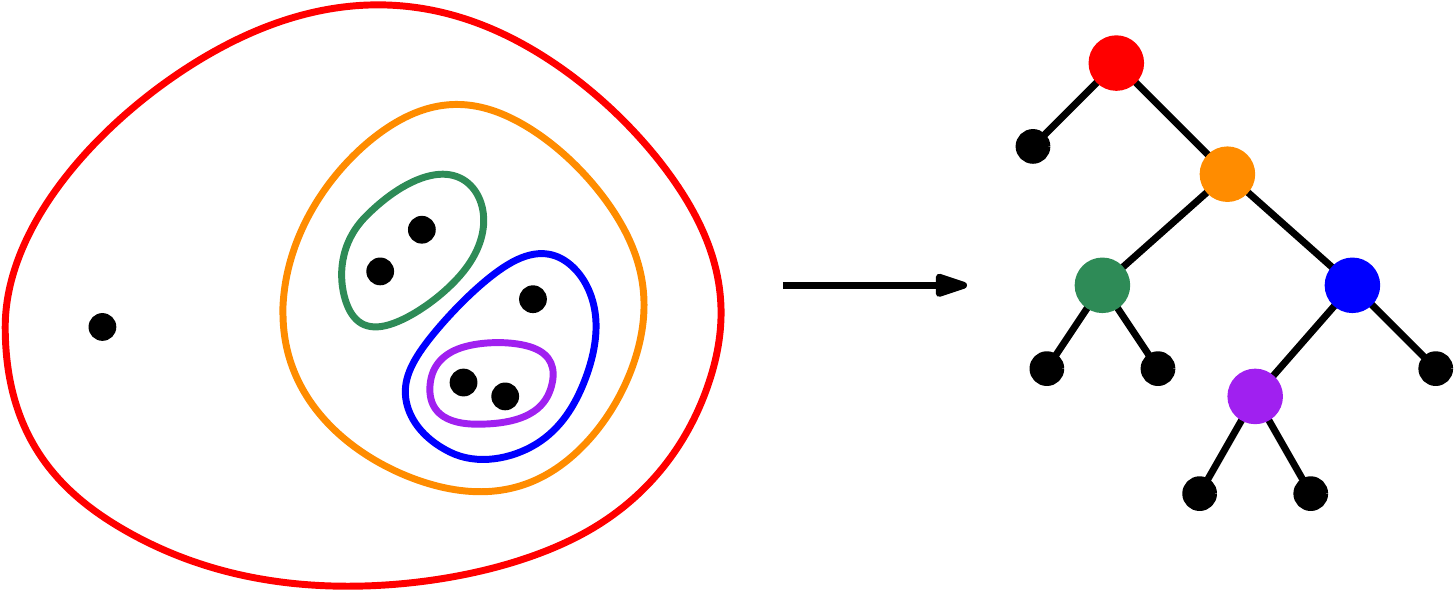}%
 \end{center}
\caption{\label{fig-hierarchy} \textbf{Left:} The set of points $K$ is iteratively divided into smaller and smaller clusters. \textbf{Right:} From this clustering algorithm we obtain a hierarchical binary tree $T^a_K(\{\phi\})$ (labels not shown), where internal vertices correspond to clusters $S \subset K$ and leaves correspond to points $z \in K$.}
\end{figure}

\begin{Rem}\label{rem-nonrecursive}
For the labeled tree $T^{a}_K(\{\phi\})$, at each internal vertex the field approximation $\psi$ can be explicitly described in terms of the fields $\{\phi\}$ as follows. Let $(S_i, m_i,\psi_i,\eta_i)$ for $i= 1,\dots, n$ be a path from the root $(S_1, m_1, \psi_1, \eta_1)$. Then, writing $m_0 = a$, we have 
\begin{equation}
\label{eq:Def-psi}
\psi_n = \sum_{i=1}^{n} \phi_{m_{i-1},m_i}(\mathrm{Left}(S_i)).
\end{equation} 
The $\gamma$-singularity approximation $\eta$ can likewise be stated non-recursively, as 
\begin{equation}
\label{eq:Def-eta}
\eta_n = \gamma \sum_{i=1}^{n} (m_{i} - m_{i-1}) |S_i|.
\end{equation}
\end{Rem}

\begin{Rem}
The choice $\mathrm{Left}(S_i)$ is arbitrary; any other deterministic choice of point in $S_i$ works. Replacing $\phi_{m_{i-1},m_i}(\mathrm{Left}(S_i))$ with the average $|S_i|^{-1} \sum_{z \in S_i} \phi_{m_{i-1},m_i}(z)$ would also work without affecting our proofs much.
\end{Rem}

\paragraph*{Definitions of key observables} In this paragraph, we provide analogous definitions of the quantities appearing in \eqref{eq-heuristic-moment}. The first one corresponds to a variant of $\P[ D_{h + \gamma \sum_j \Cov(h(z_j), h(\cdot))}$ $(0, z_i) <1 \text{ for all }i ]$, with an extra parameter $x$. For $x \in \R$, let $P^{a,x}_K$ be the probability that the tree with random labels $T^a_K(\{\phi\})$ satisfies 
\eqb\label{eq-cluster-ineq}
\psi + \eta + x \leq Q(m-a) \quad \text{ for each internal vertex labeled } (S, m, \psi,\eta).
\eqe
Note that this probability is taken over the randomness of the fields $\{\phi\}$, and that this definition yields for $|K| = 1$ that $P^{a,x}_K = 1$. Let us comment a bit on this definition and its relation with the conditions $D_{h + \gamma \sum_j \Cov(h(z_j), h(\cdot))} (0, z_i) <1$. These distances being less than one implies upper bounds for annuli crossing distances for annuli separating the origin from the singularities. The $\psi$ term corresponds to field average over these annuli, $\eta$ is an approximation for the $\gamma$-singularities and the $Q$ term stands for the scaling of the metric. Altogether, roughly speaking, $P^{0,x}_K$ is the probability that for the field $\phi_{0,\infty} + \sum_{z \in K} \gamma \log|z - \cdot|^{-1} + x$, for all clusters $S$ of $K$ the field-average approximation of annulus-crossing distances near $S$ is less than 1. 

The following observable stands for a variant of the integral in \eqref{eq-heuristic-moment}. Writing $K = \{z_1, \dots, z_k\}$ and $dz_K = dz_1 \dots dz_k$, we define
\begin{equation}
\label{eq:EstimateStarScale}
u^n_k(x) := \int_{B_n(0)^k} \frac{P_K^{0, x}}{\prod_{i < j} |z_i-z_j|^{\gamma^2}} 1_{s(K) \leq e} dz_K.
\end{equation}
In Proposition~\ref{prop-circle-moment}, we show that $u^n_k(x) < \infty$, and bound it in terms of $x$. Note that the statement $u^n_k(x)<\infty$
is comparable to~\eqref{eq-heuristic-moment} by the fact that $\exp(\gamma \Cov(h(z_i), h(z_j))) \asymp |z_i - z_j|^{-\gamma^2}$.

\smallskip

The next lemma establishes basic properties of $P_K^{a,x}$. 
To state it, we first define 
\begin{equation}
c_k := k \gamma -Q.
\end{equation}

\begin{Lem}\label{lem:P}
The $P_K^{a,x}$'s satisfy the following properties:
\begin{enumerate}
\item
\label{item:cdt1}
Monotonicity: $P_K^{0,x}$ is decreasing in $x$.

\item
\label{item:cdt2}
Markov decomposition: for the partition $I_K \cup J_K = K$ with separation distance satisfying $e^{-m} \leq s(K) < e^{-m+1}$  we have
\[
P_{K}^{0, x} = \E \left[ 1_{X_r + x + c_k \log r^{-1} \leq 0} P_{I_K}^{\log r^{-1}, X_r + x + c_k \log r^{-1} } P_{J_K}^{\log r^{-1}, X_r + x + c_k \log r^{-1}}\right],
\] 
where $r=e^{-m}$ and $X_r= \phi_{0,m}\left(\mathrm{Left}(K) \right)$ is a centered Gaussian with variance $\log r^{-1}$.

\item
\label{item:cdt3}
Scaling: $P_{r z_1,\dots rz_k}^{\log r^{-1},x} =  P_{z_1,\dots z_k}^{0,x}$ for any $r = e^{-m}$ with $m\in \Z$. 

\item
\label{item:cdt4}
Invariance by translation: $P_{z_1 + w, \dots, z_k + w}^{0,x} = P_{z_1, \dots, z_k}^{0,x}$.
\end{enumerate}
\end{Lem}

The first condition corresponds to a shift of the field. The second condition is an identity with three terms in the right-hand side: the term $X_r$ represents the coarse field, the indicator says that the ``coarse field approximation of quantum distances'' at Euclidean scale $r$ are less than 1, and the product of the two other terms represent a Markovian decomposition conditional on the coarse field. Properties \ref{item:cdt3} and \ref{item:cdt4} are clear from the translation invariance and scaling properties of $\phi_n$.

\begin{proof}
The monotonicity Property \ref{item:cdt1} is clear from the definition. 

Property~\ref{item:cdt2} follows from the inductive definition of $P^{0,x}_K$, by looking at the first split $K = I \cup J$. Indeed, recall $X_r=\phi_{0,m}\left(\mathrm{Left}(K) \right)$. The event $\{ X_r + x + c_k \log r^{-1} \leq 0\}$ corresponds to inequality~\eqref{eq-cluster-ineq} for the root vertex $(K, m, \phi_{0,m}(\mathrm{Left}(K)), m k \gamma)$. 

Then, if the set $K$ is decomposed as $K = I \cup J$, note that the trees $T_I^m(\{ \phi \})$ and $T_J^m(\{ \phi \})$ are independent. Indeed, $d(I,J) \geq e^{-m}$, so the restrictions of the field $\phi_{m}$ (and each finer field) to $I$ and $J$ are independent. Therefore, since $(\phi_{0,m}(\mathrm{Left}(K)), T_I^m(\{ \phi \}), T_J^m(\{ \phi \}))$ are independent, conditionally on $\phi_{0,m}(\mathrm{Left}(K))$, the trees $T^{m}_{I} (\{\phi\}) + (\phi_{0,m}(\mathrm{Left}(K)), m k\gamma)$ and $T^{m}_{J} (\{\phi\}) + (\phi_{0,m}(\mathrm{Left}(K)), mk\gamma)$ are independent. Thus, all conditions in the definition of $P_K^{0,x}$ associated to the child subtrees are conditionally independent. To conclude, we just have to explain that this is indeed the term $P_I^{m,X_r+x+c_k m}$ which appears. For a non-root vertex $(S,b, \psi, \eta)$ of $T^{0,x}_K$ belonging to the genealogy of $I$, the condition \eqref{eq-cluster-ineq} can be rewritten,
\[\psi + \eta + x = (X_r+\psi') + (m k \gamma + \eta') + x \leq Q b = Q(b-m) + Qm,\]
hence $\psi' + \eta' + (X_r + x + c_k m) \leq Q(a-m)$, which is exactly the condition we were looking for at the vertex $(S,b,\psi',\eta')$ in the tree $T^{m}_I(\{\phi \})$.

The scaling Property \ref{item:cdt3} follows from the scaling property of the $\phi_m$ and the observation that $s( r K) = rs(K)$ (and hence $\lceil \log s( r K)^{-1} \rceil = \log r^{-1} + \lceil \log s(K)^{-1} \rceil)$.

The invariance by translation Property \ref{item:cdt4} follows from the translation invariance of the fields $\phi_m$.
\end{proof}

Using these properties, we derive the following inductive inequality.
\begin{Lem}
\label{Lem:InductiveRelation}
For each $n,k > 0$, there exists a constant $C_{n,k}$ such that the following inductive inequality holds, for all $x \in \R$, where $X_r \sim \mc{N}(0, \log r^{-1})$.
\begin{multline*}
u^n_k(x) \leq C_{n,k} \sum_{i=1}^{k-1} \sum_{r = e^{-m}, m \geq 0} r^{k \gamma Q - \frac{1}{2} \gamma^2 k^2 -2} \times \\
 \E \left[ 1_{X_r + x + c_k \log r^{-1} \leq 0} u_i^{6k}(X_r  + x + c_k \log r^{-1}) u_{k-i}^{6k}(X_r  + x + c_k \log r^{-1}) \right].
\end{multline*}
\end{Lem}

We now turn to the proof of the inductive relation. The argument is close to that of Lemma~\ref{Lem:Preliminaries}, the difference being that we have to take care of the decoupling of $P_{K}^{0, x} $.

\begin{proof}
We first introduce some notation. In what follows we will be integrating over $k$-tuples of points $z_1, \dots, z_k$; write $K$ for this collection of points and $dz_K = dz_1\dots dz_k$. 
Write $f(K) := \prod |z - z'|^{-\gamma^2/2}$ where the product is taken over all pairs $z, z' \in K$ with $z \neq z'$. 

We first split the integral in the definition \eqref{eq:EstimateStarScale} of $u_k^n(x)$ as 
\[
u_k^n(x) = \sum_{r=e^{-m}, m \geq 0}  v_k^n(x,r)
\]
where for $r \in (0,1]$, $v_k^n(x,r)$ is defined by 
\begin{equation}
\label{Def:vkxr}
v_k^n(x,r) :=  \int_{B_n(0)^k}P_K^{0,x} f(K) 1_{r \leq s(K)\leq er}dz_K .
\end{equation}
Notice that $s(K) \leq er$ implies $\diam K \leq e k r$, so any $K$ contributing to the integral in~\eqref{Def:vkxr} is contained in a ball of radius $6k r$ centered in $r \Z^2 \cap B(0,n)$. Taking a sum over the $O(n^2 r^{-2})$ such balls and by translation invariance, we get the bound
\[
v_k^n(x,r) \leq O( n^2r^{-2}) \int_{B_{6kr}(0)^k} P_K^{0,x} f(K) 1_{r \leq s(K)\leq er} dz_K . 
\]
Write $K = I_K \cup J_K$ for the partition described before Lemma~\ref{lem-separation}. For $z \in I_K$ and $z' \in J_K$ we have $|z - z'|^{-\gamma^2} \leq s(K)^{-\gamma^2}\leq r^{-\gamma^2}$, and $s(I_K), s(J_K) \leq s(K) \leq er$ by Lemma~\ref{lem-separation}, so 
\begin{align*}
v_k^n(x,r) &\leq O(n^2r^{-2}) \int_{B_{6kr}(0)^{6k}} r^{-\gamma^2 |I_K||J_K|}  P^{0,x}_K f(I_K) 1_{s(I_K)\leq er} f(J_K) 1_{s(J_K)\leq er}\ dz_K.
\end{align*}
The Markov property decomposition~\ref{item:cdt2}  Lemma~\ref{lem:P} allows us to split $P^{0,x}_K$ into an expectation over a product of terms, yielding an upper bound of $v_k^n(x,r)$ as an integral of terms which `split' into $z_{I_K}$ and $z_{J_K}$ parts. This expression is in terms of the partition $I_K\cup J_K = K$; we can upper bound it by summing over \emph{all} $I, J \subset K$. To be precise, for each $i = 1, \dots, k-1$ we sum over all pairs $I, J \subset K$ with $|I| = i$ and $|J| = k-i$. Absorbing combinatorial terms like $\binom k i$ and the prefactor $n^2$ into the constant $C_{n,k}$, we get
\begin{multline*}
v_k^n(x,r) \leq C_{n,k} r^{-2} \sum_{i=1}^{k-1} r^{-\gamma^2 i (k-i)} \E_{X_r} \left[ \int_{B_{6k r}(0)^i} \frac{P_{z_1,\dots,z_i}^{\log r^{-1},X_r + x + c_k \log r^{-1}}}{\prod_{a<b} |z_a-z_b|^{\gamma^2}}  1_{s(z_1, \dots, z_i)\leq er}  d z_1 \dots dz_i \right. \\
\left.\left( \int_{B_{6k r}(0)^{k-i}} \frac{ P_{w_1,\dots, w_{k-i}}^{\log r^{-1},X_r + x + c_k \log r^{-1}}}{\prod_{a<b}|w_a-w_b|^{\gamma^2}} 1_{s(w_1,\dots, w_{k-i})\leq er} dw_1\dots dw_{k-i} \right) 1_{X_r + x + c_k \log r^{-1} \leq 0} \right].
\end{multline*}
We analyze the first integral (we can deal with the second one along the same lines). Changing the domain of integration from $B_{6kr}(0)^i$ to $B_{6k}(0)^i$, we get
\begin{multline*}
 \int_{B_{6kr}(0)^i} \frac{P_{z_1,\dots, z_i}^{\log r^{-1},X_r + x + c_k \log r^{-1}}}{\prod_{a<b} |z_a-z_b|^{\gamma^2}} 1_{s(z_1, \dots, z_i)\leq er}  d z_1\dots dz_i \\
  = r^{2i - \gamma^2 {i \choose 2}}  \int_{B_{6k}(0)^i} \frac{P_{r z_1,\dots, r z_i}^{\log r^{-1},X_r + x + c_k \log r^{-1}}}{\prod_{a<b} |z_a-z_b|^{\gamma^2}} 1_{s(z_1, \dots, z_i)\leq e}  d z_1\dots dz_i,
\end{multline*}
and then applying the scaling property \ref{item:cdt3} of $P$, the integral on the right hand side is equal to 
\[
\int_{B_{6k}(0)^i}  \frac{P_{z_1,\dots, z_i}^{ 0,X_r + x +  c_k \log r^{-1}}}{\prod_{a<b} |z_a-z_b|^{\gamma^2}} 1_{s(z_1, \dots, z_i)\leq e}  d z_1\dots dz_i =  u_i^{6k}(X_r + x + c_k \log r^{-1}).
\]
By gathering the previous bounds and identities, and noting that the power of $r$ is 
\[
r^{-2 -\gamma^2i(k-i) +2k -  \gamma^2\binom{i}2 - \gamma^2\binom{k-i}2} =  r^{\gamma k Q - \frac{1}{2}\gamma^2 k^2-2},
\]
and this completes the proof of the inductive inequality.
\end{proof}

Using the inductive relation and the base case, we derive the following proposition, which provides a bound on the quantity \eqref{eq:EstimateStarScale}  introduced at the beginning of the section. 

\begin{Prop}
\label{prop-circle-moment}
Recall that $c_k = k\gamma - Q$. 
For $x \in \R$ we have
\[u^n_k(x) \leq C_{n,k} e^{- c_k x} \quad \text{ when }k \geq 4/\gamma^2, \]
and 
\[u^n_k(x) \leq C_{n,k} \quad \text{ when }k < 4/\gamma^2, \]
where $C_{n,k}$ is a constant depending only on $n,k$.
\end{Prop}
\begin{proof}
We first address the case where $k < 4/\gamma^2$. In this setting, by the trivial bound $P^{0,x}_K \leq 1$ we have 
\[u^n_k(x) \leq \int_{B_n(0)^k} \prod_{i<j} |z_i - z_j|^{-\gamma^2} \ dz_1 \dots dz_k, \]
and the right-hand side is finite by the discussion in Section~\ref{sec-euc-vol}.

Now consider $k \geq 4/\gamma^2$. We proceed inductively, assuming that the statement of the proposition has been shown for all $k'<k$. Lemma~\ref{Lem:InductiveRelation} gives us the bound
\begin{multline}
\label{eq-inductive-bound}
u^n_k(x) \leq C_{n,k} \sum_{i=1}^{k-1} \sum_{r = e^{-m}, m \geq 0} r^{k \gamma Q - \frac{1}{2} \gamma^2 k^2 -2} \times \\
\E \left[ 1_{X_r + x + c_k \log r^{-1} \leq 0} u_i^{6k}(X_r  + x + c_k \log r^{-1}) u_{k-i}^{6k}(X_r  + x + c_k \log r^{-1}) \right],
\end{multline}
where $X_r \sim \mathcal{N}(0,\log r^{-1})$. We bound each term $u^{6k}_i u^{6k}_{k-i}$ using the inductive hypothesis. We need to split into cases based on which bound of the statement of the proposition is applicable (i.e. based on the sizes of $i, k-i$), but the different cases are almost identical, so we present the first case in detail and simply record the computation for the remaining cases.
\medskip

\noindent\textbf{Case 1: $i, k-i \geq 4/\gamma^2$.}
By the inductive hypothesis we can bound the $i$th term of~\eqref{eq-inductive-bound} by a constant times
\begin{align}
&\sum_{r = e^{-m}, m \geq 0} r^{k \gamma Q - \frac{1}{2} \gamma^2 k^2 -2} \E \left[  e^{-(c_i + c_{k-i})(X_r  + x + c_k \log r^{-1})} 1_{X_r + x + c_k \log r^{-1} \leq 0} \right] \nonumber \\
=&\sum_{r = e^{-m}, m \geq 0} r^{k \gamma Q - \frac{1}{2} \gamma^2 k^2 -2 + c_k (c_k - Q)} e^{(Q- c_k)x} \E \left[  e^{-(c_k - Q)X_r} 1_{X_r + x + c_k \log r^{-1} \leq 0} \right], \label{eq-i}
\end{align}
where we have used the identity $c_i + c_{k-i} = c_k - Q$. For each $r$ we can write the expectation in the equation~\eqref{eq-i} by a Cameron-Martin shift as
\begin{multline}
\label{eq-shift}
 \E[e^{-(c_k - Q)X_r}] \P[X_r + x + c_k \log r^{-1} - (c_k - Q) \Var(X_r) \leq 0]  \\
= r^{-\frac12(c_k - Q)^2}  \P[X_r \leq -(Q\log r^{-1} + x)].
\end{multline}
We claim that
\eqb
\label{eq-gaussian-bound}
\P[X_r \leq -(Q\log r^{-1} + x)] \leq r^{\frac12Q^2} e^{-Qx}.
\eqe
Indeed, in the case where $Q\log r^{-1} + x \geq 0$, we have by a standard Gaussian tail bound that 
\[\P[X_r \leq -(Q\log r^{-1} + x)] \leq e^{-\frac{(Q\log r^{-1} +x)^2}{2 \log r^{-1}}}= r^{\frac12Q^2} e^{-Qx} e^{-\frac{x^2}{2\log r^{-1}}} \leq  r^{\frac12Q^2} e^{-Qx}  , \]
and in the cases where $Q\log r^{-1} + x < 0$ we have 
\[\P[X_r \leq -(Q\log r^{-1} + x)] \leq 1 \leq e^{-Q(Q\log r^{-1} + x)} = r^{Q^2} e^{-Qx} \leq r^{\frac12 Q^2} e^{-Qx}.\]
Finally, we combine~\eqref{eq-i},~\eqref{eq-shift} and~\eqref{eq-gaussian-bound} to upper bound the $i$th term of~\eqref{eq-inductive-bound}. This upper bound is a sum over $r$ of terms of the form $r^{\text{power}}e^{-c_kx}$ where the power is
\[
k \gamma Q - \frac{1}{2} \gamma^2 k^2 -2 + c_k (c_k - Q) -\frac12(c_k-Q)^2+\frac12Q^2=  \frac12 Q^2 - 2>0.
\]
So we can bound the $i$th term of~\eqref{eq-inductive-bound} by a constant times
\[ \sum_{r = e^{-m}, m\geq0} r^{\frac{Q^2}2 - 2} e^{-c_k x} = O(e^{-c_kx}).\] 
\medskip

\noindent \textbf{Case 2: $i \geq 4/\gamma^2$ and $k - i < 4/\gamma^2$.}  By the inductive hypothesis we can bound the $i$th term of~\eqref{eq-inductive-bound} by a constant times
\begin{align*}
&\sum_{r = e^{-m}, m \geq 0} r^{k \gamma Q - \frac{1}{2} \gamma^2 k^2 -2} \E \left[  e^{-c_i(X_r  + x + c_k \log r^{-1})} 1_{X_r + x + c_k \log r^{-1} \leq 0} \right]  \\
=&\sum_{r = e^{-m}, m \geq 0} r^{k \gamma Q - \frac{1}{2} \gamma^2 k^2 -2 + c_ic_k} e^{- c_ix} \E \left[  e^{-c_iX_r} 1_{X_r + x + c_k \log r^{-1} \leq 0} \right] \\
=&\sum_{r = e^{-m}, m \geq 0} r^{k \gamma Q - \frac{1}{2} \gamma^2 k^2 -2 + c_ic_k - \frac12 c_i^2} e^{- c_ix} \P \left[  X_r \leq -((c_k - c_i)\log r^{-1} + x) \right] \\
\leq& \sum_{r = e^{-m}, m \geq 0} r^{k \gamma Q - \frac{1}{2} \gamma^2 k^2 -2 + c_ic_k - \frac12 c_i^2 + \frac12 (c_k - c_i)^2} e^{- c_kx} \\
=& \sum_{r = e^{-m}, m \geq 0} r^{\frac12Q^2 - 2} e^{- c_kx}  = O(e^{-c_k x}).
\end{align*}
Note that by symmetry Case 2 also settles the case where $i < 4/\gamma^2$ and $k-i \geq 4/\gamma^2$.
\medskip

\noindent\textbf{Case 3: $i, k-i < 4/\gamma^2$.} By the inductive hypothesis we can bound the $i$th term of~\eqref{eq-inductive-bound} by a constant times
\begin{align*}
\sum_{r = e^{-m}, m \geq 0} r^{k \gamma Q - \frac{1}{2} \gamma^2 k^2 -2} \P \left[  X_r \leq -(c_k \log r^{-1} + x) \right] 
\leq& \sum_{r = e^{-m}, m \geq 0} r^{k \gamma Q - \frac{1}{2} \gamma^2 k^2 -2 + \frac12 c_k^2} e^{- c_kx} \\
=& \sum_{r = e^{-m}, m \geq 0} r^{\frac12Q^2 - 2} e^{- c_kx}  = O(e^{-c_k x}).
\end{align*}
This completes the proof.
\end{proof}

The proof of Proposition~\ref{prop-circle-moment} depends on the exponent $\frac12 Q^2 - 2 = \frac12(\frac2\gamma - \frac\gamma2)^2$ being positive. If we make a slight perturbation to our definitions, so long as the resulting exponent is still positive, we get a variant of Proposition~\ref{prop-circle-moment}. In particular, for $\delta > 0$, we define $P^{a,x,\delta}_K$ similarly to $P^{a,x}_K$ by replacing the inequality~\eqref{eq-cluster-ineq}  with $\psi + \eta + x \leq (Q+\delta)(m-a)$, and define $u^{n,\delta}_k$ analogously to \eqref{eq:EstimateStarScale} with $P^{0,x,\delta}_K.$ We record the following result as a corollary since the proof follows the same steps as in the proof of  Proposition~\ref{prop-circle-moment}.

\begin{Cor}
\label{Cor-Variant}
For $k \geq 1$ and $n \geq 1$, for $\delta$ small enough, there exist constants $C_{n,k,\delta}$ and $c_{k, \delta}$ such that,
\[u^{n,\delta}_k(x) \leq C_{n,k,\delta} e^{- c_{k,\delta} x} \quad \textrm{for all }x\in\R \text{ when }k \geq 4/\gamma^2, \]
and 
\[u^{n,\delta}_k(x) \leq C_{n,k,\delta} \quad \quad \textrm{for all }x\in\R\text{ when }k < 4/\gamma^2. \]
Furthermore, $\lim_{\delta\to 0}c_{k, \delta} = k\gamma - Q$  for fixed $k$.
\end{Cor}

 \begin{Rem}\label{rem-delta}  Alternatively, one could modify the definition of $u^n_k(x)$ in~\eqref{eq:EstimateStarScale} to have a different denominator $|z_i - z_j|^{\gamma^2 + \delta}$. Namely, by setting
\[
\hat{u}^{n,\delta}_k(x) := \int_{B_n(0)^k} \frac{P_K^{0, x}}{\prod_{i < j} |z_i-z_j|^{\gamma^{2}+\delta}} 1_{s(K) \leq e} dz_K,
\]
the statement of Corollary~\ref{Cor-Variant} applies to $\hat{u}^{n,\delta}_k(x)$ instead of $u^{n,\delta}_k(x)$.
\end{Rem}

\subsection{Moment bounds for the whole-plane GFF} \label{sec-positive-whole-plane}

In this section, we use our previous estimate (Proposition~\ref{prop-circle-moment} or its variant Corollary~\ref{Cor-Variant}) to obtain the moment bounds for a whole-plane GFF $h$ such normalized such that $h_1(0)=0$ and therefore prove Proposition \ref{Prop:MomentsWholePlane}.  Additionally, in this section we write $C$ or $C_{k,\delta}$ to represent large constants depending only on $k$ and $\delta$, and may not necessarily represent the same constant in different contexts or equations.

\subsubsection*{Proxy estimate for whole-plane GFF}

Recall the notation $A_{s,r} := B_r(0) \setminus \ol{B_{s}(0)}$ for $0 < s < r$. We introduce the following proxy
\begin{equation}
\label{eq:DefProxy}
P_h^{r,d} := \{ z \in \C:  D_h(z, \partial B_{r/4}(z)) \leq d \}.
\end{equation}
The set $P_h^{r,d}$ contains points whose ``local distances'' are small. We work with $P_h^{r,d}$ because the event $z \in P_h^{r,d}$ depends only on the field $h|_{B_{r/4}(z)}$, and is thus more tractable than the event $z \in \cB_1(0; D_h)$ (which depends on the field in a more ``global'' way). Moreover we have $\cB_1(0;D_h) \cap A_r \subset P_h^{r,1} \cap A_r$, so to bound from above $\mu_h(\cB_1(0; D_h))$ it suffices to bound from above the  volume of the proxy set. We emphasize that $P_h^{r,d}$ is different from the quantity $P_K^{a,x}$ introduced in \eqref{eq-cluster-ineq}: the former is associated with a field $h$ and is considered on the full plane without restriction; the latter is associated with $\star$-scale invariant fields, and the capital letter $K$ refers to a finite number of points where the condition is localized.

\begin{Prop}
\label{Prop:EstimateGFFbounded}
Let $h$ be a whole-plane GFF such that $h_1(0) =0$. For $k \geq 4/\gamma^2$, $\delta \in (0,1/2)$, there exists a constant $C_{k,\delta}$ such that for all $x \in \R$,
\[
\E \left[\mu_h \left(B_{10}(0) \cap P_h^{1,e^{-\xi x}} \right)^{k} \right] \leq C_{k,\delta} e^{- c_{k,\delta} x},
\]
where we recall that $c_k = k \gamma -Q$ and $c_{k,\delta} \to c_k$ as $\delta \to 0$. 
\end{Prop}
In fact, for $x>0$ it is possible, by using tail estimates for side-to-side distances, to show that the decay is Gaussian in $x$. We do not need this result so we omit it.

\begin{proof}

In order to keep the key ideas of the proof transparent, we postpone the proofs of some intermediate elementary lemmas to the end of this section. 
Consider the collection of balls
\begin{equation}
\label{Def:CollectionBalls}
\mathfrak{B} = \left\{ B_{e^{-\ell}}(z) \: : \: \ell \in \mathbb N_0, z \in e^{-\ell - 2} \Z^2, B_{e^{-\ell}}(z) \cap B_{10}(0) \neq \emptyset \right\}.
\end{equation}

We will work with three events in the proof: $E_{\delta,M}$ is a global regularity event, $F_{K,\delta,M}$ is an approximation of the event $\{ K \subset P_h^{1,e^{-\xi x}}\}$ which replaces the conditions on the metric by conditions on the field, and $F'_{K,\delta,M}$ is a variant of $F_{K,\delta,M}$ where $\gamma$-log singularities are added to the field at the points $z \in K$ (this is related to $P^{0,x}_K$). Here, $M$ is a parameter that is sent to $+\infty$ and $\delta$ is a small positive parameter.
The integer $k$ is fixed throughout the proof, so the events are allowed to depend on $k$ and we omit it in the notation.

\smallskip

\noindent\textbf{Step 1: truncating over a global regularity event $E$.} The event $E_{\delta, M}$ is given by the following criteria:
\begin{enumerate}
\item
\label{Reg:1} For all $\ell \geq 0$, the annulus crossing distance of $B \backslash 0.99B$ is at least $M^{-\xi} e^{-\xi  \ell^{\frac{1}{2}+\delta}  } e^{-\xi Q \ell}  e^{\xi \dashint_{\partial B} h}$ for all $B \in \mathfrak{B}$ with radius $e^{-\ell}$.

\item
\label{Reg:2}  For all  integers $\ell > \ell' \geq 0$, for all $B \in \mathfrak{B}$ of radius $e^{-\ell-2}$, we have $e^{-\ell} \sup_{6 k B} |\nabla \phi_{\ell',\ell}| \leq  \ell^{\frac{1}{2}+\delta} + \log M$.

\item
\label{Reg:4}
For all $\ell \geq 0$ and all $B \in \mathfrak{B}$ of radius $e^{-\ell-2}$, $\dashint_{\partial B} \phi_{\ell,\infty} \leq   \ell^{\frac{1}{2}+\delta}  + \log M $.

\item
\label{Reg:3}
$ \norme{\phi - h}_{\D}  = \norme{g}_{\D} \leq  \log M$.
\end{enumerate}
As we see later in Lemma~\ref{Lem:RegularityEvent}, for fixed $\delta$ the event $E_{\delta, M}$ occurs with superpolynomially high probability in $M$ as $M \to \infty$. Therefore, when looking at moments of $\mu_h (\mathcal{B}_1(0;D_h) \cap \D)$, one can restrict to moments truncated on $E_{\delta,M}$.

\medskip

By using Property \ref{Reg:3} of $E_{\delta,M}$ and  the definition of $\mu_{\phi}$ as a Gaussian multiplicative chaos (see Section \ref{sec-GMC}), we get
\[
\E\left[ \mathbbm1_{E_{\delta, M}} \mu_h\left( B_{10}(0) \cap P_h^{1,e^{-\xi x}} \right)^k \right]  \leq C_k M^{\gamma k} \E\left[ \mathbbm1_{E_{\delta, M}} \mu_{\phi}\left(B_{10}(0) \cap P_h^{1,e^{-\xi x}} \right)^k \right]
\]
and 
\begin{multline*}
\E\left[ \mathbbm1_{E_{\delta, M}} \mu_{\phi}\left(B_{10}(0) \cap P_h^{1,e^{-\xi x}} \right)^k \right] \\   =  \E \left[ \int_{B_{10}(0)^k} \mathbbm1_{E_{\delta, M}} \mathbbm 1\{z_i \in  P_h^{1,e^{-\xi x}} \text{ for all }i\} \mu_{\phi}(dz_1) \dots \mu_{\phi}(dz_k) \right] \\
\leq   \E \left[ \int_{B_{10}(0)^k} \mathbbm 1_{F_{K,\delta,M}} \mu_{\phi}(dz_1) \dots \mu_{\phi}(dz_k) \right],
\end{multline*}
where the event $F_{K,\delta,M}$ is defined in the following lemma.  In the first inequality above, the constant $C_k$ appears from the difference of definition between Gaussian multiplicative chaos measures and the Liouville quantum gravity measure; the former one is defined by renormalizing by a pointwise expectation whereas the latter one by $\eps^{\frac{\gamma^2}{2}}$.
\begin{Lem}
\label{Lem:ClaimEvent}
For $k \geq 2$, there exists a constant $C$ so that for any $k$-tuple of points $K = \{z_1, \dots, z_k\}  \subset \D$ we have the inclusion of events
\[
E_{\delta, M} \cap \{z_i \in P_h^{1,e^{-\xi x}} \text{ for all }i = 1,\dots, k \} \subset F_{K,\delta,M}
\]
where $F_{K,\delta,M}$ is the event that for all vertices $(S,m,\psi, \eta)$ of $T^0_K(\{\phi\})$ we have 
\begin{equation}
\label{eq:CdtEprior}
\psi + x < Q m + Cm^{\frac{1}{2}+\delta}  + C\log M.
\end{equation}
\end{Lem}
Essentially, Lemma~\ref{Lem:ClaimEvent} holds because $K \subset P^{1,e^{-\xi x}}_h$ implies that distances near each cluster are small. Then for each cluster, Property~\ref{Reg:1} of $E_{\delta,M}$  lets us convert bounds on distances to bounds on circle averages of $h$, Property~\ref{Reg:2} lets us replace the coarse field circle average with the coarse field evaluated at any nearby point, and Properties~\ref{Reg:4} and~\ref{Reg:3} allow us to neglect the fine field and the random continuous function $h - \phi$; this gives~\eqref{eq:CdtEprior}.

\smallskip

\noindent\textbf{Step 2: shifting LQG mass as $\gamma$-singularties.} We then use the following lemma to replace the terms $\mu_{\phi}(dz_i)$'s by $dz_i$ and $\gamma$-singularities. 
\begin{Lem}
\label{Lem:ShiftSingularities}
If $f$ is a bounded nonnegative measurable function, and $C_{a,b}$ are the covariances of $\phi_{a,b}$ (defined as in~\eqref{Def:kernel}), we have
\begin{align*}
& \E \left[ \int_{B_{10}(0)^k} f(\phi,z_1, \dots, z_k, \phi_1, \dots, \phi_{\ell}, \dots)  \mu_{\phi}(dz_1) \dots \mu_{\phi}(dz_k)\right] \\
& \leq   \int_{B_{10}(0)^k} \E \left[ f (\phi+ \gamma \sum_{i \leq k} C_{0,\infty}(\cdot, z_i), z_1, \dots, z_k, \phi_1 +  \gamma \sum_{i \leq k} C_{0,1}(\cdot, z_i), \dots, \phi_{\ell} + \gamma \sum_{i \leq k} C_{\ell-1,\ell}(\cdot, z_i), \dots ) \right] \\ 
& \times \exp \left(\frac{\gamma^2}{2}\sum_{i \neq j} C_{0,\infty} (z_i, z_j) \right)  dz_1, \dots dz_k .
\end{align*}
\end{Lem}
We apply Lemma \ref{Lem:ShiftSingularities} with $f = 1_{F_{K,\delta,M}}$ and we get
\begin{align*}
\E \int_{B_{10}(0)^k} \mathbbm 1_{F_{K,\delta,M}} \mu_{\phi}(dz_1) \dots \mu_{\phi}(dz_k) 
 \leq  \int_{B_{10}(0)^k}\P[F'_{K,\delta,M}] \exp (\frac{\gamma^2}{2}\sum_{i \neq j} C_{0,\infty}(z_i,z_j)) dz_1 \dots dz_k ,
\end{align*}
where $F'_{K,\delta,M}$ is the event that in the labeled tree $T^0_K(\{\phi\})$, for any path from the root $(S_1, m_1, \psi_1, \eta_1)$ to $(S_n, m_n, \psi_n, \eta_n)$, we have
\eqb\label{eq-tilted}
\psi_n + \gamma \sum_{i=1}^{n} \sum_{z \in K} C_{m_{i-1},m_i}(z, \mathrm{Left}(S_i)) + x \leq  Q m_n + C m_n^{\frac{1}{2}+\delta}  + C \log M .
\eqe
Note that by Lemma~\ref{lem-cov} below,~\eqref{eq-tilted} implies that for each vertex $(S_n, m_n, \psi_n, \eta_n)$ we have
\eqb\label{eq-nicest}
\psi_n + \eta_n  + x \leq  (Q + \delta) m_n + C \log M + 2C.
\eqe
(The term $2C$ comes from Lemma \ref{lem-cov} and the bound $C m_n^{\frac{1}{2}+\delta} \leq \delta m_n + C$, using that $\delta\in(0,1/2)$.) Now, the probability that~\eqref{eq-nicest} occurs for each vertex is precisely $P_K^{0,x- C \log M - 2C,\delta}$, defined in just before the Corollary \ref{Cor-Variant}, so we conclude that $\P[F'_{K,\delta,M}] \leq P_K^{0,x - C \log M - 2C,\delta}$.

\begin{Lem}\label{lem-cov}
For $k \geq 2$, there exists $C_k$ such that for $K \in B_{10}(0)^k$, for any path from the root $(S_1, m_1, \psi_1, \eta_1)$ to $(S_n, m_n, \psi_n, \eta_n)$ in the labeled tree $T^0_K(\{\phi\})$ we have, writing $m_0 = 0$, 
\begin{multline*}
\left| \eta_n - \gamma \sum_{i=1}^{n}  \sum_{z \in K} C_{m_{i-1},m_i}(z, \mathrm{Left}(S_i))\right|  \\ 
=  \left|  \gamma \sum_{i=1}^{n} (m_{i} - m_{i-1}) |S_i| -   \gamma \sum_{i=1}^{n}  \sum_{z \in K} C_{m_{i-1},m_i}(z, \mathrm{Left}(S_i))\right|   < C. 
\end{multline*}
\end{Lem}
By Proposition \ref{Prop:StarScaleProp}, for $K \subset B_{10}(0)$ we have $\exp (\frac{\gamma^2}{2}\sum_{i \neq j} C_{0,\infty}(z_i, z_j))  \leq C \prod_{i < j} | z_i - z_j|^{-\gamma^2}$. Combining all of the above bounds yields
\[
\E\left[ \mathbbm1_{E_{\delta, M}} \mu_h\left( B_{10}(0) \cap P_h^{1,e^{-\xi x}} \right)^k \right]  \leq  C_k M^{\gamma k} \int_{B_{10}(0)^k} \frac{P_{z_1, \dots, z_k}^{0, x-C\log M - 2C,\delta}}{\prod |z_i - z_j |^{\gamma^2}} dz_1 \dots d z_k.
\]
Finally, by Corollary~\ref{Cor-Variant} we conclude that for all $x \in \R$ we have
\begin{equation}
\label{eq:Step2Bound}
\E\left[ \mathbbm1_{E_{\delta, M}} \mu_h\left( B_{10}(0) \cap P_h^{1,e^{-\xi x}} \right)^k \right]  \leq  C_{k,\delta} M^{C} e^{-c_{k,\delta}x}.
\end{equation}

\noindent\textbf{Step 3: concluding the proof.}  
By Markov's inequality, we get,
\begin{align}
\label{eq:inter0}
\Pro [\mu_h(B_{10}(0) \cap P_{h}^{1,e^{-\xi x}}) \geq t ]  & \leq \Pro [ E_{\delta,M}^c ] + \Pro [E_{\delta,M}, \mu_h(B_{10}(0) \cap P_{h}^{1,e^{-\xi x}}) \geq t  ] \nonumber \\
&  \leq  \Pro [ E_{\delta,M}^c ] +  t^{-k} \E [1_{E_{\delta,M}}  \mu_h(B_{10}(0) \cap P_{h}^{1,e^{-\xi x}})^k].
\end{align}
The second term is bounded by~\eqref{eq:Step2Bound}. To control the first term,  we use the following lemma.
\begin{Lem}
\label{Lem:RegularityEvent}
For fixed $\delta \in (0,1/2)$, the regularity event $E_{\delta,M}$ occurs with superpolynomially high probability as $M \to \infty$. 
\end{Lem}
Combining these bounds, namely starting from \eqref{eq:inter0}, using \eqref{eq:Step2Bound} and the previous lemma, we get, for all $\delta$, $k$, $p$, a constant $C_{\delta,k,p}$ such that for all $x \in \R$ and for all $M,t > 0$,
\[
\Pro [\mu_h(B_{10}(0) \cap P_{h}^{1,e^{-\xi x}}) \geq t ]  \leq C_{k,\delta,p}  \left( M^{-p} + t^{-k} M^{C} e^{-c_{k,\delta} x}\right).
\]
By taking $M = t^{k/(p+C)} e^{c_{k,\delta} x/(p+C)}$, we get
\[
\Pro [\mu_h(B_{10}(0) \cap P_{h}^{1,e^{-\xi x}}) \geq t ]  \leq C_{k,\delta,p} t^{-\frac{p}{p+C} k} e^{- \frac{p}{p+C} c_{k,\delta} x}
\]
so by choosing $p$ large and integrating the tail estimate to obtain moment bounds, we obtain 
\[
\E [\mu_h(B_{10}(0) \cap P_{h}^{1,e^{-\xi x}})^{k - \delta}] \leq Ce^{-(c_{k, \delta} - \delta)x}.
\]
Then, by \eqref{eq:Step2Bound} and the Cauchy-Schwartz inequality, we get
\[
\E \left[\mu_h \left(B_{10}(0) \cap P_h^{1,e^{-\xi x}} \right)^{k} \right] \leq C_{k,\delta} M^{C} e^{-c_{k,\delta}x} + \P [ E_{\delta,M}^c ]^{1/2} \E \left[\mu_h \left(B_{10}(0) \cap P_h^{1,e^{-\xi x}} \right)^{2k} \right]^{1/2}
\]
and we conclude the proof of Proposition \ref{Prop:EstimateGFFbounded} by taking $M = e^{\eps |x|}$ for some small $\eps > 0$ (indeed, for this choice of $M$ we have $\P [ E_{\delta,M}^c ] \lesssim e^{-a|x|}$ for any $a>0$, and our earlier bound says that the $2k$th moment is at most exponential in $x$).
\end{proof} 

\subsubsection*{Annuli contributions and $\alpha$-singularities.}

Here, we use the proxy estimate to study moments of metric balls when the field has singularities. The link is made with the following deterministic remark. Recall that $A_{r/2}:= B_{r/2}(0) \backslash \ol{B_{r/4}(0)}$. If $z \in \cB_1(0; D_h) \cap A_{r/2}$ then $D_h(0,\partial B_{r/4}(0)) \leq 1$ and $z \in P_h^{r,1-D_h(0,\partial B_{r/4}(0))}$ (recall \eqref{eq:DefProxy} for the definition of $P_h^{r,d}$).

In the following lemma, we will study the LQG volume of the intersection of the unit metric ball with the unit Euclidean disk. To do so, we study first the contribution of small annuli to the volume and then use a H\"older inequality to conclude.

\begin{Lem}
\label{Lem:SmallAnnuli} Let $h$ be a whole-plane GFF such that $h_1(0) = 0$. Then for $\alpha < Q$,
\[
\E \left[ \mu_{h+ \alpha \log | \cdot |^{-1}}(\cB_1(0; D_{h+ \alpha \log | \cdot |^{-1}}) \cap \D) ^k )\right]  < \infty.
\]
\end{Lem}

\begin{proof}
Note that $\cB_1(0; D_h) \cap A_{r/2} \subset P_h^{r,1} \cap A_{r/2}$ and that the latter one is measurable with respect to the field $h|_{B_r(0)}$. We use a decoupling/scaling argument as follows. We write,
\begin{align*}
\mu_h(\cB_1(0; D_h) \cap A_{r/2}) & \leq 1_{D_h(0, \partial B_{r/4}(0)) \leq 1} \mu_h(P_h^{r,1} \cap A_{r/2}) \\
& = 1_{e^{\xi h_{r}(0)} D_{h - h_{r}(0)}(0,\partial B_{r/4}(0)) \leq 1} e^{\gamma h_{r}(0)} \mu_{h-h_{r}(0)} \left( A_{r/2} \cap P_{h-h_{r}(0)}^{r,e^{-\xi h_{r}(0)}} \right),
\end{align*}
and set $\wt h :=h(r\cdot)-h_{r}(0)$. By Lemma~\ref{lem-markovish} we have the equality in law $\wt h|_\D \eqd h|_\D$, and also $\wt h|_{\D}$ is independent of $h_{r}(0)$. Using the scaling of the metric  and of the measure, we get
\begin{align}
\label{eq:replace}
& \E \left[ \mu_h(\cB_1(0; D_h) \cap A_{r/2})^k  \right] \nonumber \\
& \leq \E \left[ 1_{e^{\xi h_{r}(0)} D_{h - h_{r}(0)}(0,\partial B_{r/4}(0)) \leq 1} e^{\gamma k h_{r}(0)} \mu_{h-h_{r}(0)} \left( A_{r/2} \cap P_{h-h_{r}(0)}^{r,e^{-\xi h_{r}(0)}} \right)^k \right] \nonumber \\
& \leq  r^{k \gamma  Q} \E \left[ 1_{e^{\xi h_{r}(0)} r^{\xi Q} D_{\tilde{h}}(0,\partial B_{1/4}(0)) \leq 1} e^{\gamma k h_{r}(0)}  \mu_{\tilde{h}} \left(A_{1/2} \cap P_{\tilde{h}}^{1,r^{-\xi Q} e^{-\xi h_r(0)}} \right)^k \right].
\end{align} 

We split the expectation with $1_{D_{\tilde h}(0,\partial B_{1/4}) \leq r^{\delta}}$ and 
$1_{D_{\tilde h}(0,\partial B_{1/4})  \geq r^{\delta}}$. Note first that for $p > 1$, by Proposition \ref{Prop:EstimateGFFbounded} and a moment computation for the exponential of a Gaussian variable with variance constant times $\log r^{-1}$,
\[
 \E \left[  e^{\gamma p k h_{r}(0)}  \mu_{\tilde{h}} \left(A_{1/2} \cap P_{\tilde{h}}^{1,r^{-\xi Q} e^{-\xi h_r(0)}} \right)^{kp} \right]  \leq C r^{\mathrm{power}},
\]
for some power whose value does not matter. Indeed, because of the superpolynomial decay of the event $\{D_{\tilde h}(0,\partial B_{1/4})  \leq r^{\delta} \}$ coming from Proposition \ref{prop-dist}, the quantity
\begin{multline*}
\E \left[ 1_{D_{\tilde h}(0,\partial B_{1/4})  \leq r^{\delta}} e^{\gamma k h_{r}(0)}  \mu_{\tilde{h}} \left(A_{1/2} \cap P_{\tilde{h}}^{1,r^{-\xi Q} e^{-\xi h_r(0)}} \right)^k \right] \\
\leq
\P[D_{\tilde h}(0,\partial B_{1/4})  \leq r^{\delta}]^{1/q} \E \left[  e^{\gamma p k h_{r}(0)}  \mu_{\tilde{h}} \left(A_{1/2} \cap P_{\tilde{h}}^{1,r^{-\xi Q} e^{-\xi h_r(0)}} \right)^{kp} \right]^{1/p}  
\end{multline*}
decays superpolynomially fast in $r$, by using H\"older's inequality with $\frac{1}{p} + \frac{1}{q} = 1$. 

From now on, we truncate on the event $\{D_{\tilde h}(0,\partial B_{1/4})  \geq r^{\delta} \}$ and we want to bound from above
\[
r^{k \gamma  Q} \E \left[ 1_{e^{\xi h_{r}(0)} r^{\xi Q+\delta} \leq 1} e^{\gamma k h_{r}(0)}  \mu_{\tilde{h}} \left(A_{1/2} \cap P_{\tilde{h}}^{1,r^{-\xi Q} e^{-\xi h_r(0)}} \right)^k \right].
\]
By Proposition \ref{Prop:EstimateGFFbounded}, since $A_{1/2} \subset B_{10}(0)$ and $h_r(0)$ is independent of $\wt h|_\D$, by writing $c_{k,\delta} = k \gamma - Q + \alpha_{\delta}$ for some small $\alpha_{\delta}$, we get
\begin{align*}
& r^{k \gamma  Q} \E \left[ 1_{e^{\xi h_{r}(0)} r^{\xi Q+\delta}  \leq 1} e^{\gamma k h_{r}(0)}  \mu_{\tilde{h}} \left(A_{1/2} \cap P_{\tilde{h}}^{1,r^{-\xi Q} e^{-\xi h_{r}(0)}} \right)^k \right] \\
& \leq C_k  r^{k \gamma  Q} r^{-c_{k,\delta} Q} \E \left[ 1_{e^{\xi h_{r}(0)} r^{\xi Q+\delta}  \leq 1} e^{\gamma k h_{r}(0)}  e^{-c_{k,\delta} h_{r}(0)} \right] \\
& = C_k r^{Q^2-Q\alpha_{\delta}} \E \left[ 1_{e^{\xi h_{r}(0)} r^{\xi Q+\delta} \leq 1} e^{(Q -\alpha_{\delta}) h_{r}(0)} \right].
\end{align*}
Furthermore, since
\[
\E \left[ 1_{e^{\xi h_{r}(0)} r^{\xi Q+\delta} \leq 1} e^{(Q -\alpha_{\delta}) h_{r}(0)} \right] \leq \E \left[e^{(Q -\alpha_{\delta}) h_{r}(0)} \right],
\]
by a Gaussian computation we get
\[
\E \left[ \mu_h(\cB_1(0; D_h) \cap A_{r/2})^k  \right]  \leq  C_k r^{\frac{1}{2}Q^2+\beta_{\delta}},
\]
for some arbitrarily small $\beta_{\delta}$. 

Furthermore, note  that when one replaces $h$ by $h + \alpha \log | \cdot |^{-1}$ for $\alpha < Q$, we get 
\begin{equation}
\label{eq:AlphaSing}
\E \left[ \mu_{h+ \alpha \log | \cdot |^{-1}}(\cB_1(0; D_{h+ \alpha \log | \cdot |^{-1}})  \cap A_{r/2})^k  \right]  \leq  C_k r^{\frac{1}{2}(Q-\alpha)^2 + \beta_{\delta} }.
\end{equation}
Indeed, on $A_{r/2}$, $\alpha \log | \cdot |^{-1}$ is of order $- \log r + O(1)$ so the volume term contributes an additional $r^{- k \gamma \alpha}$. Furthermore, by monotonicity, we can replace the intersection of the unit $D_{h + \alpha \log | \cdot |^{-1}}$-metric ball with $A_{r/2}$ by an order $r^{\xi \alpha}$ $D_h$-metric ball intersected with $A_{r/2}$.  Then, instead of using $\cB_1(0; D_h) \cap A_{r/2} \subset P_h^{r,1} \cap A_{r/2}$ at the beginning of the proof, we use  $\cB_{r^{\alpha \xi}}(0; D_h) \cap A_{r/2} \subset P_h^{r,r^{\alpha \xi}} \cap A_{r/2}$. Then we note that the term $r^{\xi Q}$ in \eqref{eq:replace} is replaced by $r^{\xi (Q-\alpha)}$. Therefore,~\eqref{eq:AlphaSing} follows by replacing $Q$ with $Q-\alpha$.

We can conclude as follows. Set $V_{r}^{\gamma,\alpha} := \mu_{h+ \alpha \log | \cdot |^{-1}}(\cB_1(0; D_{h+ \alpha \log | \cdot |^{-1}}\cap A_r))$. By monotone convergence,
\begin{align*}
\E \left[ \mu_{h+ \alpha \log | \cdot |^{-1}}(\cB_1(0; D_{h+ \alpha \log | \cdot |^{-1}}) \cap \D)^k\right] & = \lim_{n \to \infty} \E \left[ \left( \sum_{i=0}^{n} V_{2^{-i}}^{\gamma,\alpha} \right)^k \right].
\end{align*}
We introduce some deterministic $\Lambda >1$ to be chosen. By H\"older's inequality we get
\[
\left( \sum_{i=0}^{n} V_{2^{-i}}^{\gamma,\alpha} \right)^k  = \left( \sum_{i=0}^{n} \Lambda^i V_{2^{-i}}^{\gamma,\alpha} \Lambda^{-i} \right)^k  \leq \left( \sum_{i=0}^n \Lambda^{k i} (V_{2^{-i}}^{\gamma,\alpha})^k \right) \left( \sum_{i=0}^n \Lambda^{-i \frac{k}{k-1}}  \right)^{k-1}.
\]
Taking expectations, and using the bound \eqref{eq:AlphaSing}, we get, uniformly in $n$,
\[
\E \left[ \left( \sum_{i=0}^{n} V_{2^{-i}}^{\gamma,\alpha} \right)^k \right] \leq \left(  \frac{1}{1- \Lambda^{- \frac{k}{k-1}}} \right)^{k-1} \sum_{i=0}^{\infty} \Lambda^{k i}  2^{-i (\frac{1}{2}(Q-\alpha)^2+\beta_{\delta})}.
\]
Taking $\Lambda$ close enough to one such that $\Lambda^k 2^{-\frac{1}{2} (Q-\alpha)^2+\beta_{\delta}} < 1$, this series is absolutely convergent, as desired. 
\end{proof}

\begin{Lem}[Large annuli]
\label{Lem:LargeAnnuli}
Let $h$ be a whole-plane GFF such that $h_1(0) = 0$. Then, for $\alpha < Q$, $\E \left[ \mu_{h + \alpha \log | \cdot |^{-1}}(\cB_1(0; D_{h+ \alpha \log | \cdot |^{-1}})  \cap \C \setminus \D) ^k\right]  < \infty$.
\end{Lem}

\begin{proof}
The proof uses the proxy estimate and a decomposition over annuli with a scaling argument. This is similar to Lemma \ref{Lem:SmallAnnuli}. We point out here only the main differences with the proof of this lemma.

Write $D_h(0,\partial B_{R/4}(0)) =: R^{\xi Q} e^{\xi h_{R/4}(0)} X_R$.  Since $\cB_1(0; D_h) \cap A_{R} \subset P_h^{R,1} \cap A_{R}$
\begin{align*}
\E [\mu_h (\cB_1(0;D_h) \cap A_R)^k] & \leq \E [1_{D_h(0,\partial B_{R/4}(0)) \leq 1} \mu_h (P_h^{R,1} \cap A_R)^k] \\
& = \E [1_{R^{\xi Q} e^{\xi h_{R/4}(0)} X_R \leq 1} e^{k \gamma h_{R/4}(0)} \mu_{h-h_{R/4}(0)} (P_{h-h_{R/4}(0)}^{R,e^{-\xi h_{R/4}(0)}} \cap A_R)^k]
\end{align*}
We truncate again with $1_{X_R \leq R^{-\delta}}$ and $1_{X_R \geq R^{-\delta}}$. Because of the superpolynomial decay of $\Pro(X_R \leq R^{-\delta})$, the term associated with the former truncation is negligible compared to the other one. Furthermore, since we will have some room at the level of exponent, we will simply assume that $\delta = 0$ for the remaining steps. By using that $h-h_{R/4}(0))|_{A_{R/4,2R}(0)}$ is independent of $h_{R/4}(0)$ and that the proxy  $P_h^{R,x} \cap A_r$ is measurable with respect to $h|_{A_{R/4,2R}}$, we get by scaling,
\begin{multline*}
 \E (1_{R^{ Q} e^{h_{R/4}(0)}  \leq 1} \mu_{h-h_{R/4}(0)} (P_{h-h_{R/4}(0)}^{R,e^{-\xi h_{R/4}(0)}} \cap A_R)^k)  \\
 = R^{k \gamma Q}  \E (1_{R^{Q} e^{h_{R/4}(0)}  \leq 1} e^{k \gamma h_{R/4}(0)} \mu_{\tilde{h}} (P_{\tilde{h}}^{1,e^{-\xi h_{R/4}(0)} R^{-\xi Q}} \cap A_1)^k)
\end{multline*}
At this stage we use the estimate from Proposition \ref{Prop:EstimateGFFbounded}. Therefore, we compute
\begin{multline*}
R^{k \gamma Q}  \E (1_{h_{R/4}(0) \leq - Q \log R}   e^{k \gamma h_{R/4}(0)} e^{-c_k (h_{R/4}(0) + Q \log R)} ) \\
 = R^{k \gamma Q} e^{-c_k Q \log R}  \E \left( 1_{h_{R/4}(0) \leq -Q \log R} e^{Q h_{R/4}(0)} \right) 
\end{multline*}
and by using the Cameron-Martin formula we get
\begin{align*}
R^{k \gamma Q} e^{-c_k Q \log R}  & \E \left( 1_{h_{R/4}(0) \leq -Q \log R} e^{Q h_{R/4}(0)} \right)  \\ 
& \approx R^{Q^2} R^{\frac{Q^2}{2}}  \E \left( 1_{h_{R/4}(0) \leq -Q \log R} e^{Q h_{R/4}(0) - \frac{1}{2} Q^2 \log R/4} \right) \\
& \approx R^{ \frac{3}{2}Q^2} \Pro \left( h_{R/4}(0) \leq -2 Q \log R \right) \approx R^{-\frac{Q^2}{2}}.
\end{align*}
where $A_R \approx B_R$ if $A_R/B_R = R^{o(1)}$. So this gives
\[
\E [\mu_h (\cB_1(0;D_h) \cap A_R)^k]  \leq R^{-\frac{Q^2}{2}+o(1)}
\] 
The rest of the proof, namely taking into account all the annuli contributions and using H\"older inequality, is the same as the one of Lemma \ref{Lem:SmallAnnuli}.
\end{proof}

\begin{proof}[Proof of Proposition \ref{Prop:MomentsWholePlane}]
Let $h$ be a whole-plane GFF such that $h_1(0) = 0$ and fix $\alpha < Q$. The proof follows easily by writing
\begin{multline*}
\mu_{h + \alpha \log | \cdot |^{-1}}(\cB_1(0; D_{h+ \alpha \log | \cdot |^{-1}})) \\
=  \mu_{h + \alpha \log | \cdot |^{-1}}(\cB_1(0; D_{h+ \alpha \log | \cdot |^{-1}}) \cap \D )+ \mu_{h + \alpha \log | \cdot |^{-1}}(\cB_1(0; D_{h+ \alpha \log | \cdot |^{-1}}) \cap \C \setminus \D )
\end{multline*}
and using the inequality $(x+y)^k \leq 2^{k-1} (x^k + y^k)$  together with  Lemma \ref{Lem:SmallAnnuli} and Lemma \ref{Lem:LargeAnnuli}.
\end{proof}

\begin{Lem}[Upper bound for small metric balls]
\label{Lem:UpperSmallBalls}
For $\eps \in (0,1)$, $k \geq 1$, there exists a constant $C_{k,\eps}$ such that  for all $s \in (0,1)$,
\[
\E [\mu_h(\mathcal {B}_s(0;D_h))^k] \leq C_{k,\eps}  s^{k d_{\gamma} - \eps}
\]
\end{Lem}
\begin{proof}
The proof is very similar to the one of Lemma \ref{Lem:SmallAnnuli}, therefore we omit the details and just provide the differences. By replacing $1$ by $s$ in the proof, we get $\E [\mu_h(\mathcal {B}_s(0;D_h) \cap A_r)^k] \leq C_k s^{k d_{\gamma} - c_{\gamma}} r^{\frac{Q^2}{2}}$
where $c_{\gamma} = \frac{d_{\gamma}}{\gamma}Q$. By using H\"older's inequality, we get $\E [\mu_h(\mathcal {B}_s(0;D_h) \cap A_r)^k] \leq C_{kp}^{1/p} s^{k d_{\gamma} - \frac{c_{\gamma}}{p}} r^{\frac{Q^2}{2p}}$. We then take $p$ such that $c_{\gamma}/p < \eps$ and the rest of the proof follows the same line as those of Lemma \ref{Lem:SmallAnnuli}.
\end{proof}

\subsubsection*{Proofs of the intermediate lemmas for Proposition \ref{Prop:EstimateGFFbounded}}

We recall here the definition of the event $E_{\delta,M}$ (recall the definition of $\mathfrak{B}$ in \eqref{Def:CollectionBalls}). It is given by the following criteria:
\begin{enumerate}
\item
\label{RegBis:1} For all $\ell \geq 0$, the annulus crossing distance of $B \backslash 0.99B$ is at least $M^{-\xi} e^{-\xi  \ell^{\frac{1}{2}+\delta}  } e^{-\xi Q \ell}  e^{\xi \dashint_{\partial B} h}$ for all $B \in \mathfrak{B}$ with radius $e^{-\ell}$,

\item
\label{RegBis:2}  for all  integers $\ell > \ell' \geq 0$, for all $B \in \mathfrak{B}$ of radius $e^{-\ell-2}$, we have $e^{-\ell} \sup_{6 k  B} |\nabla \phi_{\ell',\ell}| \leq  \ell^{\frac{1}{2}+\delta} + \log M$,

\item
\label{RegBis:4}
for all $\ell \geq 0$ and for all $B \in \mathfrak{B}$ of radius $e^{-\ell-2}$, $\dashint_{\partial B} \phi_{\ell,\infty} \leq   \ell^{\frac{1}{2}+\delta}  + \log M $,

\item
\label{RegBis:3}
and $ \norme{\phi - h}_{\D}  = \norme{g}_{\D} \leq  \log M$.
\end{enumerate}
 \begin{proof}[Proof of Lemma \ref{Lem:ClaimEvent}] We prove here that for any $k$-tuple of points $K = \{z_1, \dots, z_k\} \subset \D$ we have
\begin{multline*}
E_{\delta, M} \cap \{z_i \in P_h^{1,e^{-\xi x}} \text{ for all }i = 1,\dots, k \} \\
 \subset \lbrace \psi + x \leq Q m + 8k^2(m^{\frac{1}{2}+\delta}  + \log M) \text{ for each vertex }(S,m,\psi, \eta) \text{ of } T^0_K(\{\phi\}) \rbrace.
\end{multline*}
Fix $K$ and consider any vertex $(S,m,\psi, \eta)$ of $T^0_K(\{\phi\})$. Recall first that by \eqref{eq:Def-psi},
\begin{equation}
\label{eq:Reminder}
\psi = \psi_n = \sum_{i=1}^{n} \phi_{m_{i-1},m_i}(\mathrm{Left}(S_i)),
\end{equation}
where we write $(S_i, m_i, \psi_i, \eta_i)$ for the path from the root $(S_1, m_1, \psi_1, \eta_1)$ to $(S_n, m_n, \psi_n, \eta_n) = (S,m,\psi, \eta)$. The proof is to compare a circle average around $z \in S$ (which can be bounded since $z \in P_h^{1,e^{-\xi x}}$) with the right-hand side above. Pick any point $z \in S$. Since $z \in P^{1,e^{-\xi x}}_h$, 
\[D_h(z, \partial B_{e^{-m-1}}(z)) \leq D_h(z, \partial B_{1/4} (z)) \leq e^{-\xi x}, \]
and we can find a ball $B \in \mathfrak{B}$, centered at a point in $e^{-m-4} \Z^2$ with radius $e^{-m-2}$ whose boundary separates $z$ from $\partial B_{e^{-m-1}}(z)$. Hence the annulus crossing distance of $B \backslash 0.99B$ is at most $e^{-\xi x}$. By Property~\ref{RegBis:1}, we have,
\[ M^{-\xi} e^{-\xi (m+2)^{\frac12 + \delta}} e^{-\xi Q(m+2)} e^{\xi \dashint_{\partial B} h} \leq e^{-\xi x},\]
or equivalently
\eqb\label{eq-circle-approx}
\dashint_{\partial B} h + x \leq  Q(m+2) + (m+2)^{\frac12 + \delta} + \log M.
\eqe
Now we lower bound $\dashint_{\partial B} h$ in term of \eqref{eq:Reminder} by using properties \ref{Reg:2}, \ref{Reg:4} and \ref{Reg:3} of $E_{\delta, M}$. 
\begin{itemize}
\item By Property~\ref{RegBis:3} we have
\[\dashint_{\partial B} h \geq \sum_{i=1}^{n} \dashint_{\partial B}\phi_{m_{i-1}, m_{i}} + \dashint_{\partial B} \phi_{m, \infty} - \log M. \]

\item For each $i$, notice that $z \in S_i$, and so $d(z, \mathrm{Left}(S_i)) \leq e k e^{-m_i}$ by~\eqref{eq-sep-diam}. Consequently, by Property~\ref{RegBis:2} we have for each $i = 1, \dots, n$
\[\dashint_{\partial B} \phi_{m_{i-1},m_i} \geq \phi_{m_{i-1}, m_i} (\mathrm{Left}(S_i)) - 4k m_{i}^{\frac12 + \delta} - 4k \log M. \]

\item By Property~\ref{RegBis:4} we have
\[\dashint_{\partial B} \phi_{m, \infty} \geq -m^{\frac12 + \delta} - \log M. \]
\end{itemize}
Combining these yields (see Remark~\ref{rem-nonrecursive})
\[\dashint_{\partial B} h \geq \sum_{i=1}^{n} \phi_{m_{i-1}, m_{i}}(\mathrm{Left}(S_i)) - 6k^2 m^{\frac12 + \delta} - 6k^2 \log M = \psi - 6k^2 m^{\frac12 + \delta} - 6k^2 \log M. \]
Together with~\eqref{eq-circle-approx}, this gives
$ \psi + x \leq Q m + 8k^2(m^{\frac{1}{2}+\delta}  + \log M)$ and concludes the proof.
\end{proof}

\begin{proof}[Proof of Lemma \ref{Lem:ShiftSingularities}]
This is an application of the Cameron-Martin theorem. We outline here the main idea, assuming for notational simplicity that the function $f$ depends only on $\phi, z_1, \dots, z_k$. The argument works the same way for $f$ depending also on $(\phi_n)_{n \geq 0}$. 

Assume first that $f$ is continuous. Fix $k \geq 2$, $\delta > 0$ and set $C_{\delta} := \{ (z_1, \dots, z_k) \in B_{10}(0)^k : \min_{i < j} |z_i - z_j| \geq \delta \}$. Then, by using Fatou's lemma and the Cameron-Martin formula, we have
\begin{align*}
& \E \left[ \int_{B_{10}(0)^k \cap C_{\delta}} f(\phi,z_1, \dots, z_k)  \mu_\phi(dz_1) \dots \mu_\phi(dz_k)\right] \\
& \leq \liminf_{\eps \to 0} \E \left[ \int_{B_{10}(0)^k \cap C_{\delta}} f(\phi,z_1, \dots, z_k)  \frac{ e^{\gamma \phi_{\eps}(z_1)}}{\E[e^{\gamma \phi_{\eps}(z_1)}]} \dots  \frac{ e^{\gamma \phi_{\eps}(z_k)}}{{\E[e^{\gamma \phi_{\eps}(z_k)}]}} dz_1 \dots dz_k   \right] \\
& = \liminf_{\eps \to 0}  \int_{B_{10}(0)^k \cap C_{\delta}}  \frac{dz_1 \dots dz_k}{e^{-\frac{\gamma^2}{2} \sum_{i \neq j} \Cov ( \phi_{\eps}(z_i),\phi_{\eps}(z_j))}} \E \left[ e^{\gamma \sum_{i \leq k} \phi_{\eps}(z_i) - \frac{\gamma^2}{2} \Var (\sum_{i \leq k} \phi_{\eps}(z_i))} f(\phi,z_1, \dots, z_k)  \right]   \\
&  = \liminf_{\eps \to 0}  \int_{B_{10}(0)^k \cap C_{\delta}} \frac{dz_1 \dots dz_k}{e^{-\frac{\gamma^2}{2} \sum_{i \neq j} \Cov ( \phi_{\eps}(z_i),\phi_{\eps}(z_j))}} \E \left[ f(\phi + \gamma \sum_{i \leq k} \Cov(\phi(\cdot),\phi_{\eps}(z_i)),z_1, \dots, z_k)  \right]    \\
& = \int_{B_{10}(0)^k \cap C_{\delta}} \frac{dz_1 \dots dz_k}{ e^{-\frac{\gamma^2}{2} \sum_{i \neq j} \Cov ( \phi(z_i),\phi(z_j))}} \E \left[ f(\phi + \gamma \sum_{i \leq k} \Cov(\phi(\cdot),\phi(z_i)),z_1, \dots, z_k)  \right] .
\end{align*} 
where we used the dominated convergence theorem in the last equality (the term $\sum_{i\neq j} \Cov (\phi(z_i), \phi(z_j))$ is uniformly bounded for $(z_1,\dots, z_n) \in C_\delta$). The Cameron-Martin formula is used by writing
\[
\gamma \sum_{i \leq k} \phi_{\eps}(z_i) = \langle \phi,  \gamma \sum_{i \leq k} \rho_{\eps,z_i} \rangle
\]
where $\rho_{\eps,z_i}$ denote the uniform probability measure on the circle $\partial B_{\eps}(z_i)$. 
Note that the above inequality was only shown for continuous $f$, but we can approximate general  bounded nonnegative measurable $f$ by a sequence of continuous $f_n$ which converge pointwise to $f$, and apply the dominated convergence theorem. Thus the above inequality holds for general $f$.

Finally, letting $\delta$ going to zero and using the monotone convergence theorem, we get 
\begin{multline*}
 \E \left[ \int_{B_{10}(0)^k} f(\phi,z_1, \dots, z_k)  \mu_\phi(dz_1) \dots \mu_\phi(dz_k)\right] \\ \leq   \int_{B_{10}(0)^k} e^{\frac{\gamma^2}{2} \sum_{i \neq j} \Cov ( \phi(z_i),\phi(z_j))} \E \left[ f(\phi + \gamma \sum_{i \leq k} \Cov(\phi(\cdot),\phi(z_i)),z_1, \dots, z_k)  \right] dz_1 \dots dz_k.
\end{multline*}
This concludes the proof.
\end{proof}

\begin{proof}[Proof of Lemma~\ref{lem-cov}]
	It suffices to show that for some constant $C$, for each $z \in K$ and each $i = 1,\dots, n$, writing $w = \mathrm{Left}(S_i)$ we have 
	\[ \left|C_{m_{i-1}, m_i}(z, w) - (m_i - m_{i-1})1_{z \in S_i} \right| < C.\]
	If $z \not \in S_i$, then by definition $d(z, w) \geq d(z, S_i) \geq e^{-m_{i-1}}$. This is larger than the range of dependence of $\phi_{m_{i-1}, m_i}$, so $C_{m_{i-1}, m_i}(z, w) = 0$ as desired. 
	
	Now suppose $z \in S_i$. By~\eqref{eq-sep-diam}, we know that $S_i$ is contained in a ball of radius $6ke^{-m_i}$; by translation invariance we may assume this ball is centered at the origin. On $B_{6k}(0)\times B_{6k}(0)$, the correlation of $\phi_{0,\infty}$ is $C_{0, \infty}(\cdot, \cdot) = \log | \cdot - \cdot |^{-1} + q(\cdot- \cdot)$ for some bounded continuous $q$. Thus, by scale invariance, we can write
\begin{align*}
C_{m_{i-1}, m_i}(z,w) & = C_{0,m_i - m_{i-1}} (e^{m_{i-1}}z,e^{m_{i-1}} w) \\
& = \log |e^{m_{i-1}}(z -w)|^{-1} - C_{m_{i}- m_{i-1}, \infty}(e^{m_{i-1}}z, e^{m_{i-1}} w) + O(1).
\end{align*}
But again by scale invariance we have
\[
C_{m_{i}- m_{i-1}, \infty}(e^{m_{i-1}}z, e^{m_{i-1}} w) = C_{0, \infty}(e^{m_{i}}z, e^{m_{i}}w) = \log |e^{m_{i}} (z-w)|^{-1} + O(1). 
\]
Comparing these two equations we conclude that $C_{m_{i-1}, m_{i}}(z, w) = m_{i} - m_{i-1} + O(1)$, as needed.
\end{proof}

Finally we check the bound on the regularity event $E$.

\begin{proof}[Proof of Lemma \ref{Lem:RegularityEvent}] We prove here the estimate of the occurence of the event $E_{\delta,M}$. 
	
	\smallskip
	
	For all  integers $\ell > \ell' \geq 0$, for all $B \in \mathfrak{B}$ of radius $e^{-\ell-2}$, the probability that  $e^{-\ell} \sup_{6 kB} |\nabla \phi_{\ell',\ell}| >  \ell^{\frac{1}{2}+\delta} + \log M$ is $\leq C e^{-c (\log M)^2} e^{-c \ell^{1+2\delta}}$ by Lemma \ref{Lem:GradientEst}. Therefore, the probability that Condition \ref{RegBis:2} does not hold is $\leq C e^{-c (\log M)^2} \sum_{\ell \geq 0}  \ell e^{2\ell} e^{-c \ell^{1+2\delta}}$. 
	
	\smallskip
	
	For Condition \ref{RegBis:4},
	for a $B \in \mathfrak{B}$ of size $e^{-\ell-2}$, by scaling  $\dashint_{\partial B} \phi_{\ell,\infty}$ is distributed as $\dashint_{\partial B_0} \phi_{0,\infty}$ where $B_0$ is of size $e^{-2}$ and this is a centered Gaussian variable with bounded variance. Therefore, the probability it is at least $  \ell^{\frac{1}{2}+\delta}  + \log M $ is less than $ Ce^{-c (\ell^{\frac{1}{2}+\delta} + \log M)^2} \leq C e^{-c \ell^{1+2\delta}} e^{-c (\log M)^2}$. For each $\ell$, there are $O(e^{2 \ell})$ balls of size $e^{-\ell-2}$ in $\mathfrak{B}$, hence the probability that Condition \ref{RegBis:4} does not hold is less than $C e^{-c (\log M)^2} \sum_{\ell \geq 0}  e^{2\ell} e^{-c \ell^{1+2\delta}} $. 
	
	\smallskip
	
	For Condition \ref{RegBis:3}, since $\phi - h$ is continuous by Proposition~\ref{Prop:StarScaleProp}, and applying Fernique's theorem, the probability that $ \norme{\phi - h}_{\D} \leq  \log M$ occurs is $\geq 1- Ce^{-c (\log M)^2}$. For Condition \ref{Reg:1}, we use Proposition \ref{prop-dist} and again a union bound.
\end{proof}

\section{Negative moments}\label{sec-negative}

In this section, we prove the following lower bound on the LQG volume of the unit metric ball. 
\begin{Prop}[Negative moments of LQG ball volume]
\label{Prop:NegativeMoments}
Let $h$ be a whole-plane GFF normalized so $h_1(0) = 0$. Then 
\[\E\left[\mu_h(\mc{B}_1(0;D_h))^{-p} \right] < \infty \text{ for all } p \geq 0.\]
This result also holds if we instead consider the LQG measure and metric associated with the field $\wt h = h - \alpha\log |\cdot|$ for $\alpha < Q$.
\end{Prop}
In Section~\ref{sec-lower-unit}, we prove the finiteness of negative moments of $\mu_h(\cB_1(0;D^\D_h))$, the unit ball with respect to the $\D$-internal metric $D^\D_h$. This immediately implies Proposition~\ref{Prop:NegativeMoments} since $\cB_1(0;D^\D_h) \subset \cB_1(0;D_h)$. In Section~\ref{sec-lower-general} we bootstrap our results to obtain lower bounds on $\mu_h(\cB_s(0; D_h))$ for $s \in (0,1)$; these lower bounds will be useful in our applications in Section~\ref{sec-applications}.

\subsection{Lower tail of the unit metric ball volume}\label{sec-lower-unit}
The goal of this section is the following result.
\begin{Prop}[Superpolynomial decay of internal metric ball volume lower tail]
\label{prop-lower}
Let $h$ be a whole-plane GFF normalized so $h_1(0) = 0$. Let $D^\D_h: \D\times \D \to \R$ be the internal metric in $\D$ induced by $D_h$, and $\cB_1(0; D^\D_h) \subset \D$ the $D^\D_h$-metric ball. Then for any $p>0$, for all sufficiently large $C>0$ we have
\[\P\left[\mu_h(\cB_1(0;D^\D_h)) \geq C^{-1} \right] \geq 1 - C^{-p} .\]
This result also holds if we instead consider the LQG measure and metric associated with the field $\wt h = h - \alpha\log |\cdot|$ for $\alpha < Q$.
\end{Prop} 

Let $N >1$ be a parameter which we keep fixed as $C \to \infty$ (taking $N$ large yields $p$ large in Proposition~\ref{prop-lower}) and define
\[ k_0 = \left\lfloor \frac1N\log C \right\rfloor, \qquad k_1 = \lfloor N \log C \rfloor.\]
Let $P$ be a $D_{\wt h}$-geodesic from $0$ to $\partial B_{e^{-k_0}}(0)$. See Figure~\ref{fig-internal-ball} (left) for the setup.

\textit{Proof sketch of Proposition~\ref{prop-lower}.} The proof follows several steps. Each step below holds with high probability.
\begin{itemize}
\item We find an annulus $B_{e^{-k+1}}(0) \backslash \ol{B_{e^{-k}}(0)}$ with $k > k_0$ not too large, such that the annulus-crossing length of $P$ is not too small. This is possible because the $D_{\tilde{h}}$-length of $P$ between $\partial B_{e^{-k_1}}(0)$ and $\partial B_{e^{-k_0}}(0)$ is at least $C^{-\beta}$ for some fixed $\beta > 0$. We conclude that the circle average $\wt h_{e^{-k}}(0)$ is not small ($\wt h_{e^{-k}}\gtrsim -\log C$).
\item We find a $D_h$-metric ball which is ``tangent'' to $\partial B_{e^{-k}}(0)$ and $\partial B_{e^{-k-1}}(0)$.
Then, by Proposition~\ref{prop-euc-ball}, this metric ball (and hence $\cB_1(0; D^\D_{\wt h})$) contains a Euclidean ball $B$ with Euclidean radius not too small (say $e^{-(1+\zeta)k}$ for small $\zeta > 0$). Since $\wt h_{e^{-k}}(0)$ is not small, neither is the average of $\wt h$ on $\partial B$ (i.e. $\dashint_{\partial B} \wt h \gtrsim -\log C$).
\item Finally, we have a good lower bound on $\mu_{\wt h}(B)$ in terms of the average of $\wt h$ on $\partial B$, so we find that $B$ has not-too-small LQG volume. Since $B$ lies in $\cB_1(0; D^\D_{\wt h})$, we obtain a lower bound $\mu_{\wt h}(\mc B_1(0; D^\D_{\wt h}))\gtrsim C^{-\text{power}}$. This last exponent does not depend on $N$, so we may take $N \to \infty$ to conclude the proof of Proposition~\ref{prop-lower}.
\end{itemize}

We now turn to the details of the proof. Let $L_k$ be the $D_{\wt h}$-length of the subpath of $P$ from $0$ until the first time one hits $\partial B_{e^{-k}}(0)$. We emphasize that $L_k$ is \emph{not} the $D_{\wt h}$ distance from $0$ to $\partial B_{e^{-k}}(0)$.

\begin{Lem}[Length bounds along $P$]
\label{lem-Lk}
There exist positive constants $c = c(\gamma,\alpha)$ and $\beta = \beta(\gamma,\alpha)$ independent of $N$ such that for sufficiently large $C$, with probability $1 - O(C^{-c N})$ the following all hold:
\eqb \label{eq-Lk0}
L_{k_0} > C^{-\beta},
\eqe
\eqb \label{eq-Lk1}
L_{k_1} < C^{-\beta - 1},
\eqe
\eqb \label{eq-Lk}
L_{k-1} - L_k <  C\exp\left(-k \xi (Q-\alpha) + \xi h_{e^{-k}}(0) \right) \quad \text{ for all }k \in [k_0 +1 , k_1].
\eqe
\end{Lem}

\begin{proof}
We focus first on \eqref{eq-Lk0}. Using Proposition~\ref{prop-dist} to bound the annulus crossing distance of $B_{e^{-k_0}}(0)\backslash \overline{ B_{e^{-k_0-1}}(0)}$, we see that with superpolynomially high probability as $C \to \infty$ we have
\eqb\label{eq-k0-bound}
L_{k_0} \geq C^{-1} \left( e^{-k_0}\right)^{\xi (Q-\alpha)} \exp(\xi h_{e^{-k_0}}(0)).
\eqe
Note that since $\Var(h_{e^{-k_0}}(0)) = k_0 \leq N^{-1} \log C$, we have
\[\P\left[ \xi h_{e^{-k_0}}  (0)< -\log C \right] \leq \exp\left(- \frac{(\log C)^2}{2 \xi^2 N^{-1} \log C}\right) = C^{-c N} \]
for $c = 1/(2\xi^2)$. Notice that when we have both \eqref{eq-k0-bound} and $\{\xi h_{e^{-k_0}} \geq -\log C\}$, then
\[L_{k_0} \geq  C^{-1} \cdot C^{-\xi (Q-\alpha)/N} \cdot C^{-1}  \geq C^{-\beta}\]
for the choice $\beta  =2 + \xi (Q-\alpha)$. Thus \eqref{eq-Lk0} holds with probability $1 - O(C^{-c N})$.

To prove the upper bound \eqref{eq-Lk}, we glue paths to bound $L_{k-1} - L_k$. By Proposition~\ref{prop-dist} and a union bound, with superpolynomially high probability as $C \to \infty$ the following event $E_C$ holds:
\begin{itemize}
\item For each $k \in [k_0 + 1, k_1]$, there exists a path from $\partial B_{e^{-k+1}}(0)$ to $\partial B_{e^{-k-1}}(0)$ and paths in the annuli $B_{e^{-k}}(0) \backslash \overline{B_{e^{-k-1}}(0)}$ and $B_{e^{-k+2}}(0) \backslash \overline{B_{e^{-k+1}}(0)}$ which separate the circular boundaries of the annuli,  and such that each of these path has $D_{\wt h}$-length at most $\frac13C \exp\left(-k \xi (Q-\alpha) + \xi h_{e^{-k}}(0) \right)$.  
\end{itemize}
Since the segment on $P$ measured by $L_{k-1}-L_k$ is the restriction of a geodesic which crosses a larger annulus, by triangular equality, \eqref{eq-Lk} holds on $E_C$.

Finally, we check that for our choice of $\beta$, the inequality \eqref{eq-Lk1} holds with probability $1 - C^{-c N}$ (possibly by choosing a smaller value of $c > 0$). 
By the triangle inequality, $L_{k_1}$ is bounded from above by the sum of the $D_{\wt h}$-distance from the origin to $\partial B_{e^{-k_1 + 1}}(0)$ plus the $D_{\wt h}$-length of any circuit in the annulus $B_{e^{-k_1+1}}(0) \backslash \ol{B_{e^{-k_1}}(0)}$. Hence, using the circuit bound on $E_C$, we have
\[L_{k_1} \leq D_{\wt h}(0,\partial B_{e^{-k_1+1}}(0)) + C e^{-k_1 \xi (Q-\alpha)} e^{\xi h_{e^{-k_1}}(0)}.\]
By scaling of the metric, $D_{\wt h}(0,\partial B_{e^{-k_1+1}}(0))$ is bounded from above by $e^{\xi h_{e^{-k_1+1}}(0)}  e^{(-k_1+1) \xi (Q-\alpha)}Y$ where $Y$ is distributed as $D_{\wt h}(0,\partial B_1(0))$. Now, since $k_1 = \lfloor N \log C \rfloor$ and $h_{e^{-k_1}}(0)$ has variance $N \log C$, by a Gaussian tail estimate we get
\[\P\left[  h_{e^{-k_1}}  (0) > \frac{1}{4} k_1 (Q-\alpha)  \right] \leq  C^{-c N}. \]
Furthermore, since $Y$ has some finite small moments for $\alpha <Q$ (by \cite[Theorem 1.10]{DFGPS}), the Markov's inequality provides
\[ \P \left[ Y e^{-\frac{1}{4} k_1 \xi (Q-\alpha)} > 1  \right] \leq C^{-cN}. \]
Altogether, we obtain \eqref{eq-Lk1} with probability $1-O(C^{-cN})$. 
\end{proof}

As an immediate consequence of the above lemma, we can find a scale $k \in (k_0, k_1]$ such that $\cB_1(0;D_{\wt h}^{\D})$ intersects $\partial B_{e^{-k}}(0)$, and the field average at scale $k$ is large. We introduce here a small parameter $\zeta > 0$ which does not depend on $C$, whose value we fix at the end.

\begin{Lem}[Existence of large field average near $\cB_1(0;D^\D_{\wt h})$]\label{lem-large-field}
Consider $c$ and $\beta$ as in Lemma~\ref{lem-Lk}. With probability $1 - O(C^{-c N})$, there exists $k \in [k_0, k_1]$ such that $D_{\wt h}(0, \partial B_{e^{-k}}(0)) < 1$ and 
\eqb \label{eq-estimate-field}
-k (Q-\alpha) + h_{e^{-k}}(0)  \geq - \xi^{-1}(\beta + 2)\log C;
\eqe
moreover, there exists a Euclidean ball $B_r(z)$ with $r = e^{-k(1+\zeta)}$ and $z \in r \Z^2$ such that $B_r(z) \subset B_{e^{-k}}(0) \backslash \overline{B_{e^{-k-1}}(0)}$ and $B_r(z) \subset \cB_1(0; D^\D_{\wt h})$. 
\end{Lem}

\begin{proof}
To prove~\eqref{eq-estimate-field}, we first claim that when the event of Lemma~\ref{lem-Lk} holds, there exists $k \in [k_0 + 1, k_1]$ such that $L_k < 1$ and $L_{k-1} - L_k \geq C^{-\beta - 1}$. Let $k_\star$ be the smallest $k \in (k_0,k_1]$ such that $L_{k_\star} < C^{-\beta}$, then 
\[\sum_{k = k_\star}^{k_1} L_{k-1} - L_k = L_{k_\star - 1} - L_{k_1} \geq C^{-\beta} - C^{-\beta - 1}.\]
Since the LHS is a sum over at most $N \log C$ terms, we indeed find some index $k \in [k_\star, k_1]$ such that 
\[L_{ k-1} - L_k \geq \frac{C^{-\beta} - C^{-\beta - 1}}{N \log C} > C^{-\beta - 1}.\]
For this choice of $k$, we have $D_{\wt h}(0, \partial B_{e^{-k}}(0)) \leq L_k \leq L_{k_\star} < C^{-\beta} < 1$, and by~\eqref{eq-Lk} we have~\eqref{eq-estimate-field} also.

\begin{figure}[ht!]
\begin{center}
  \includegraphics[scale=0.60]{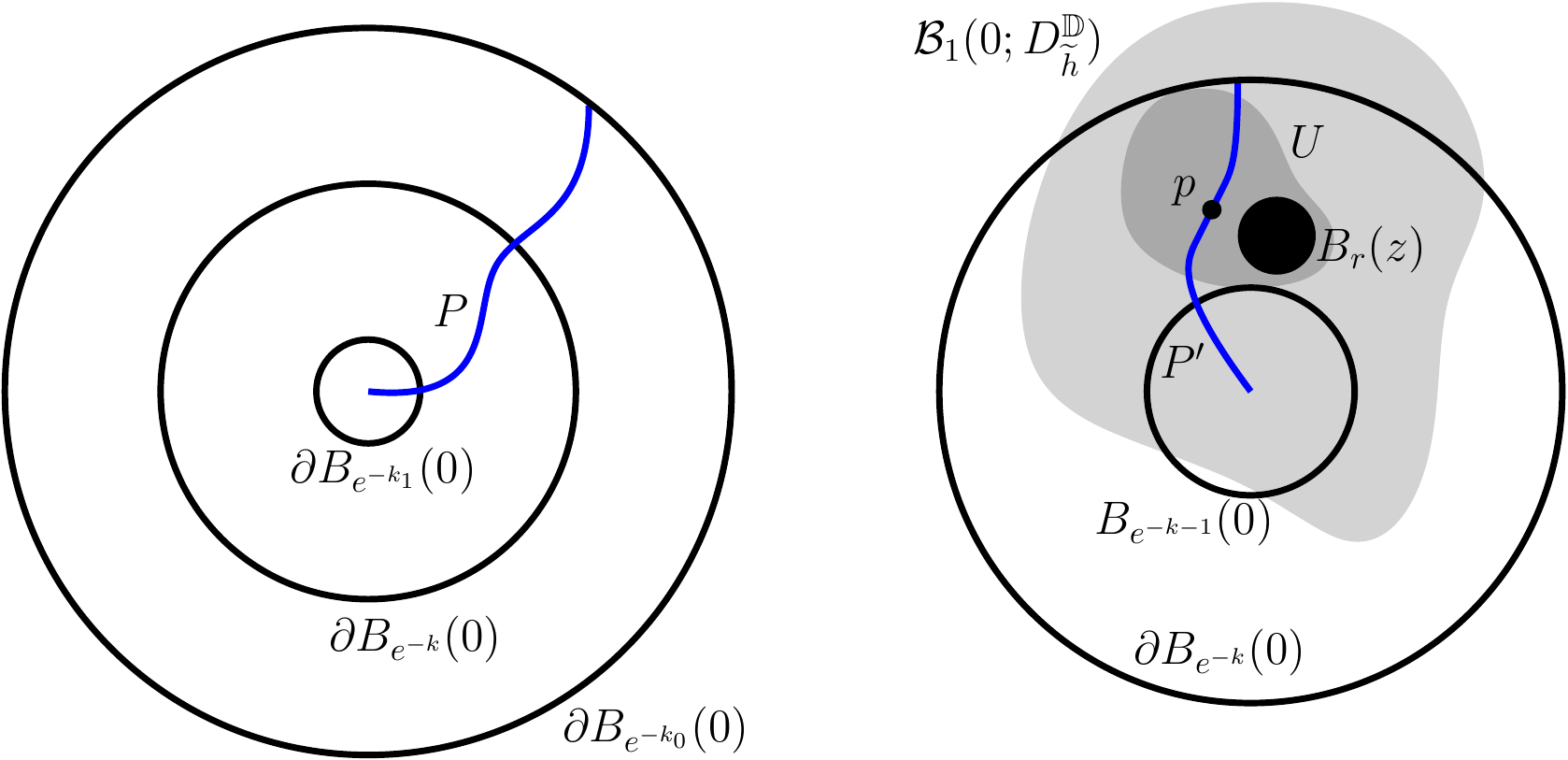}%
 \end{center}
\caption{\label{fig-internal-ball} \textbf{Left: }Setup of Lemma~\ref{lem-Lk}. Given $C$ that we eventually sent to $\infty$, we take the circles with radii $e^{-k_0} \approx C^{-1/N}$ and $e^{-k_1} = C^{-N}$, and draw all circles with radii $e^{-k}$ with $k_0 \leq k \leq k_1$. In Lemma~\ref{lem-large-field} we follow the geodesic $P$ from the outer circle to the inner until we find an annulus on which the geodesic segment is long. \textbf{Right: }Illustration of the second assertion of Lemma~\ref{lem-large-field}. We find a $D_h$-metric ball $U\subset \cB_1(0; D^\D_{\wt h})$ such that $U$ is ``tangent'' to $\partial B_{e^{-k}}$ and $\partial B_{e^{-k-1}}$, then apply Proposition~\ref{prop-euc-ball} to find a Euclidean ball $B_r(z) \subset U$.}
\end{figure}

Now we turn to the second assertion of the lemma; see Figure~\ref{fig-internal-ball} (right). Let $P'$ be a $D_{\wt h}$-geodesic from $0$ to $\partial B_{e^{-k}}(0)$. By the continuity of $D_{\wt h}$, we can find a point $p \in P'$ in the annulus $B_{e^{-k}}(0) \backslash \overline{B_{e^{-k-1}}(0)}$ such that $D_{h+(k+1)\alpha} (p, \partial B_{e^{-k}}(0)) = D_{h+(k+1)\alpha} (p, \partial B_{e^{-k-1}}(0))$; let $U$ be the $D_{h + (k+1)\alpha}$-ball with this radius centered at $p$. 

We claim that $U \subset \cB_1(0; D^\D_{\wt h})$. We assume that $\alpha \geq 0$ (the other case is similar). Since $(k+1)\alpha \geq \alpha \log |\cdot|^{-1} \geq k \alpha$ on $B_{e^{-k}}(0) \backslash \overline{B_{e^{-k-1}}(0)}$, we have for all $w \in U$ that 
\[D^\D_{\wt h}(p,w) \leq e^{\xi \alpha} D^\D_{h + \alpha k}(p,w) \leq e^{\xi \alpha} D^\D_{h + \alpha k}(p, \partial B_{e^{-k}}(0)) \leq e^{\xi \alpha} D^\D_{\wt h}(p, \partial B_{e^{-k}}(0)),\]
and consequently 
\[D^\D_{\wt h}(0, w) \leq D^\D_{\wt h}(0,p) + D^\D_{\wt h}(p, w) \leq  D^\D_{\wt h}(0,p) + e^{\xi \alpha} D^\D_{\wt h}(p, \partial B_{e^{-k}}(0)) \leq e^{\xi \alpha} D^\D_{\wt h} (0, \partial B_{e^{-k}}(0));\]
this last inequality follows from the fact that $p$ lies on $P'$ so $D^\D_{\wt h}(0,p) + D^\D_{\wt h}(p, \partial B_{e^{-k}}(0)) = D^\D_{\wt h} (0, \partial B_{e^{-k}}(0))$. Since $D^\D_{\wt h} (0, \partial B_{e^{-k}}(0)) \leq L_{k_\star} < C^{-\beta}$, we conclude that $D^\D_{\wt h}(0, w) < e^{\xi \alpha} C^{-\beta} \leq 1$, and hence $U \subset \cB_1(0; D^\D_{\wt h})$. 

Since $U$ is a $D_{h+(k+1)\alpha}$ metric ball, it is also a $D_h$ metric ball. Furthermore, since $\diam (U) \in (\frac12 e^{-k}, 2e^{-k})$,  Proposition~\ref{prop-euc-ball} gives us a Euclidean ball of radius $e^{-k(1+\zeta/2)}$ in $U$, and hence a Euclidean ball $B_{r}(z) \subset U$ with $z \in r\Z^2$. Since $U$ lies in $B_{e^{-k}}(0) \backslash \overline{B_{e^{-k-1}}(0)}$ and in $\cB_1(0; D^\D_{\wt h})$, so does $B_r(z)$, so we have shown Lemma~\ref{lem-large-field}. 
\end{proof}

Finally, we need a regularity event to say that the $\mu_{\wt h}$-volumes of Euclidean balls are close to their field average approximations, and that the field does not fluctuate too much on each scale. The bounds in the following lemma are standard in the literature. We introduce a large parameter $q>0$ that does not depend on $C$, and fix its value at the end. 

\begin{Lem}[Regularity of field averages and ball volumes]\label{lem-field-fluctuation}
Fix $\zeta \in (0,1)$ and $q>0$. Then for all sufficiently large $C> C_0(q,\zeta,N)$, with probability $1 - C^{-\zeta( \frac{q^2}{2N} - 2N - 1)}$ the following is true. For each $k \in [k_0, k_1]$, writing $r  = e^{-k(1+\zeta)}$, for all $z \in r \Z^2$ such that $B_r(z) \subset B_{e^{-k}} (0) \backslash\overline{ B_{e^{-k-1}}(0)}$ we have
\eqb\label{eq-field-fluctuation}
\left|  h_r(z) - h_{e^{-k}} (0)  \right| < k q\zeta 
\eqe
and
\eqb\label{eq-ball-field-average}
\mu_{\wt h}(B_r(z)) \geq C^{-1} r^{\gamma Q} \exp(\gamma \wt h_r(z)).
\eqe
\end{Lem}
\begin{proof}
By standard GFF estimates, we have $\Cov\left(h_r(z), h_{e^{-k}}(0)\right) = k + O(1)$, $\Var h_r (z) = -\log r + O(1) = k(1+\zeta) + O(1)$ and $\Var h_{e^{-k}}(0) = k + O(1)$.
Consequently, 
\alb
\Var\left(h_{r}(z) - h_{e^{-k}} (0)\right) = \zeta k + O(1),
\ale
and hence by the Gaussian tail bound,
\[\P \left[ | h_r(z) - h_{e^{-k}} (0) | < kq\zeta  \right] \geq 1 - O(e^{-\frac{q^2\zeta k}{2}}). \]
Taking a union bound over all $O(e^{2k\zeta})$ points in $r\Z^2 \cap B_{e^{-k}}(0)$, then summing over all $k \in [k_0, k_1]$, we see that the probability \eqref{eq-field-fluctuation} holds for all $k$ and all suitable $z$ is at least
\[1 - O\left( \sum_{k=k_0}^{k_1} e^{2k \zeta} e^{-q^2 \zeta k /2} \right) \geq 1 -  O\left( N \log C  \cdot e^{2k_1 \zeta}  e^{-q^2 \zeta k_0 /2} \right) \geq 1 - C^{-\zeta(\frac{q^2}{2N} - 2N - 1)}. \]

Now, we establish that for each fixed choice of $k,z$, the inequality \eqref{eq-ball-field-average} holds with superpolynomially high probability as $C \to \infty$ (then we are done by a union bound over a collection of polynomially many $k,z$); since $ -\alpha \log |\cdot| -\alpha k$ is bounded on the annulus, it suffices to show \eqref{eq-ball-field-average} with $\wt h$ replaced by $h + \alpha k$ (or equivalently by $h$, since both sides of the equation \eqref{eq-ball-field-average} scale the same way under adding a constant to the field). By the Markov property of the GFF (Lemma~\ref{lem-markov}) we can decompose $h = \mathfrak h + \wh h$, where $\mathfrak h$ is a distribution which is harmonic in $B_{2r}(z)$, and $\wh h$ is a zero boundary GFF in the domain $B_{2r}(z)$; moreover $\mathfrak h$ and $\wh h$ are independent. We can then write
\alb
\mu_h(B_r(z)) &\geq e^{\gamma \inf_{B_r(z)} \mathfrak h } \mu_{\wh h} (B_r(z)) \\
&= (2r)^{\gamma Q} e^{\gamma h_r(z)} e^{-\gamma \wh h_r(z)} e^{ \gamma \inf_{B_r(z)} \mathfrak h  - \gamma \mathfrak h (z)} \mu_g (B_{\frac12}(0)),
\ale
where $g := \wh h(2r \cdot + z)$ has the law of a zero boundary GFF on $\D$. (This follows from an affine change of coordinates mapping $B_{2r}(z) \mapsto \D$; then by the coordinate change formula $\mu_{\wh h} (B_r(z)) = (2r)^{\gamma Q} \mu_g (B_{\frac12}(0))$.)

Since $\wh h_r(z)$ is a mean zero Gaussian with fixed variance, and by the quantum volume lower bound~\eqref{eq:LeftTailsEucli}, we have $e^{-\gamma \wh h_r(z)}  \geq C^{-1/3}$ and $\mu_g (B_{\frac12}(0)) \geq C^{-1/3}$ with superpolynomially high probability in $C$. Combining these bounds with the above estimate, with superpolynomially high probability in $C$ we have
\[\mu_h(B_r(z)) \geq (2r)^{\gamma Q} C^{-2/3} e^{ \gamma \inf_{B_r(z)} \mathfrak h  - \gamma \mathfrak h (z)}.\]
Hence we are done once we check that with superpolynomially high probability in $C$,
\eqb\label{eq-bound-harmonic}
e^{ \gamma \inf_{B_r(z)} \mathfrak h  - \gamma \mathfrak h (z)} \geq C^{-1/3}.
\eqe
Since $h = \frak h + \wh h$ and $\frak h, \wh h$ are independent, for $x, x' \in B_r(z)$ we have
\[\Var\left( \mathfrak h (x) - \mathfrak h(x') \right) \leq \Var\left( h_r (x) - h_r(x') \right) = O(1). \]
Moreover, by the scale and translation invariance of the GFF modulo additive constant and the fact that $\mathfrak h$ is continuous in $B_{\frac32r}(z)$, we know that $\mathfrak h (z) - \inf_{B_r(z)} \mathfrak h > -\infty$ and has a law independent of $r,z$, so by the Borell-TIS inequality we see that for some absolute constants $m,c$, we have
\[\P \left[ \mathfrak h(z) - \inf_{B_r(z)} \mathfrak h > u + m \right] \leq e^{-c u^2} \quad \text{ for all } u>0.\]
This immediately implies \eqref{eq-bound-harmonic}. Thus, for each fixed choice of $k,z$, the inequality \eqref{eq-ball-field-average} holds with superpolynomially high probability as $C \to \infty$. Taking a union bound, we obtain \eqref{eq-ball-field-average}.
\end{proof}

\begin{proof}[Proof of Proposition~\ref{prop-lower}]
Let $c,\beta$ be as in Lemma~\ref{lem-Lk}. We will work with parameters $N, \zeta, q$, and choose their values at the end. Assume that the events of Lemmas~\ref{lem-large-field} and~\ref{lem-field-fluctuation} hold; this occurs with probability at least $1 - C^{-c N} - C^{-\zeta (\frac{q^2 }{2N} - 2N - 1)}$. Let $k$, $r$, and $B_r(z)$ be as in Lemma~\ref{lem-large-field}.

We now lower bound the quantum volume of $B_r(z)$. By \eqref{eq-estimate-field} and \eqref{eq-field-fluctuation}, we see that
\alb
r^{\gamma Q} \exp\left(\gamma \wt h_r (z)\right) & \geq  \exp\left(-\gamma kQ(1+\zeta) + \gamma h_r(z) + \gamma \alpha  k \right) \\
&\geq \exp(-\gamma \zeta k(Q+q) - \gamma k(Q-\alpha) + \gamma h_{e^{-k}}(0) ) \\
&\geq \exp(-\gamma \zeta k(Q+q)) C^{-\frac{\gamma}{\xi}(\beta +2)} \\
&\geq C^{-\gamma \zeta N ( Q + q )} C^{-\frac{\gamma}{\xi}(\beta + 2)} .
\ale
The last inequality follows from $k \leq k_1 = \lfloor N \log C \rfloor$. Choose $q = N^3$ and $\zeta = N^{-4}$. Then by the above inequality, \eqref{eq-ball-field-average}, and $B_r(z) \subset \cB_1(0;D^\D_{\wt h})$, we see that for a constant $\beta' = \beta'(\gamma)>0$ we have 
\[ \mu_{\wt h}(\cB_1(0;D^\D_{\wt h})) \geq \mu_{\wt h}(B_r(z)) \geq C^{-\beta'}.\]
Since this occurs with probability $1 - C^{-c N} - C^{-\zeta(\frac{q^2 }{2N} - 2N - 1)} = 1 - O(C^{-c N}) $, and $N$ can be made arbitrarily large, we have proved Proposition~\ref{prop-lower}. 
\end{proof}

\subsection{Lower tail of small metric balls}\label{sec-lower-general}
Using Proposition~\ref{prop-lower} and the scaling properties of the LQG metric and measure, we can easily prove a similar result for metric balls centered at the origin of all radii $s \in (0,1)$. We emphasize that in the following proposition, we are considering the $D_h$-metric balls, rather than $D^\D_{h}$-metric balls.

\begin{Lem}
\label{lem-lower-all-r}
Let $h$ be a whole-plane GFF normalized so $h_1(0) = 0$. For any $p>0$, there exists $C_p$ such that for all $C > C_p$ and $s \in (0,1)$, we have
\[\P\left[\mu_h(\cB_s(0;D_h) ) \geq C^{-1} s^{d_\gamma} \right] \geq 1 - C^{-p} .\]
\end{Lem} 

\begin{proof}
The process $t \mapsto h_{e^{-t}}(0)$ for $t \geq 0$ evolves as standard Brownian motion started at $0$. Fix $s \in (0,1)$ and let $T>0$ be the first time $t>0$ that $-Qt + h_{e^{-t}}(0) = \xi^{-1} \log s$. Notice that 
\begin{align*}
h(e^{-T} \cdot ) + Q \log e^{-T} & = \left(h(e^{-T} \cdot ) -  h_{e^{-T}}(0) \right) -QT + h_{e^{-T}}(0) \\
& =   \left(h(e^{-T} \cdot ) -  h_{e^{-T}}(0) \right) + \xi^{-1} \log s.
\end{align*}
By Lemma~\ref{lem-markovish}, conditioned on $T$, we have $(h(e^{-T} \cdot ) + Q \log e^{-T})\big|_{\D}  \stackrel{d}{=} (\wh h + \xi^{-1} \log s)\big|_{\D}$ where $\wh h$ is a whole-plane GFF normalized to have mean zero on $\partial \D$. Couple these fields to agree. By the Weyl scaling relations and the change of coordinates formula for quantum volume and distances, and the locality property of the internal metric (Axiom~\hyperref[axiom-locality]{II}), we have the internal metric relation 
\[D^{e^{-T} \D}_{h} (e^{-T}z, e^{-T}w) = D^\D_{\wh h + \xi^{-1} \log s} (z,w) = s D^\D_{\wh h} (z,w)\]
and the volume measure relation 
\[\mu_h(e^{-T} \cdot) = \mu_{\wh h + \xi^{-1} \log s} (\cdot) = s^{d_\gamma} \mu_{\wh h}(\cdot).\]
Thus we can relate the quantum volume of the internal metric balls $\cB_s(0; D^{e^{-T} \D}_h) \subset e^{-T}\D$ and $\cB_1(0; D^\D_{\wh h})$:
\[\mu_h\left( \cB_s(0; D^{e^{-T}\D }_{h}) \right) = s^{d_\gamma} \mu_{\wh h}( \cB_1(0; D^\D_{\wh h + \xi^{-1} \log s}) ),\]
and consequently we have 
\[\left\{ \mu_h(\cB_s(0;D^{e^{-T} \D}_h)  ) \geq C^{-1} s^{d_\gamma} \right\} =  \left\{ \mu_{\wh h }(\cB_1(0;D^\D_{ \wh h }) ) \geq C^{-1}  \right\}.\]
Since $\mu_h(\cB_s(0;D_h)) \geq \mu_h(\cB_s(0;D^{ e^{-T} \D}_h) )$, our claim follows from Proposition~\ref{prop-lower}.
\end{proof}

\section{Applications and other results}\label{sec-applications}

\subsection{Uniform volume estimates and Minkowski dimension}

\label{sec-UVE}

In this section, we prove the remaining assertions of Theorem \ref{thm-main}. Namely, the Minkowski dimension of a bounded open set $S$ is almost surely equal to $d_{\gamma}$ and  for any compact set $K\subset \C$ and $\eps > 0$, we have, almost surely
\[\sup_{s \in (0,1)} \sup_{z \in K} \frac{\mu_h(\cB_s(z;D_h))}{s^{d_\gamma - \eps}} < \infty \quad \text{and} \quad \inf_{s\in (0,1)} \inf_{z \in K} \frac{\mu_h(\cB_s(z;D_h))}{s^{d_\gamma + \eps}} > 0. \]

Since the whole-plane GFF modulo additive constants has a translation invariant law, we can deduce a version of Lemma~\ref{lem-lower-all-r} for metric balls centered at $z\neq 0$. 
\begin{Prop}[Uniform lower tail for $\mu_h(\cB_s(z;D_h))$]
\label{prop-lower-general}
Let $h$ be a whole-plane GFF normalized so $h_1(0) = 0$, and $K \subset \C$ be any compact set. For any $p>0$, there exists $C_{p,K} > 0$ such that
\[\sup_{s\in(0,1), z\in K} \P\left[\mu_h(\cB_s(z;D_h) ) \geq C^{-1} s^{d_\gamma} \right] \geq 1 - C^{-p}  \quad \textrm{for each }C > C_{p,K}.\]
\end{Prop} 

\begin{proof}
Fix $z \in K$. We can write $h = \wh h + X$ where $\wh h$ is a whole-plane GFF normalized so $\wh h_1(z) = 0$, and $X = h_1(z)$ is a random real number. On the event $\{ |X| \leq \gamma^{-1} \log C\}$ we have $C^{-1} \leq e^{\gamma X}\leq C$, so 
\begin{align*}
\{ \mu_h(\cB_s(z;D_h)) < C^{-3} s^{d_\gamma}\} &= \{ e^{\gamma X}  \mu_{\wh h} (\cB_{e^{-\xi X}s} (z; D_{\wh h}))    < C^{-3} s^{d_\gamma} \} \\
&\subset \{ C^{-1} \mu_{\wh h} (\cB_{C^{-1/d_\gamma}s} (z; D_{\wh h}))    < C^{-3} s^{d_\gamma} \}  \cup \{|X| > \gamma^{-1} \log C \} \\
&= \{  \mu_{\wh h} (\cB_{C^{-1/d_\gamma}s} (z; D_{\wh h}))    < C^{-1} (C^{-1/d_\gamma}s)^{d_\gamma} \}  \cup \{|X| > \gamma^{-1} \log C \}.
\end{align*}
In the last line, the first event is superpolynomially rare in $C$ by Lemma~\ref{lem-lower-all-r}, and the second because $X$ is a centered Gaussian. Note that $\Var X = \Var h_1(z)$ is uniformly bounded for all $z \in K$, so the decay of the second event is uniform for $z \in K$. This completes the proof. 
\end{proof}

Similarly, we can bootstrap Lemma~\ref{Lem:UpperSmallBalls} to a statement uniform for $D_h$-balls centered in a compact set.  
\begin{Prop}[Uniform upper tail for $\mu_h(\cB_s(z;D_h))$]
\label{prop-upper-general}
Let $h$ be a whole-plane GFF normalized so $h_1(0) = 0$. For any compact set $K \subset \C$, $p>0$, $\eps \in (0,1)$, there exists a constant $C_{p,\eps,K}>0$ such that
\[\sup_{s\in(0,1), z\in K} \P\left[\mu_h(\cB_s(z;D_h) ) \leq C s^{d_\gamma-\eps} \right] \geq 1 - C^{-p} \quad \textrm{for each }C > C_{p,\eps,K}.\]
\end{Prop} 
\begin{proof}
We note that Lemma~\ref{Lem:UpperSmallBalls} implies an upper bound version of Lemma~\ref{lem-lower-all-r} (with an exponent of $d_\gamma - \eps$ instead of $d_\gamma$), and we deduce Proposition~\ref{prop-upper-general} in the same way that we obtain Proposition~\ref{prop-lower-general} from Lemma~\ref{lem-lower-all-r}.
\end{proof}

Before moving to the proof of the almost sure uniform estimate, we first prove volume bounds on a countable collection of metric balls. 
\begin{Lem}\label{lem-lower-bc}
For any $\eps > 0$ and bounded open set $2\D$, the following is true almost surely. For all sufficiently large $m$, for all $z \in 2^{-m} \Z^2 \cap 2\D$, and for all dyadic $s = 2^{-k} \in (0,1]$ we have 
\[s^{d_\gamma - \eps} 2^{\eps m} > \mu_h(\cB_s(z; D_h)) > s^{d_\gamma + \eps} 2^{-\eps m}. \]
\end{Lem}
\begin{proof}
The proof is a straightforward application of Propositions \ref{prop-upper-general} and~\ref{prop-lower-general} and the Borel-Cantelli lemma. We prove the lower bound; the upper bound follows the same argument. 

Pick any large $p>0$, and let $C_{p,2\D}$ be the constant from Proposition~\ref{prop-lower-general}. Consider any $m$ such that $2^{\eps m} > C_{p,2\D}$, then for any $z \in 2\D$ we have
\[\P \left[\mu_h(\cB_s(z; D_h)) > s^{d_\gamma + \eps} 2^{-\eps m} \text{ for all dyadic } s \in (0,1]\right] > 1 - 2^{-\eps p m} \sum_{\text{dyadic }s} s^{\eps p}. \]
Taking a union bound over all the $O(2^{2m})$ points in $2^{-m} \Z^2 \cap 2\D$ yields
\begin{multline*}
\P \left[\mu_h(\cB_r(z; D_h)) > s^{d_\gamma + \eps} 2^{-\eps m} \text{ for all dyadic } s \in (0,1] \text{ and } z \in 2^{-m} \Z^2 \cap 2\D \right]  \\
> 1 - O(2^{-(\eps p -2) m} )\sum_{\text{dyadic }s} s^{\eps p}. 
\end{multline*}
For $p$ large enough we have $\eps p -2 > 0$, so by the Borel-Cantelli lemma, a.s. at most finitely many of the above events fail, i.e. the lower bound of Lemma~\ref{lem-lower-bc} holds. The upper bound follows the same argument. 
\end{proof}

With this lemma and the bi-H\"older continuity of $D_h$ with respect to Euclidean distance, we can prove the second part of Theorem~\ref{thm-main}.

\begin{proof}[Proof of Theorem~\ref{thm-main} part 2.]

We first prove that a.s. for some random $r \in (0,1)$, we have 
\eqb \label{eq-small-r}
\inf_{s \in (0,r]} \inf_{z \in \D} \frac{\mu_h(\cB_s(z; D_h))}{s^{d_\gamma + \zeta}}  > 0.
\eqe

We use the bi-H\"older continuity of $D_h$ with respect to Euclidean distance (see e.g. \cite[Theorem 1.7]{DFGPS}) and the Borel-Cantelli lemma to obtain the following. There exist deterministic constants $\chi, \chi' > 0$ and random constant $c, C$ such that, almost surely,
 \[c|u - v|^{\chi'} \leq D_h(u,v) \leq C |u-v|^\chi \quad \text{ for all }u,v \in 2\D. \]
Moreover, Proposition~\ref{prop-euc-ball} and Borell-Cantelli yield that a.s. every metric ball $B$ contained in $2\D$ and having sufficiently small Euclidean diameter contains a Euclidean ball of radius at least $\diam (B)^{2}$. 

Consequently, for all sufficiently small $s$ and any $z \in \D$, we have
\[\frac s2 \leq C \diam (\cB_{s/2}(z; D_h))^\chi,\]
and since any two points in $\cB_{s/2}(w;D_h)$ have $D_h$-distance at most $s$, the bi-H\"older lower bound gives
\[c \diam (\cB_{s/2}(z; D_h))^{\chi'} \leq s. \]
Since the ball $\cB_{s/2}(z; D_h)$ has a small diameter, it a.s. contains a Euclidean ball of radius at least $\diam (\cB_{s/2}(z; D_h))^2 \geq (s/2C)^{2/\chi}$ hence contains a point $w \in 2^{-m} \Z^2$ with $m = \lceil -\frac2\chi \log_2 (s/2C) \rceil < -\frac3\chi \log_2 (s/2C) $. 

Thus, for a random constant $c'$, for sufficiently small $s$, applying Lemma~\ref{lem-lower-bc} to $m$ as above and dyadic $s_1 \in (\frac s4, \frac s2]$, we have
\[\mu_h (\cB_{s/2}(w; D_h)) \geq \mu_h (\cB_{s_1}(w; D_h)) \geq s_1^{d_\gamma + \eps} \cdot 2^{-\eps m} \geq \left( \frac s4 \right)^{d_\gamma + \eps} \left(\frac s{2C} \right)^{\frac{3\eps}{\chi}}= c' s^{d_\gamma + \eps + \frac{3\eps}{\chi}}. \]
Since $w \in \cB_{s/2}(z; D_h)$, by the triangle inequality we have $\cB_{s/2}(w;D_h) \subset \cB_s(z; D_h)$, so 
\[ \mu_h(\cB_s(z; D_h)) > c' s^{d_\gamma + 3\eps + \frac{3\eps}{\chi}}.\]
Almost surely, this holds for all sufficiently small $s>0$ and all $z \in \D$. Choosing $\eps > 0$ so that $\eps + \frac{3\eps}\chi < \zeta$, we obtain~\eqref{eq-small-r}.

The supremum analog of~\eqref{eq-small-r} follows almost exactly the same proof, except that instead of finding a ``dyadic'' metric ball inside each radius $s$ metric ball, we find a dyadic metric ball $\wt \cB$ (with dyadic radius $s_1 \in [2s, 4s)$) around each metric ball $\cB$, then apply Lemma~\ref{lem-lower-bc} to upper bound $\mu_h(\wt \cB)$ (and hence $\mu_h(\cB)$). 

Now, we extend~\eqref{eq-small-r} to a supremum/infimum over all $s \in (0,1]$. For any $s \in (r,1]$ and $z \in \D$, we have 
\[\frac{\mu_h(\cB_s(z; D_h))}{s^{d_\gamma + \zeta}} \geq \mu_h(\cB_s(z; D_h)) \geq  r^{d_\gamma + \zeta} \frac{\mu_h(\cB_r(z; D_h))}{r^{d_\gamma + \zeta}} ,\]
and noting that a.s. for sufficiently large $R$ we have $D_h(\D, \partial B_R(0)) > 1$, 
\[\frac{\mu_h(\cB_s(z; D_h))}{s^{d_\gamma - \zeta}} \leq r^{-d_\gamma + \zeta} \mu_h(B_R(0)) < \infty.\]
This concludes the proof of the uniform volume estimates.
\end{proof}

Finally, we prove the statement from Theorem \ref{thm-main} about the Minkowski dimension of a set.

\begin{proof}[Proof of Theorem \ref{thm-main}, part 3.] Consider any bounded measurable set $S$ containing an open set and fix $\delta \in (0,1)$. Let $N_{\eps}^{S}$ be the minimal number of LQG metric balls with radius $\eps$ needed to cover the set $S$ and denote by $\mc{C}_{\eps}$ the set of centers associated to such a covering. Then, since
\[
\mu_h(S) \leq \sum_{z \in \mc{C}_{\eps}} \mu_h(\mc{B}_{\eps}(z;D_h)) \leq N_{\eps}^S \max_{z \in \mc{C}_{\eps}} \mu_h(\mc{B}_{\eps}(z;D_h)),
\]
the uniform volume estimate and the fact that $\mu_h(S) > 0$ a.s.  imply that for every $\delta >0$, we have the a.s. lower bound $\liminf_{\eps \to 0}  \frac{\log N_{\eps}^S}{\log \eps^{-1}} \geq d_{\gamma}- \delta$. Now, denote by $M_{\eps}^S$ the maximal number of pairwise disjoint LQG metric balls with radius $\eps$ whose union is included in $S$. Denote by $\mc{D}_{\eps}$ the set of centers associated to such a collection of metric balls. Note that $M_{\eps}^S \geq N_{2\eps}^S$. Therefore,
\[
\mu_h(S) \geq \sum_{z \in \mc{D}_{\eps}}\mu_h(\mc{B}_{\eps}(z;D_h)) \geq M_{\eps}^S \min_{z \in \mc{D}_{\eps}} \mu_h(\mc{B}_{\eps}(z;D_h)) \geq N_{2\eps}^S \min_{z \in \mc{D}_{\eps}} \mu_h(\mc{B}_{\eps}(z;D_h))
\]
from which we get the a.s. upper bound $\limsup_{\eps \to 0}  \frac{\log N_{\eps}^S}{\log \eps^{-1}} \leq d_{\gamma} +\delta$
by the uniform volume estimate and the fact that $\mu_h(S) < \infty$ almost surely. Letting $\delta \to 0$ completes the proof.
\end{proof}

\subsection{Estimates for Liouville Brownian motion metric ball exit times}\label{sec-exit}
Liouville Brownian motion is, roughly speaking, Brownian motion associated to the LQG metric tensor ``$e^{\gamma h} (dx^2 + dy^2)$'', and was rigorously constructed independently in the works \cite{GRVbm} and \cite{B15}. These papers consider fields different from our field $h$ (a whole-plane GFF normalized so $h_1(0) = 0$), but their results are applicable in our setting. This can be verified either directly or by local absolute continuity arguments.
 
Liouville Brownian motion was defined in \cite{GRVbm, B15} by applying an $h$-dependent time-change to standard planar Brownian motion. Letting $B_t$ be standard planar Brownian motion from the origin sampled independently from $h$, we can define Liouville Brownian motion as $X_t = B_{F^{-1}(t)}$ for $t \geq 0$, where $F$ is a random time-change defined $h$-almost surely. The function $F(t)$ should be understood as the quantum time elapsed at Euclidean time $t$, and has the following explicit description. Defining the approximation
\eqb\label{eq-clock}
F^\eps(t) = \int_0^t \eps^{\gamma^2/2}e^{\gamma h_{\eps}(B_s)}ds,
\eqe
and writing $T_R$ for the Euclidean time that $B_t$ exits the ball $B_R(0)$, the sequence $F^\eps|_{[0,T_R]}$ converges almost surely as $\eps \to 0$ to $F|_{[0,T_R]}$ in the uniform metric \cite[Theorem 1.2]{B15}.

For a set $X \subset \C$ and $z \in \C$, denote by $\tau_h(z;X)$ the first exit time of the Liouville Brownian motion started at $z$ from the set $X$. We discuss now the results of \cite{GRVbm} on the moments of $\tau_h(z;B_1(z))$ and of $F(t)$, i.e. the moments of the elapsed quantum time at some Euclidean time. These results are analogous to the moments of the LQG volume of a Euclidean ball (Section~\ref{sec-euc-vol}).

\begin{Prop}[Moments of quantum time {\cite[Theorem 2.10, Corollary 2.12, Corollary 2.13]{GRVbm}}] For all $q \in (-\infty, 4/\gamma^2)$, $t > 0$, the following holds,
\[\E [ \tau_h(0;B_1(0))^q ] +  \E[F(t)^q] < \infty. \]
\end{Prop}
Heuristically, the nonexistence of large moments is due to the Brownian motion hitting regions of small Euclidean size but large quantum size. On the other hand, the random set $\cB_1(0; D_h)$ in some sense avoids such regions.

In this section we prove the finiteness of \emph{all} moments of the LBM first exit time of $\cB_1(0; D_h)$, which we abbreviate as $\tau$, and discuss the moments of $\tau_h(0;\cB_s(0; D_h))$ for small $s \in (0,1)$. 

\paragraph*{Upper bound for LBM exit time of metric balls}

\begin{thm}[Positive moments for quantum exit time of metric ball]\label{thm-LBM}
Let $h$ be a whole-plane GFF normalized so $h_1(0) = 0$, and consider Liouville Brownian motion associated to $h$.
Let $\tau$ be the first exit time of the Liouville Brownian motion started at the origin from the ball $\cB_1(0; D_h)$, i.e. 
\[\tau = \inf \{ t \geq 0 \: : \: X_t \not \in \cB_1(0; D_h)\}. \]
Then
\[\E[\tau^k] < \infty \quad \text{ for all }k \geq 0. \]
\end{thm}
\noindent\textbf{Proof sketch:} In computing $\E[\tau^k]$, by first averaging out the randomness of $(B_t)_{t \geq 0}$, we obtain an expectation in $h$ of an integral over $k$-tuples of points in $\cB_1(0; D_h)$; this is similar to the integral in Step 1 of the proof of Proposition~\ref{Prop:EstimateGFFbounded}, but with additional log-singularities between these points. Because the arguments of Proposition~\ref{Prop:EstimateGFFbounded} had some room in the exponents, the log-singularities pose no issue for us, and we can carry out the same arguments from Section~\ref{section-positive}. We will be succinct when adapting these arguments.

Let $\tau_n$ be the quantum time LBM spends in the annulus $A_{2^n} := B_{2^n}(0) \backslash \overline{B_{2^{n-1}}(0)}$ before exiting $\cB_1(0; D_h)$. As in \cite[(B.2)]{GRVbm}, we have the following representation of $\E[\tau_n^k]$ for $k$ a positive integer, which follows from taking an expectation over the standard Brownian motion $(B_t)_{t \geq 0}$ used to define $(X_t)_{t \geq 0}$ (see \eqref{eq-clock}),
\eqb \label{eq-tau-k}
\E[\tau_n^k] = \E\left[ \int_{(A_{2^n})^k} f(z_1, \dots, z_k, h) \mathbbm1\{ z_1, \dots, z_k \in \cB_1(0;D_h)\} \mu_h(dz_1)\dots \mu_h(dz_k)\right],
\eqe
and where, writing $t_0 = 0$ and $z_0 = 0$ for notational convenience,  $f$ is given by
\begin{align} 
f(z_1, \dots, z_k, h) := & \int_{0 \leq t_1 \leq \dots \leq t_k < \infty} \frac{k!}{(2\pi)^{k/2} \prod_{i=1}^k (t_i - t_{i-1})}\exp \left( -\sum_{i=1}^k \frac{|z_i - z_{i-1}|^2}{2(t_i - t_{i-1})} \right) \label{eq-f}\\
&\times\P \left[B|_{[0,t_k]} \text{ stays in } \cB_1(0;D_h) \mid h, B_{t_i} = z_i \text{ for }i=1,\dots, k \right] dt_1 \dots dt_k. \nonumber
\end{align}
The function $f(z_1, \dots, z_k)$ is an integral of the Brownian motion transition density at times $t_1, \dots, t_k$ times the conditional probability that the Brownian motion does not escape $\cB_1(0; D_h)$. 
We will need the following bound on $f$, whose proof is postponed to the end of the section. 
\begin{Lem}\label{lem-bound-f}
There exists a constant $C > 0$ such that for all sufficiently large $R > 0$, on the event $\{ \cB_1(0; D_h) \subset B_R(0)\}$ we have
\[f(z_1, \dots, z_k, h) \leq C \left(\log R \right)^k g(z_1, \dots, z_k) \quad \text{ for all }z_1, \dots, z_k \in R\D,\]
where, recalling $z_0 =0$,
\[g(z_1, \dots, z_k)=\prod_{i=1}^k \max\left(- \log |z_i - z_{i-1}|,1 \right).\]
\end{Lem}

\begin{proof}[Proof of Theorem~\ref{thm-LBM}]
Our strategy is to fix some large $R > 0$ then truncate on the event $E'_R := \{ \cB_1(0;D_h) \subset B_R(0)\}$. Subsequently, we show an analog of Proposition~\ref{Prop:EstimateGFFbounded}, and use it to bound $\E[ \tau_n^k \mathbbm1_{E'_R}]$ for all $n$. Combining these, we obtain a bound on $\E[ \tau^k \mathbbm1_{E'_R}]$. Finally, we verify that $\P[E'_R]$ decays sufficiently quickly in $R$, and we are done. 
\medskip

\noindent\textit{Step 1: Proving an analog of Proposition~\ref{Prop:EstimateGFFbounded}.} Recall the definition  $P_h^{r,d} = \{ z \in \C:  D_h(z, \partial B_{r/4}(z)) \leq d \}$ in~\eqref{eq:DefProxy}.  The argument of Proposition~\ref{Prop:EstimateGFFbounded} bounded 
\[\E\left[ \int_{(A_{1})^k} \mathbbm1\{ z_1, \dots, z_k \in A_1 \cap P_h^{1, e^{-\xi x}}\} \mu_h(dz_1)\dots \mu_h(dz_k)\right]\] 
by using a Cameron-Martin shift (placing $\gamma$-log singularities at each $z_i$ and replacing $\prod \mu_h(dz_i)$ by $\prod_{i<j} |z_i - z_j|^{-\gamma^2} \prod dz_i$), then using Proposition~\ref{prop-circle-moment} to bound the integral. Recalling Remark \ref{rem-delta}, Proposition~\ref{Prop:EstimateGFFbounded} can be proved even if the exponent $\gamma^2$ is made slightly larger. Any such exponent increase will upper bound the log-singularities of $g$, hence we have the following analog of Proposition~\ref{Prop:EstimateGFFbounded}:
\[\E\left[ \int_{(A_{1})^k} g(z_1, \dots, z_k)\mathbbm1\{ z_1, \dots, z_k \in A_1 \cap P_h^{1, e^{-\xi x}}\} \mu_h(dz_1)\dots \mu_h(dz_k)\right] \lesssim e^{-c_{k,\delta} x}.\] 
\medskip

\noindent\textit{Step 2: Bounding $\E[ \tau_n^k \mathbbm1_{E'_R}]$ for each $n$.} We start with $n = 0$.
Using Lemma~\ref{lem-bound-f} and \eqref{eq-tau-k} (and noting that $\cB_1(0;D_h)\cap A_1 \subset A_1 \cap P_h^{1,1}$), we obtain that $\E[ \tau_0^k \mathbbm1_{E'_R}]$ is bounded from above by
\[
\left( \log R \right)^k \E\left[ \int_{(A_1)^k} g(z_1, \dots, z_k) \mathbbm1\{ z_1, \dots, z_k \in \cB_1(0;D_h)\} \mu_h(dz_1)\dots \mu_h(dz_k)\right] \lesssim \left( \log R \right)^k,
\]
where the last inequality follows from Step 1. Likewise, building off of Step 1, similar arguments as in Lemmas~\ref{Lem:SmallAnnuli} and~\ref{Lem:LargeAnnuli} yield
\[\E \left[\tau_n^k \mathbbm1_{E'_R}\right] \lesssim \left\{
	\begin{array}{ll}
		\left( \log R\right)^k  2^{- \frac{Q^2}{2} |n|} 2^{\alpha_\delta |n|}  & \mbox{if } n < 0 ,\\
		\left( \log R\right)^k 2^{- \frac{Q^2}{2} n}  & \mbox{if } n > 0 .
	\end{array}
\right. \]
for some arbitrarily small $\alpha_\delta >0$.

\medskip

\noindent\textit{Step 3: Bounding the upper tail of $\tau$.} By H\"older's inequality (see end of proof of Lemma~\ref{Lem:SmallAnnuli}), the above bounds on $\E \left[\tau_n^k \mathbbm1_{E'_R}\right]$ yield
\[\E \left[\tau^k \mathbbm1_{E'_R}\right] \lesssim \left(\log R\right)^k. \]
By Lemma~\ref{lem-euc-decay} (see end of section) we also have for some fixed $a > 0$ that
\[ \P[(E'_R)^c] \leq R^{-a}\]
Combining these assertions, we have
\[\P\left[ \tau > t \right] \lesssim \P\left[(E'_R)^c\right] + \E \left[\tau^k \mathbbm1_{E'_R}\right]t^{-k} \lesssim R^{-a} + \left( \log R\right)^kt^{-k}.   \]
Taking $R$ equal to some large power of $t$, we conclude that for all $p < k$ we have $\E[ \tau^p] < \infty$. Taking $k \to \infty$, we obtain Theorem~\ref{thm-LBM}. 
\end{proof}

\begin{proof}[Proof of Lemma~\ref{lem-bound-f}]
We instead prove the stronger statement
\[f(z_1, \dots, z_k, h) \leq C \prod_{i=1}^k \left( \log R  - \log |z_i - z_{i-1}| \right) \quad \text{ for all }z_1, \dots, z_k \in A_1.\]
We split the integral~\eqref{eq-f} into two parts (integrating over $t_k < R^2$ and $t_k \geq R^2$ respectively), and bound each part separately.

There exists $p > 0$ such that the following is true: Let $t \geq 1/k$ and consider a Brownian bridge of duration $t$ with endpoints $B_0, B_t$ specified in $\D$. Then this Brownian bridge stays in $\D$ with probability at most $e^{-pt}$. If $t_k \geq R^2$, then there exists some $i \in \{1, \dots, k\}$ such that $t_i - t_{i-1} \geq t_k/k \geq R^2/k$, and so $B|_{[t_{i-1}, t_i]}$ conditioned on $B_{t_{i-1}} = z_{i-1}$ and $B_{t_i} = z_i$ stays in $R\D$ with probability at most $e^{-pt_k/kR^2}$. This allows us to upper bound the integral~\eqref{eq-f} on the restricted domain with $t_k \geq R^2$:
\begin{align}
&\int_{0 \leq t_1 \leq \dots \leq t_k < \infty} \frac{k! ~~~~ dt_1 \dots dt_k}{(2\pi)^{k/2} \prod_{i=1}^k (t_i - t_{i-1})}\exp \left( -\sum_{i=1}^k \frac{|z_i - z_{i-1}|^2}{2(t_i - t_{i-1})} -\frac p {kR^2} (t_i - t_{i-1})\right)   \nonumber\\
=&\frac{k!}{(2\pi)^{k/2}} \prod_{i=1}^k \int_0^\infty  \frac1t\exp \left( -\frac{|z_i - z_{i-1}|^2}{2t} -\frac p {kR^2} t \right)  dt \nonumber = O\left(\prod_{i=1}^k \left(\log R - \log|z_i - z_{i-1}| \right) \right) \nonumber,
\end{align}
by using the bound $\int_0^{\infty} e^{-t/x} e^{-1/t} \frac{dt}{t} \leq \int_0^1 e^{-1/t} \frac{dt}{t} + \int_x^\infty e^{-t/x} \frac{dt}{t} + \int_1^x \frac{dt}{t} \leq C + \log x$ for $x \geq 1$ and a change of variable.

Now we upper bound the integral~\eqref{eq-f} on the restricted domain $0 \leq t_1 \leq \dots \leq t_k < R^2$:
\begin{align}
&\int_{0 \leq t_1 \leq \dots \leq t_k <R^2} \frac{k!}{(2\pi)^{k/2} \prod_{i=1}^k (t_i - t_{i-1})}\exp \left( -\sum_{i=1}^k \frac{|z_i - z_{i-1}|^2}{2(t_i - t_{i-1})} \right)  dt_1 \dots dt_k \nonumber\\
\leq&\frac{k!}{(2\pi)^{k/2}} \prod_{i=1}^k \int_0^{R^2}  \frac1t\exp \left( -\frac{|z_i - z_{i-1}|^2}{2t}  \right)  dt \nonumber = O\left(\prod_{i=1}^k \left(\log R - \log|z_i - z_{i-1}| \right) \right) \nonumber,
\end{align}
where the final inequality follows from $\int_0^{R^2} e^{-a/2t} \frac{dt}{t} = \int_0^{1} e^{-1/2u} \frac{du}{u} + \int_1^{R^2 a^{-2}} e^{-1/2u} \frac{du}{u} \leq C + \log R^2 a^{-2}$. 
Combining these two upper bounds, we are done. 
\end{proof}
\begin{Lem}[Polynomial tail for Euclidean diameter of $\cB_1(0; D_h)$]\label{lem-euc-decay}
Let $h$ be a whole-plane GFF with $h_1(0) = 0$. Then for all $a \in (0,Q^2/2)$, for all sufficiently large $R$ we have
\[\P\left[ \cB_1(0; D_h) \subset B_R(0)\right] \geq 1 - R^{-a}. \]
\end{Lem}
\begin{proof}
Fix $\eps > 0$ small. By Proposition~\ref{prop-dist} we have with superpolynomially high probability as $R \to \infty$ that 
\[D_h(0, \partial B_R(0)) \geq D_h(\partial B_{R/2}(0), \partial B_R(0)) \geq R^{\xi(Q - \eps)} e^{\xi h_R(0)}.  \]
By a standard Gaussian tail bound we also have
\[\P[h_R(0) > -(Q - \eps) \log R] \leq \exp \left( -\frac{(Q-\eps)^2 \log R}{2} \right) = R^{-(Q-\eps)^2/2}. \]
Combining these two bounds, we see that with probability $1-O(R^{-(Q-\eps)^2/2})$ we have $D_h(0, \partial B_R(0)) > 1$, as desired. 
\end{proof}

\paragraph*{Lower bound for LBM exit time of metric balls}

\begin{thm}
\label{Thm:NegPassage}
Recall that $\tau$ is the first exit time of the Liouville Brownian motion $(X_{t})_{t \geq 0}$ from the LQG metric ball $\mc{B}_1(0;D_h)$. For all $k \geq 1$, we have
\[
\E [\tau^{-k}] < \infty.
\]
\end{thm}
We now sketch the proof. We restrict to a regularity event on which annulus-crossing distances and the quantum time taken to cross an annulus are well approximated by field averages. We can find a collection of annuli separating $0$ from $X_\tau$. Gluing circuit and crossing paths associated to the annuli, we obtain a path from $0$ to $X_{\tau}$. Since the $D_h$-length of these is bounded from above by a circle average approximation, the condition $D_h(0, X_\tau) = 1$ gives a lower bound for a certain sum of (exponentials of) circle averages terms. Raising the exponent by a factor of $d_\gamma$ by Jensen's inequality, we get a lower bound for a circle average approximation of the quantum time spent across these annuli. Thus $\tau$ is unlikely to be very small. 

\smallskip

Consider standard Brownian motion $(B_t)_{t\geq0}$ started at the origin, and recall that Liouville Brownian motion is given by a random time-change: $X_t = B_{F^{-t}(t)}$, where the quantum clock $F$ is formally given by $F(t) = \int_0^t e^{\gamma h(B_s)} ds$ (see \eqref{eq-clock}). Consider an annulus $A_{r/e, r}(z)$ with $0 \not \in A_{r/e, r}(z)$. Define $\tau_r(z)$ to be the quantum passage time of the annulus. That is, for the case where the annulus encircles the origin, writing $t_1$ for the first time $B_t$ hits $\partial B_{r}(z)$, and $t_0$ for the last time before $t_1$ that $B_t$ hits $\partial B_{r/e}(z)$, we set $\tau_r(z) = F(t_1) - F(t_0)$, and define it analogously in the case that the annulus does not encircle the origin. 

We need the following input, which can be seen as a variant of \cite[Proposition 2.12]{GRVbm} combined with the scaling relation \cite[Equation (2.25)]{GRVbm} and which can be obtained by using the same techniques.

\begin{Prop}
\label{Prop:CrossingTimes}
For any compact set $K \subset \C$, there exists a random variable $X \geq 0$ having all negative moments such that the following is true. 
For fixed $r \in (0,1)$ and $z \in K$ such that $0 \not \in A_{r/e,r}(z)$, the quantum passage time  $\tau_r(z)$ is stochastically dominated by $r^{\gamma Q} e^{\gamma h_r(z)} X$.
\end{Prop}
As an immediate consequence of the $r=1$ case of this proposition, we have the following.
\begin{Cor}\label{cor-truncate-disk}
The event $\{X_\tau \not \in \D \text{ and } \tau < C^{-1}\}$ is superpolynomially unlikely as $C \to \infty$. 
\end{Cor}

Similarly to Section \ref{sec-lower-unit}, we set 
\[k_1 = \lfloor N \log C \rfloor.\]

\begin{Lem} \label{lem-holder} There exist $\gamma$-dependent constants $\chi,c > 0$ so that the following holds. Consider the event $E_C$ that each ball $B_{e^{-k_1}}(z)$ included in $2\D$ has quantum diameter at most $2 e^{-\chi k_1}$. Then, $E_C$ occurs with probability at least $1 - e^{-cN}$.
\end{Lem}

\begin{proof}
This is an application of the H\"older estimate \cite[Proposition 3.18]{DFGPS} which implies that there exist positive constants $\chi, \alpha$ such that, as $\eps \to 0$,  with probability at least $1 - \eps^{\alpha}$,
\[
D_h(u,v) \leq |u-v|^{\chi}, \quad \forall u,v \in 2\D \text{ with } |u-v| \leq \eps.
\]
Therefore, taking $\eps = e^{-k_1}$, for $z$ such that $B_{e^{-k_1}}(z) \subset 2\D$, for all $w \in B_{e^{-k_1}}(z)$, $D_h(z,w) \leq e^{-\chi k_1}$ and the quantum diameter of that ball is bounded from above by twice this upper bound.
\end{proof}

We consider the grid $\Z_{C} := \frac{1}{100} e^{-k_1} \Z^2$. 

\begin{Lem}
\label{Lem:GlobalDecoupling}
Consider the event $F_{C}$ that for every point $z  \in \Z_C \cap 2\D$, for all $k \in [0,k_1]$, the following conditions hold. There is a circuit of $D_h$-length at most $e^{-k \xi Q } e^{\xi h_{e^{-k}}(z)} C$ in the annulus $A_{e^{-k-1},e^{-k}}(z)$, the crossing length $D_h(\partial B_{e^{-k-1}}(z), \partial B_{e^{-k+1}}(z))$ is at most $e^{-k \xi Q } e^{\xi h_{e^{-k}}(z)} C$, $\tau_{e^{-k}}(z) \geq e^{-k \gamma Q } e^{\gamma h_{e^{-k}}(z)} C^{-1}$ and, finally, $|h_{e^{-k}}(z)-h_{e^{-k+1}}(z)| \leq \xi^{-1} \log C$. Then, $F_C$ occurs with superpolynomially high probability as $C \to \infty$.

\begin{proof}
This follows from Proposition \ref{Prop:CrossingTimes} and Proposition  \ref{prop-dist} together with a union bound.
\end{proof}
\end{Lem}

\begin{proof}[Proof of Theorem \ref{Thm:NegPassage}]
We will show that $P[\tau > C^{-1}]$ occurs with superpolynomially high probability. By Corollary~\ref{cor-truncate-disk} and Lemmas~\ref{lem-holder} and~\ref{Lem:GlobalDecoupling}, we see that the probability of $\{\tau < C^{-1}\text{ and } X_\tau \not \in \D\} \cup E_C^c \cup F_C^c$ is at most $C^{-cN}$ for some fixed $c$. 

Now restrict to the event $\{ X_\tau \in \D\} \cap E_C \cap F_C$; we show that for some constant $\alpha$ not depending on $C,N$ we have $\tau > C^{-\alpha}$ for sufficiently large $C$, then we are done since $N$ is arbitrary.  On this event the distances $D_h(0,\partial B_{e^{-k_1}}(0))$ and $D_h(X_\tau, \partial B_{e^{-k_1}}(X_\tau))$ are small, so we have $D_h(\partial B_{e^{-k_1}}(0), \partial B_{e^{-k_1}}(X_\tau)) \geq \frac12$. Let $w \in \Z_C$ be the closest point to $X_{\tau}$, and grow the annuli centered at $0$ and $w$ until they first hit; let $k_* \in [0,k_1]$ satisfy $2e^{-k_*}\leq |w| < 2e^{-k_*+1}$. By Lemma \ref{Lem:GlobalDecoupling} we get
\[
\tau \geq \sum_{k \in [k_*, k_1]} \tau_{e^{-k}}(0) +  \tau_{e^{-k}}(w)  \geq C^{-1} \sum_{k \in [k_*, k_1]} e^{-k \gamma Q} e^{\gamma h_{e^{-k}}(0)} +  e^{-k \gamma Q} e^{\gamma h_{e^{-k}}(w)} 
\]
and, by taking an additional annulus crossing and circuit, using the circle average regularity between two annuli,
\[
\frac{1}{2} \leq D_h(\partial B_{e^{-k_1}}(0), \partial B_{e^{-k_1}}(X_\tau))  \leq  10 C^2 \sum_{k \in [k_*, k_1]} e^{-k \xi Q} e^{\xi h_{e^{-k}}(0)} +  e^{-k \xi Q} e^{\xi h_{e^{-k}}(w)}.
\]
Therefore, by raising the inequality above to the power $d_{\gamma}$ and using Jensen's inequality for the right-hand side, as well as the lower bound for $\tau$, we get
\[
\frac{1}{2^{d_{\gamma}}} \leq (10C^2)^{d_{\gamma}} k_1^{d_{\gamma}-1}  \sum_{k \in [k_*, k_1]} e^{-k \gamma Q} e^{\gamma h_{e^{-k}}(0)} +  e^{-k \gamma Q} e^{\gamma h_{e^{-k}}(w)}  \leq (10C^2)^{d_{\gamma}} k_1^{d_{\gamma}-1}   C \tau.
\]
hence $\tau \geq C^{-\alpha}$ for some fixed power $\alpha$ and $C$ large enough. Since $N$ is arbitrary ($\alpha$ does not depend on $N$), we conclude the proof of Theorem~\ref{Thm:NegPassage}.
\end{proof}

\paragraph*{Scaling relations for small balls}

Finally we explain the behavior of small ball exit times. Recall that $\tau_h(z;\mc{B}_s(z;D_h))$ is  the first time that Liouville Brownian motion started at $z$ exits the ball $\cB_s(z; D_h)$. 
\begin{thm}
Let $h$ be a whole-plane GFF normalized so $h_1(0) = 0$, and let $K\subset \C$ be any compact set. For any $\eps \in (0,1)$, there exists a constant $C_{p,\eps,K}$ so that for $C > C_{p,\eps,K}$, for all $s \in (0,1)$ and $z \in K$ we have
\eqb\label{eq-lbm-upper}
\P[\tau_h(z;\mc{B}_s(z;D_h)) \leq C s^{d_\gamma - \eps}] \geq 1 - C^p, 
\eqe
and 
\eqb\label{eq-lbm-lower}
\P[ \tau_h(z;\mc{B}_s(z;D_h)) \geq C^{-1} s^{d_\gamma}] \geq 1 - C^p. 
\eqe
\end{thm}
\begin{proof}
We first discuss the proofs of~\eqref{eq-lbm-upper} and~\eqref{eq-lbm-lower} for the specific case $z = 0$. 
For the $z=0$ upper bound, recall that we proved $\E[\tau_h(0; \cB_1(0; D_h))^k] <\infty$ for all $k>0$ in Theorem~\ref{thm-LBM} by adapting  the proof of Proposition~\ref{Prop:MomentsWholePlane}	. An extension of these arguments like in Lemma~\ref{Lem:UpperSmallBalls} yields $\E[\tau_h(0; \cB_s(0; D_h))^k] \lesssim s^{kd_\gamma - \eps}$ with implicit constant depending only on $k, \eps$, and hence by Markov's inequality, for all $s \in (0,1)$ and  sufficiently large $C$ that
\eqb\label{eq-lbm-upper-0}
\P[\tau_h(0;\mc{B}_s(0;D_h)) \leq C s^{d_\gamma - \eps}] \geq 1 - C^p.
\eqe
For the $z=0$ lower bound, Theorem~\ref{Thm:NegPassage} gives $\E[\tau_h(0; \cB_1(0;D_h))^{-k}]< \infty$ for all $k > 0$, and applying the rescaling argument of Lemma~\ref{lem-lower-all-r} then yields for all $s \in (0,1)$ and sufficiently large $C$ that
\eqb\label{eq-lbm-lower-0}
\P[ \tau_h(0;\mc{B}_s(0;D_h)) \geq C^{-1} s^{d_\gamma}] \geq 1 - C^p. 
\eqe

Finally, the arguments of Proposition~\ref{prop-lower-general} allow us to extend~\eqref{eq-lbm-upper-0} and~\eqref{eq-lbm-lower-0} to~\eqref{eq-lbm-upper} and~\eqref{eq-lbm-lower}.
\end{proof}

\subsection{Recovering the conformal structure from the metric measure space structure of $\gamma$-LQG}

\label{sec-recover}

The Brownian map is constructed as a random metric measure space (see \cite{LeGall13, LeGalltakagi}) and  has been proved to be the Gromov-Hausdorff limit of uniform triangulations and $2p$-angulations in \cite{LeGall07, LeGall10, LeGall13,Miermont13}. The Brownian map was later endowed with a canonical conformal structure (i.e. an embedding into a flat domain, defined up to conformal automorphism of the domain) via identification with $\sqrt{8/3}$-LQG \cite{MS15b,  MS16a, MS16b, MS16} but this construction was non-explicit. The work of \cite{GMS18} gives an explicit way to recover the conformal structure of a Brownian map from its metric measure space structure, and their proof mostly carries over directly to the general setting $\gamma \in (0,2)$, except for certain Brownian map metric ball volume estimates of Le Gall \cite{LeGall10}.  The missing ingredient for general $\gamma$ was exactly the uniform volume estimates \eqref{eq:UniformVolumeEst}(cf. \cite[Lemma 4.9]{GMS18}).

As an immediate consequence of \eqref{eq:UniformVolumeEst}  and the arguments of \cite{GMS18} (see discussion before \cite[Remark 1.3]{GMS18}), we obtain the following generalization of \cite[Theorem 1.1]{GMS18} to all $\gamma \in (0,2)$. Let $h$ be a whole-plane GFF normalized so $h_1(0) = 0$, and write $\cB^\bullet_R(0; D_h)$ for the \emph{filled} $D_h$-ball centered at $0$ with radius $R$ (i.e. the union of $\cB_R(0; D_h)$ and all $\mu_h$-finite complementary regions). Let $\mathcal P^\lambda$ be a sample from the intensity $\lambda$ Poisson point process associated to $\mu_h$. We can obtain a $D_h$-Voronoi tessellation of $\C$ into cells $\{H_{z}^\lambda\}_{z \in P^\lambda}$ by defining $H_z^\lambda = \{ w \in \C \: : \: D_h(w,z) \leq D_h(w,z')\: \forall \: z' \in P^\lambda\}$. We define a graph structure on $\mathcal P^\lambda$ by saying that $z,z' \in P^\lambda$ are adjacent if their Voronoi cells $H^\lambda_z, H^\lambda_{z'}$ intersect along their boundaries, and define $\partial P^\lambda$ to be the vertices corresponding to Voronoi cells intersecting the boundary. Let $Y^\lambda$ be a simple random walk on $\mathcal P^\lambda$ started from the point whose Voronoi cell contains $0$, extend $Y^\lambda$ from the integers to $[0,\infty)$ by interpolating along $D_h$-geodesics, and finally stop $Y^\lambda$ when it hits $\partial P^\lambda$.  
\begin{thm}[Generalization of {\cite[Theorem 1.1]{GMS18}}]\label{thm-RW}
As $\lambda \to \infty$, the conditional law of $Y^\lambda$ given $(\C, 0, D_h, \mu_h)$ converges in probability as $\lambda \to 0$ to standard Brownian motion in $\C$ started at $0$ and stopped when it hits $\partial \cB^\bullet_R(0; D_h)$ (viewed as curves modulo time parametrization).
\end{thm}
Here, the metric on curves modulo time parametrization is given as follows. For curves $\eta_j : [0,T_j] \to \C$ ($j=1,2$), we set 
\[d(\eta_1, \eta_2) = \inf_\phi \sup_{t \in [0,T_1]} |\eta_1(t) - \eta_2(\phi(t))| \]
where the infimum is over increasing homeomorphisms $\phi: [0,T_1] \to [0,T_2]$.
We remark that the convergence in Theorem~\ref{thm-RW} holds uniformly for the random walk and Brownian motion started in a compact set, and moreover holds for a range of quantum surfaces such as quantum spheres, quantum cones, quantum wedges, and quantum disks; see \cite[Theorem 3.3]{GMS18}. Consequently, the Tutte embedding of the Poisson-Voronoi tessellation of the quantum disk converges to the quantum disk as $\lambda \to \infty$ (see the proof of \cite[Theorem 1.2]{GMS18}).
\begin{proof}
Since we have the estimates~\eqref{eq:UniformVolumeEst}, the general $\gamma \in (0,2)$ version of \cite[Theorem 3.3]{GMS18} holds. In particular, Theorem~\ref{thm-RW} holds if we replace the field $h$ with that of a $0$-quantum cone. By comparing $h$ to the field of a 0-quantum cone and using local absolute continuity arguments, we obtain Theorem~\ref{thm-RW}.
\end{proof}

Notice that the construction of $Y^\lambda$ involves only the pointed metric measure space structure of $(\C, 0, D_h, \mu_h)$, so Theorem~\ref{thm-RW} roughly tells us that we can recover the conformal structure of $(\C, 0, D_h, \mu_h)$ from its metric measure space structure. The following variant of \cite[Theorem 1.2]{GMS18} makes this observation explicit, resolving a question of \cite{GM19uniqueness}.

\begin{thm}[Pointed metric measure space $(\C, 0, D_h, \mu_h)$ determines conformal structure]\label{thm-conformal}
Let $h$ be a whole-plane GFF normalized so $h_1(0) = 0$. Almost surely, given the pointed metric measure space $(\C, 0, D_h, \mu_h)$, we can recover its conformal embedding into $\C$ and hence recover $h$ (both modulo rotation and scaling).
\end{thm}
\begin{proof}
To simplify the notation, suppose the two-pointed metric measure space $(\C, 0, 1, D_h, \mu_h)$ is given, then we show we can recover exactly the embedding of $\mu_h$ in $\C$ (otherwise, one can arbitrarily pick any other point from the pointed metric measure space and use that in place of 1, and only recover the embedded measure modulo rotation and scaling). Since $\mu_h$ (with its embedding in $\C$) determines $h$ \cite{BSS14} and hence $D_h$, it suffices to recover $\mu_h$.

Consider $R$ large so $1 \in \cB^\bullet_R(0; D_h)$. In the same way that \cite[Theorem 1.1]{GMS18} is used to prove \cite[Theorem 1.2]{GMS18},  we can use Theorem~\ref{thm-RW} to obtain an embedding of the two-pointed metric measure space $(\cB^\bullet_R(0; D_h), 0, 1 , D_h, \mu_h)$ into the unit disk $\D$ with the correct conformal structure and sending $0$ to $0$ and $1$ to a point in $(0,1)$. 

This is done by taking a $\lambda$-intensity Poisson-Voronoi tessellation of $(\cB^\bullet_R(0; D_h), 0,1,D_h, \mu_h)$, and embedding its adjacency graph $P^\lambda$ in $\D$ via the \emph{Tutte embedding} $\Phi^\lambda$: let $x_0, \dots, x_n$ be the vertices in $\partial P^\lambda$ in counterclockwise order with $x_0$ arbitrarily chosen, and let $z_0$ (resp. $z_1$) be the vertex corresponding to the Poisson-Voronoi cell containing 0 (resp. 1). Define the map $\wt \Phi^\lambda:P^\lambda \to \ol \D$ via $\wt \Phi^\lambda(z_0) = 0, \wt \Phi (x_0) = 1$ and $\wt \Phi^\lambda(x_j) = e^{2\pi i p_j}$ where $p_j$ is the probability that $Y^\lambda$ hits $\partial P^\lambda$ at one of the points $x_0, \dots, x_j$, and extend $\wt \Phi$ to the rest of $P^\lambda$ so it is discrete harmonic. Finally, define $\Phi^\lambda(z) = e^{i \theta} \wt \Phi^\lambda(z)$ where $\theta \in [0,2\pi)$ is chosen so $\Phi(z_1) \in \R$. 
 Taking $\lambda \to \infty$, the $\Phi^\lambda$-pushforward of the counting measure on the vertices of the embedded graph normalized by $\lambda^{-1}$ converges weakly in probability to the desired conformally embedded measure. See \cite[Section 3.3]{GMS18} for details.

Rescale this embedding (and forget the metric) to obtain an equivalent two-pointed measure space $(c_R \D, 0, 1 , \mu^R)$ with the LQG measure and conformal structure. That is, there exists a conformal map $\varphi^R : \cB^\bullet_R(0; D_h) \to c_R\D$ such that $\varphi^R(0) = 0, \varphi^R(1) = 1$, and the pushforward $(\varphi^R)^* \mu_h$ equals $\mu^R$. We emphasize that since we are only given $(\C, 0, 1, D_h, \mu_h)$ as a two-pointed metric measure space, we know neither the embedding $\cB^\bullet_R(0; D_h)\subset \C$ nor the conformal map $\varphi^R$, but we do know $c_R$ and $\mu^R$.

Now, by a simple estimate on the distortion of conformal maps \cite[Lemma 2.4]{sphere} (stated for the cylinder $\R \times [0,2\pi]$ but applicable to our setting via the map $z \mapsto e^{-z}$), we see that for any compact $K\subset \C$ we have $\lim_{R \to \infty} \sup_{z \in K} |\varphi^R(z)-z| = 0$ and $\lim_{R \to \infty} \sup_{z \in K} |(\varphi^R)^{-1}(z)-z| = 0$. Thus, for any fixed rectangle $A$, the measure of the symmetric difference $\mu_h\left(A \triangle (\varphi^R)^{-1}(A)\right)$ converges to zero as $R \to \infty$; this implies $\lim_{R \to \infty} |\mu^R(A) - \mu_h(A)| = 0$. Since $\mu^R$ is a function of the two-pointed metric measure space $(\C, 0, 1, D_h, \mu_h)$, we conclude that $\mu_h(A)$ is also. Therefore the two-pointed metric measure space $(\C, 0, 1, D_h, \mu_h)$ determines $\mu_h$ and hence $h$. 
\end{proof}

\begin{appendix}

\section{Proof of the inductive relation for small moments}

\begin{Lem}
\label{Lem:Preliminaries}
Recall $v_k(r)$ and $u_k(r)$ from~\eqref{eq:uv}. The following relation holds.
\begin{equation}
v_k(r) \leq C r^{-2} \sum_{i=1}^{k-1} {k\choose i}  (4k)^{\gamma^2 i (k-i)} r^{-\gamma^2 i (k-i)}   u_i(4r) u_{k-i}(4r).
\end{equation}
\end{Lem}

\begin{proof}
Set $f_k(z_1,\dots z_k) := \prod_{i<j} |z_i-z_j|^{-\gamma^2}$. Note that when $\max_{i<j} |z_i - z_j| \leq r$, the $k$ points are included in $B(z_1,r)$ which itself is included in a ball of radius $4r$ centered at at point of $r \Z^2 \cap \D$. Since $f_k$ is a function of the pairwise distance, which is translation invariant, we get
\begin{align*}
v_k(r) & = \int_{\D^k} \frac{1_{r/2 \leq \max_{i < j} |z_i-z_j| \leq r} }{\prod_{i < j} |z_i-z_j|^{\gamma^2}} dz_1 \dots dz_k \\
 & \leq C r^{-2} \int_{4 r \D^k} 1_{ r/2 \leq\max_{i < j} |z_i-z_j|} f_k(z_1,\dots, z_k) dz_1 \dots dz_k
\end{align*}
Then, take two points at distance $r/2$ in $4r \D$, say $z$ and $w$ among $\lbrace z_1, \dots, z_k \rbrace$. We cut $k+1$ orthogonal sections of same width to the segment $[z,w]$. At least one should be empty and this separates two clusters of points, $I = \lbrace z_{p_1}, \dots , z_{p_i} \rbrace$ and $J = \lbrace z_{q_1}, \dots , z_{q_{k-i}} \rbrace$ for some $1 \leq i \leq k-1$. All points between the two clusters $I$ and $J$ are separated by $|z-w|/(k+1) \geq r/4k$.  We decouple $f_k(z_1, \dots, z_k)$ for two clusters $I$ and $J$ of size $i$ and $k-i$ by $f_k(z_1, \dots, z_k) \leq  (4k)^{\gamma^2 i (k-i)} r^{-\gamma^2 i (k-i)} f_i(I) f_{k-i}(J)$. In particular, splitting over the possibles cases we get
\[
v_k(r)\le C r^{-2} \sum_{i=1}^{k-1}  \sum_{I} (4k)^{\gamma^2 i (k-i)} r^{-\gamma^2 i (k-i)}   \int_{4 r \D^k} f_i(I) f_{k-i}(J) dz_1 \dots dz_k,
\]
where for each $i$, $I$  ranges over all subsets of $\{z_1\dots, z_k \}$ with $i$ elements. 
This gives 
\[
v_k(r)\le C r^{-2} \sum_{i=1}^{k-1} {k\choose i}  (4k)^{\gamma^2 i (k-i)} r^{-\gamma^2 i (k-i)}   u_i(4r) u_{k-i}(4r).
\]
and completes the proof.
\end{proof}

\section{Whole-plane GFF and $\star$-scale invariant field}

\label{Sec:Whole-plane-Star-scale}

In this section we recall some properties of $\star$-scale invariant fields and explain that the whole-plane GFF modulo constants can be seen as a $\star$-scale invariant field. 

We will denote by $\mc{S}(\C)$ the space of space of Schwartz functions and by $L^{2}(\C)$ the space of square integrable functions, on $\C$.  For $f,g \in L^2(\C)$, let $\langle f, g \rangle$ stands for the $L^2(\C)$ inner product. Furthermore, $*$ denotes the convolution operator.

\paragraph*{$\star$-scale invariant field $\phi$} We introduce here the field $\phi  = \sum_{k \geq 1} \phi_k$ we work with in Section \ref{Sec:InductiveStarScale}. The notation and definition are close to the one in \cite[Section 2.1]{DF18} and we refer the reader to this Section for more details.

Consider $k$, a smooth, radially symmetric and nonnegative bump function supported in $B_{1/(2e)}(0)$, such that $k$ is normalized in $L^2(\C)$. We set $c = k * k$ which has therefore compact support included in $B_{1/e}(0)$ and satisfies $c(0) =1$. We consider a space-time white noise $\xi(dx,dt)$ on $\C \times [0,\infty)$ and define the random Schwartz distribution
\[
\phi(x) := \int_0^1 \int_{\C} k \left( \frac{x-y}{t} \right) t^{-3/2} \xi(dy,dt).
\]
The covariance kernel of $\phi$ is given by $\E ( \phi(x) \phi(x')) = \int_0^1 c ( \frac{x-x'}{t} ) \frac{dt}{t}$. We decompose $\phi = \sum_{k \geq 1} \phi_k$ where $\phi_k(x) := \int_{e^{-k}}^{e^{-(k-1)}} \int_{\C} k \left( \frac{x-y}{t} \right) t^{-3/2} \xi(dy,dt)$ and whose covariance kernel is given by $C_k(x,x'):=\int_{e^{-k}}^{e^{-(k-1)}} c ( \frac{x-x'}{t} ) \frac{dt}{t}$. Note that $C_k(x,x') = C_1(e^{(k-1)} x, e^{(k-1)} x')$ and that if $|x-x'| \geq e^{-1}$, $C_1(x,x') = 0$ hence $\phi_k$ has finite range dependence with range of dependence $e^{-k}$. Note also that the pointwise variance of $\phi_{0,n} := \sum_{1 \leq k \leq n} \phi_k$ is equal to $n$.

\begin{Lem}
\label{Lem:GradientEst} There exists $C,c >0$ such that for all $k \geq 0$, $x > 0$, $\Pro (e^{-k} \norme{ \nabla \phi_{0,k}}_{e^{-k} S} \geq x) \leq C e^{-c x^2}$, where $S$ denotes the unit square $[0,1] \times [0,1]$.
\end{Lem}
\begin{proof}
This is essentially the argument as in the proof of Lemma 10.1 in \cite{DF18} which we recall. By Fernique's theorem, $\Pro (\norme{ \nabla \phi_1}_{S} \geq x) \leq C e^{-c x^2}$. Therefore, by scaling, $\Pro (e^{-\ell} \norme{ \nabla \phi_\ell}_{ e^{-\ell }S} \geq x) \leq C e^{-c x^2}$ for $\ell \geq 1$. By setting $X_\ell := e^{-\ell } \norme{\nabla \phi_\ell }_{e^{-\ell } S}$, by the triangle inequality and since $e^{-k} S \subset e^{-\ell } S$ for $\ell  \leq k$, $e^{-k} \norme{ \nabla \phi_{0,k}}_{e^{-k} S} \leq \sum_{0\leq \ell  \leq k} e^{-(k - \ell )} X_\ell $. By inspecting the Laplace functional, and using that the $X_\ell $'s are independent and identically distributed, we conclude the proof of the Lemma.
\end{proof}

\paragraph*{Whole-plane GFF} We explain here why $\int_0^{\infty} k (\frac{x-y}{t}) t^{-3/2} \xi(dy,dt)$ is a whole-plane GFF modulo constants. Set $\phi_{\eps}(x) = \int_{\eps}^{\eps^{-1}} \int_{\C} k \left( \frac{x-y}{t} \right)t^{-3/2} \xi(dy,dt)$ and take $f \in \mc{S}(\C)$ such that $\int_{\C} f dx = 0$. Writing $C_{\eps}(x) := \int_{\eps}^{\eps^{-1}} c\left( \frac{x}{t} \right) \frac{dt}{t} = \int_{\eps}^{\eps^{-1}} c_t(x) \frac{dt}{t}$ with $c_t(\cdot) = c(\cdot /t)$, we have
\[
\E \left( \langle \phi_{\eps} , f \rangle^2 \right) = \int_{\C \times \C} f(x) C_{\eps}(x-y) f(y) dx dy  =  \frac{1}{(2\pi)^2} \int_{\R^2} \hat{C}_{\eps}(\xi) | \hat{f}(\xi)|^2 d \xi
\]
where our convention for the Fourier transform is $\hat g (\xi) := \int_{\C} g(x) e^{- i \xi \cdot x}$.

We compute the Fourier transform $\hat{C}_{\eps}(\xi) = \int_{\eps}^{\eps^{-1}} \hat{c}_t(\xi) \frac{dt}{t} = \int_{\eps}^{\eps^{-1}} t \hat{c}(t \xi) dt $ and since $c = k * k $, $\hat{c} = \hat{k}^2$, then $ \hat{C}_{\eps}(\xi) = \int_{\eps}^{\eps^{-1}} t \hat{k}(t \xi)^2 dt = \norme{\xi}^{-2} \int_{\eps \norme{\xi}}^{\eps^{-1} \norme{\xi}} u \hat{k}(u)^2 du$. By monotone convergence, we get
\begin{align*}
\E \left( \langle \phi_{\eps} , f \rangle^2 \right) & =  \frac{1}{(2\pi)^2} \int_{\R^2} \norme{\xi}^{-2} \int_{\eps \norme{\xi}}^{\eps^{-1} \norme{\xi}} u \hat{k}(u)^2 du | \hat{f}(\xi)|^2 d \xi \\
& \underset{\eps \to 0}{\rightarrow} \left( \int_{0}^{\infty} u \hat k(u)^2 du \right) \times \frac{1}{(2\pi)^2} \int_{\R^2} \norme{\xi}^{-2}  | \hat{f}(\xi)|^2 d\xi.
\end{align*}
Since $\hat{k}$ is radially symmetric and $k$ is normalized in $L^2$, by Plancherel's theorem $\int_0^{\infty} u \hat{k}(u)^2 du =2\pi$. Furthermore, by setting $g(x) = \int_{\C} \log | x-y| f(y) dy$ we get $\Delta g = 2\pi f$ and in Fourier modes, $- \norme{\xi}^2 \hat{g}(\xi) = 2 \pi \hat{f}(\xi)$ hence, by Plancherel's theorem,
\begin{align*}
\int_{\C^2} f(x) (- \log |x-y|) f(y) dx dy & = - \int_{\C} f(x) g(x) dx = \frac{-1}{(2\pi)^2} \int_{\R^2} \hat{f}(\xi) \hat{g}(\xi) d\xi \\
& = \frac{1}{2\pi} \int_{\R^2} \norme{\xi}^{-2} |\hat{f}(\xi)|^2 d\xi.
\end{align*}
Note that this term is finite because under the assumption $\int_{\C} f dx = 0$, we have $\hat{f}(0) = 0$ so the above singularity at the origin is compensated by the first term in the development of $\hat{f}$. Altogether, we get 
\[
\E \left( \langle \phi_{\eps} , f \rangle^2 \right)  \underset{\eps \to 0}{\rightarrow} \int_{\C^2} f(x) (-\log |x-y|) f(y) dx dy
\]
Hence the convergence of the characteristic functionals: $\E (e^{i \langle \phi_{\eps} , f \rangle}) = e^{- \frac{1}{2} \E \left( \langle \phi_{\eps} , f \rangle^2 \right)} \underset{\eps \to 0}{ \rightarrow} e^{- \frac{1}{2} \E \left( \langle h , f \rangle^2 \right)}$.

The following lemma will be useful when working with the whole plane GFF \emph{not} modulo additive constant.
\begin{Lem}
\label{Lem:CouplingFields}
There exists a coupling of the whole-plane GFF $h$ normalized such that $h_1(0)=0$ and the $\star$-scale invariant field $\phi$ such that the difference $h-\phi$ is a continuous field.
\end{Lem}

\begin{proof}
Recall the notation $\phi_{k,\ell} = \int_{e^{-\ell}}^{e^{-k}} k (\frac{x-y}{t}) t^{-3/2} \xi(dy,dt)$. We know $\phi_{-\infty,\infty}$ is a whole-plane GFF modulo constant. The fine field $\phi = \phi_{0,\infty}$ is a well-defined Schwartz distribution. Also, the gradient field $\nabla \phi_{-\infty,0}$ is a well-defined continuous Gaussian vector (this can be checked by inspecting the covariance kernel and applying the Kolmogorov continuity theorem). Thus, $\phi_{-\infty, 0}$ is well defined modulo additive constant, so $\phi_{L}:=``\phi_{-\infty,0} - \dashint_{\partial B_1(0)} \phi_{-\infty,0}$'' is a well-defined continuous Gaussian field, independent of $\phi$. By setting $g := \phi_L - \dashint_{\partial B_1(0)} \phi$, we get that $h:= \phi + g$ is a whole-plane GFF normalized such that $h_1(0) = 0$.
\end{proof}

\section{Volume of small balls in the Brownian map}\label{sec-brownian-map}
We do not use any material in this section in our proofs, but include it to facilitate an easier comparison between our argument in Section~\ref{section-positive} and the analogous result for the Brownian map  case.
Le Gall obtained the following uniform estimate on the volume of small balls in the Brownian map. For $\beta \in (0,1)$, there exists a random $K_{\beta} > 0$ such that for every $r>0$, the volume of any ball of radius $r$ in the Brownian map is bounded from above by $K_\beta r^{4-\beta}$.  Our proof of the finiteness of LQG ball volume positive moments (Section~\ref{section-positive}) shares some similarities with his only at a very high level; no explicit formulas are available in our framework, and the techniques are very different. We discuss some of the arguments used in the Brownian map setting and we refer the reader to \cite{LeGallSnakes, LeGallW, LeGall07, LeGall10} for details. This estimate was used in the proofs of the uniqueness of the Brownian map \cite{Miermont13, LeGall13}. 

\paragraph*{Tree of Brownian paths} A binary marked tree is a pair $\theta = (\tau, (h_v)_{v \in \tau})$ where $\tau$ is a binary plane tree and where for $v \in \tau$, $h_v$ is the length of the branch associated to $v$. We denote by $\Lambda_k(d \theta)$ the uniform measure on the set of binary marked trees with $k$ leaves (uniform measure over binary plane trees and Lebesgue measures for the length of the branches). $I(\theta)$ and $L(\theta)$ will denote respectively the internal nodes and leaves of $\theta$. One can define a Brownian motion on such a tree: the process is a standard Brownian motion over a branch, and after an intersection, the two processes evolve independently conditioning on the value at the node. We will denote by $P_x^{\theta}$ this process, started from the root of the tree with initial value $x$. Similarly, instead of using a Brownian motion, one can consider a 9-dimensional Bessel process and we will denote it by $Q_x^{\theta}$.

Similarly, for trees given by a contour function $(h(s))_{s \leq \sigma}$ with lifetime $\sigma$, one can associate the so-called Brownian snake given by the process $(W_s)_{s \leq \sigma}$ of Brownian type path (for each $s$, $W_s$ is a Brownian type path with lifetime $h(s)$, its last value is denoted by $\widehat{W}_s$ and corresponds to the Brownian label above the point of the tree corresponding to $s$). We can add another level of randomness by taking $h$ given by a Brownian type excursion:  $\mb{N}_0$ is the measure associated to the unconditioned lifetime It\^o excursion, $\ol{\mb{N}}_0$  is also associated to the unconditioned lifetime Ito excursion but the Brownian labels are conditioned to stay positive. 

\paragraph*{Explicit formulas}

The following explicit formula (see \cite{LeGallSnakes}, Proposition IV.2), relates the objects of the previous paragraph.  For $p \geq 1$, $x \in \R$  and $F$ a symmetric nonnegative measurable function on $W^p$, where $W$ denotes the space of finite continuous paths,
\begin{equation}
\label{eq:ExplicitBrownian}
\mb{N}_x \left[ \int_{(0,\sigma)^p} F(W_{s_1}, \dots, W_{s_p}) ds_1 \dots ds_p \right] = 2^{p-1} p!  \int  \Lambda_p (d\theta) P_x^{\theta} \left[ F((w^{(a)})_{a \in L(\theta)}) \right].
\end{equation}
Here, $w$ is the tree-indexed Brownian motion with law $P^\theta_x$ and $w^{(a)}$ the restriction of $w$ to the path joining $a$ to the root, and $\mb N_x$ is the measure $\mb{N}_0$ where each Brownian snake has its labels incremented by $x$. 
This formula involves combining the branching structure of certain discrete trees with spatial displacements. It relies on nice Markovian properties, in particular on specific properties of the It\^o measure. The proof of the uniform volume bound for metric ball is based on an explicit formula obtained in \cite{LeGallW} for the finite-dimensional marginal distributions of the Brownian tree under $\ol{\mb{N}}_0$,
\begin{multline}
\label{eq:ExplicitBessel}
\ol{\mb{N}}_0 \left[\int_{(0,\sigma)^p} F(W_{s_1}, \dots, W_{s_p}) ds_1 \dots ds_p\right]  \\
= 2^{p-1} p!  \int  \Lambda_p (d\theta) Q_0^{\theta} \left[ F((\ol{w}^{(a)})_{a \in L(\theta)} \prod_{b \in I(\theta) } \ol{V}_b^4  \prod_{c \in L(\theta)} \ol{V}_c^{-4} \right].
\end{multline}
Here, we write $\ol w$ and $\ol w^{(a)}$ for the nine-dimensional Bessel process counterparts of $w$ and $w^{(a)}$, and $\ol V_v$ for the value of the Bessel process at the vertex $v$. 
Because of the conditioning of $\ol{\mathbb{N}}_0$, the spatial displacements are given by nine-dimensional Bessel processes rather than linear Brownian motions. To derive such a formula, in \cite{LeGallW} the authors generalize \eqref{eq:ExplicitBrownian} to functionals including the range of labels and lifetime $\sigma$ and then use results on absolute continuity relations between Bessel processes, which are consequences of the Girsanov theorem (note that integrals over time of Brownian motions are integral over branches of trees of Brownian motion).

\paragraph*{Positive moment estimates}  In the proof of the upper bound on small ball volumes of the Brownian map in \cite{LeGall10}, a key estimate is to show that, for $k \geq 1$, $c_k < \infty$ where
\begin{multline}
\label{eq:VolumeBMball}
c_k := \ol{\mb{N}}_0 \left[ \left(\int_0^{\sigma} 1_{\{ \widehat{W}_s \leq 1 \}} ds \right)^k\right]   \\ 
= 2^{k-1} k!  
\int Q_0^{\theta} \left[ ( \prod_{a \in I(\theta) } \ol{V}_a^4 )  ( \prod_{b \in L(\theta)} \ol{V}_b^{-4} 1_{\ol{V}_b \leq 1})\right] \Lambda_k(d\theta) =: 2^{k-1} k! \tilde{d}_k.
\end{multline}
Note that the second inequality is obtained by using \eqref{eq:ExplicitBessel} with $F(W_{s_1}, \dots, W_{s_k}) = 1_{\widehat{W}_{s_1} \leq 1}, \dots, 1_{\widehat{W}_{s_k} \leq 1}$. The proof works by induction, introducing an additional parameter to take care of the value of the label at the splitting node in the branching structure,  by setting
\[
\tilde{d}_k(r) := \int Q_r^{\theta} \left[\left( \prod_{a \in I(\theta) } \ol{V}_a^4 \right)  \left( \prod_{b \in L(\theta)} \ol{V}_b^{-4} 1_{\ol{V}_b \leq 1}\right) \right] \Lambda_k(d\theta).
\]
In this framework, the base case and inductive relation are quite straightforward because of the exact underlying branching structure. 
Let $R$ denote a 9-dimensional Bessel process that starts from $r$. 
The base case corresponds to 
\begin{equation}
\label{eq:BWbasecase}
\tilde{d}_1(r) = \E\left[ \int_0^{\infty} R_t^{-4} 1_{\{ R_t \leq 1 \}} dt \right] = c \int_{\R^9} | r - z |^{-7} |z |^{-4} 1_{\{ |z| \leq 1 \}} dz
\end{equation}
and the inductive relation states
\begin{equation}
\label{eq:BWinductive}
\tilde{d}_{\ell}(r) = \E \left[ \int_0^{\infty} R_t^4 \left( \sum_{j=1}^{\ell-1} \tilde{d}_j(R_t)  \tilde{d}_{\ell -j}(R_t)  \right) \right].
\end{equation}
Now, one can easily derive the bounds $\tilde{d}_1(r) \leq M r^{-2} \wedge r^{-7}$ and for $j \geq 2$ $\tilde{d}_j(r) \leq M_j 1 \wedge r^{-7}$. We underline that the exact branching structure of the framework is expressed through the equality \eqref{eq:BWinductive}.

\paragraph*{Comparison}

Let us compare our proof of the finiteness of positive moments with the one in the Brownian map setting. In our setup, no nice branching structure for distances is known. Furthermore, by working with a given embedding or a restriction to a specific domain, we have to carry in the analysis information about the Euclidean domain, including an additional layer of difficulty.  

In the case of the Brownian map, when one considers the ``volume'' associated with \eqref{eq:VolumeBMball} thanks to the explicit formulas \eqref{eq:ExplicitBrownian} and \eqref{eq:ExplicitBessel}, one ends up with branching Bessel processes on uniform trees. In our framework, analogous observables of ``distances'' are not well understood so far.  Instead, circle averages processes are tractable. They evolve as correlated Brownian motions. These are a good proxy for the metric because of the superconcentration of side-to-side crossing distances. Furthermore, when one weights the distribution with singularities (after a Cameron-Martin argument), these Brownian motions are shifted by drifts. (Note that the passage from \eqref{eq:ExplicitBrownian} to \eqref{eq:ExplicitBessel} uses Girsanov.)

Similarities can be seen as the level of induction where the value of the Bessel process at the first node can be compared with the value of the circle average of the field in at the first branching as well in our hierarchical decomposition. Therefore, Lemma \ref{Lem:InductiveRelation}  is similar to~\eqref{eq:BWinductive} and Proposition \ref{prop-circle-moment}  to~\eqref{eq:VolumeBMball}.

\end{appendix}

\bibliographystyle{abbrv}
\bibliography{biblio}

\begin{thebibliography}{10}

\bibitem{ARV13}
R.~Allez, R.~Rhodes, and V.~Vargas.
\newblock Lognormal {$\star$}-scale invariant random measures.
\newblock {\em Probab. Theory Related Fields}, 155(3-4):751--788, 2013.

\bibitem{Aru17}
J.~{Aru}.
\newblock {Gaussian multiplicative chaos through the lens of the 2D Gaussian
  free field}.
\newblock {\em arXiv:1709.04355}, 2017.

\bibitem{B15}
N.~Berestycki.
\newblock Diffusion in planar {L}iouville quantum gravity.
\newblock {\em Ann. Inst. Henri Poincar\'{e} Probab. Stat.}, 51(3):947--964,
  2015.

\bibitem{BerestyckiGffLqg}
N.~Berestycki.
\newblock Introduction to the {G}aussian {F}ree {F}ield and {L}iouville
  {Q}uantum {G}ravity.
\newblock {\em Lecture notes available on the webpage of the author}, 2016.

\bibitem{Berestycki17}
N.~Berestycki.
\newblock An elementary approach to {G}aussian multiplicative chaos.
\newblock {\em Electron. Commun. Probab.}, 22:Paper No. 27, 12, 2017.

\bibitem{BSS14}
N.~Berestycki, S.~Sheffield, and X.~Sun.
\newblock Equivalence of {L}iouville measure and {G}aussian free field.
\newblock {\em 1410.5407}, 2014.

\bibitem{DDDF19}
J.~Ding, J.~Dub{\'e}dat, A.~Dunlap, and H.~Falconet.
\newblock Tightness of {L}iouville first passage percolation for $\gamma \in
  (0,2)$.
\newblock {\em Publ. Math. Inst. Hautes \'{E}tudes Sci.}, 132:353--403, 2020.

\bibitem{DD16}
J.~Ding and A.~Dunlap.
\newblock Liouville first-passage percolation: subsequential scaling limits at
  high temperature.
\newblock {\em Ann. Probab.}, 47(2):690--742, 2019.

\bibitem{DD18}
J.~Ding and A.~Dunlap.
\newblock Subsequential scaling limits for {L}iouville graph distance.
\newblock {\em Comm. Math. Phys.}, 376(2):1499--1572, 2020.

\bibitem{DG18}
J.~Ding and E.~Gwynne.
\newblock The fractal dimension of {L}iouville quantum gravity: universality,
  monotonicity, and bounds.
\newblock {\em Comm. Math. Phys.}, 374(3):1877--1934, 2020.

\bibitem{ding2019heat}
J.~Ding, O.~Zeitouni, and F.~Zhang.
\newblock Heat kernel for {L}iouville {B}rownian motion and {L}iouville graph
  distance.
\newblock {\em Comm. Math. Phys.}, 371(2):561--618, 2019.

\bibitem{DF18}
J.~Dub\'{e}dat and H.~Falconet.
\newblock Liouville metric of star-scale invariant fields: tails and {W}eyl
  scaling.
\newblock {\em Probab. Theory Related Fields}, 176(1-2):293--352, 2020.

\bibitem{DFGPS}
J.~Dub\'{e}dat, H.~Falconet, E.~Gwynne, J.~Pfeffer, and X.~Sun.
\newblock Weak {LQG} metrics and {L}iouville first passage percolation.
\newblock {\em Probab. Theory Related Fields}, 178(1-2):369--436, 2020.

\bibitem{DS11}
B.~Duplantier and S.~Sheffield.
\newblock Liouville quantum gravity and {KPZ}.
\newblock {\em Invent. Math.}, 185(2):333--393, 2011.

\bibitem{GHSS18}
C.~Garban, N.~Holden, A.~Sep\'{u}lveda, and X.~Sun.
\newblock Negative moments for {G}aussian multiplicative chaos on fractal sets.
\newblock {\em Electron. Commun. Probab.}, 23:Paper No. 100, 10, 2018.

\bibitem{GRVbm}
C.~Garban, R.~Rhodes, and V.~Vargas.
\newblock Liouville {B}rownian motion.
\newblock {\em Ann. Probab.}, 44(4):3076--3110, 2016.

\bibitem{Notice-Gwynne}
E.~{Gwynne}.
\newblock {Random surfaces and Liouville quantum gravity}.
\newblock {\em To appear in Not. Am. Math. Soc. arXiv:1908.05573}, 2019.

\bibitem{GHS16}
E.~Gwynne, N.~Holden, and X.~Sun.
\newblock A distance exponent for {L}iouville quantum gravity.
\newblock {\em Probab. Theory Related Fields}, 173(3-4):931--997, 2019.

\bibitem{MOT-Survey}
E.~{Gwynne}, N.~{Holden}, and X.~{Sun}.
\newblock {Mating of trees for random planar maps and Liouville quantum
  gravity: a survey}.
\newblock {\em arXiv:1910.04713}, 2019.

\bibitem{GM19confluence}
E.~Gwynne and J.~Miller.
\newblock Confluence of geodesics in {L}iouville quantum gravity for $\gamma
  \in (0,2)$.
\newblock {\em To appear in Ann. Probab. arXiv:1905.00381}, 2019.

\bibitem{GM19uniqueness}
E.~{Gwynne} and J.~{Miller}.
\newblock Existence and uniqueness of the {L}iouville quantum gravity metric
  for $\gamma \in (0,2)$.
\newblock {\em arXiv:1905.00383}, 2019.

\bibitem{GM19local}
E.~Gwynne and J.~Miller.
\newblock Local metrics of the {G}aussian free field.
\newblock {\em To appear in Annales de l'Institut Fourier. arXiv:1905.00379},
  2019.

\bibitem{harmonic-CRT}
E.~Gwynne, J.~Miller, and S.~Sheffield.
\newblock Harmonic functions on mated-{CRT} maps.
\newblock {\em Electron. J. Probab.}, 24:Paper No. 58, 55, 2019.

\bibitem{GMS18}
E.~Gwynne, J.~Miller, and S.~Sheffield.
\newblock The {T}utte embedding of the {P}oisson-{V}oronoi tessellation of the
  {B}rownian disk converges to {$\sqrt{8/3}$}-{L}iouville quantum gravity.
\newblock {\em Comm. Math. Phys.}, 374(2):735--784, 2020.

\bibitem{gwynne2019kpz}
E.~Gwynne and J.~Pfeffer.
\newblock {KPZ} formulas for the {L}iouville quantum gravity metric.
\newblock {\em To appear in Trans. Amer. Math. Soc. arXiv:1905.11790}, 2019.

\bibitem{JSW18}
J.~Junnila, E.~Saksman, and C.~Webb.
\newblock Decompositions of log-correlated fields with applications.
\newblock {\em Ann. Appl. Probab.}, 29(6):3786--3820, 2019.

\bibitem{Kahane85}
J.-P. Kahane.
\newblock Sur le chaos multiplicatif.
\newblock {\em Ann. Sci. Math. Qu\'{e}bec}, 9(2):105--150, 1985.

\bibitem{LRV-Mabuchi}
H.~{Lacoin}, R.~{Rhodes}, and V.~{Vargas}.
\newblock {Path integral for quantum Mabuchi K-energy}.
\newblock {\em arXiv:1807.01758}, 2018.

\bibitem{LeGallSnakes}
J.-F. Le~Gall.
\newblock {\em Spatial branching processes, random snakes and partial
  differential equations}.
\newblock Lectures in Mathematics ETH Z\"{u}rich. Birkh\"{a}user Verlag, Basel,
  1999.

\bibitem{LeGall07}
J.-F. Le~Gall.
\newblock The topological structure of scaling limits of large planar maps.
\newblock {\em Invent. Math.}, 169(3):621--670, 2007.

\bibitem{LeGall10}
J.-F. Le~Gall.
\newblock Geodesics in large planar maps and in the {B}rownian map.
\newblock {\em Acta Math.}, 205(2):287--360, 2010.

\bibitem{LeGall13}
J.-F. Le~Gall.
\newblock Uniqueness and universality of the {B}rownian map.
\newblock {\em Ann. Probab.}, 41(4):2880--2960, 2013.

\bibitem{LeGalltakagi}
J.-F. Le~Gall.
\newblock Brownian geometry.
\newblock {\em Jpn. J. Math.}, 14(2):135--174, 2019.

\bibitem{LeGallW}
J.-F. Le~Gall and M.~Weill.
\newblock Conditioned {B}rownian trees.
\newblock {\em Ann. Inst. H. Poincar\'{e} Probab. Statist.}, 42(4):455--489,
  2006.

\bibitem{Miermont13}
G.~Miermont.
\newblock The {B}rownian map is the scaling limit of uniform random plane
  quadrangulations.
\newblock {\em Acta Math.}, 210(2):319--401, 2013.

\bibitem{MQ18}
J.~Miller and W.~Qian.
\newblock The geodesics in {L}iouville quantum gravity are not
  {S}chramm-{L}oewner evolutions.
\newblock {\em Probab. Theory Related Fields}, 177(3-4):677--709, 2020.

\bibitem{MS16a}
J.~{Miller} and S.~{Sheffield}.
\newblock Liouville quantum gravity and the {B}rownian map {II}: geodesics and
  continuity of the embedding.
\newblock {\em arXiv:1605.03563}, 2016.

\bibitem{MS16b}
J.~{Miller} and S.~{Sheffield}.
\newblock Liouville quantum gravity and the {B}rownian map {III}: the conformal
  structure is determined.
\newblock {\em arXiv:1608.05391}, 2016.

\bibitem{MS16}
J.~Miller and S.~Sheffield.
\newblock Quantum {L}oewner evolution.
\newblock {\em Duke Math. J.}, 165(17):3241--3378, 2016.

\bibitem{ig4}
J.~Miller and S.~Sheffield.
\newblock Imaginary geometry {IV}: interior rays, whole-plane reversibility,
  and space-filling trees.
\newblock {\em Probab. Theory Related Fields}, 169(3-4):729--869, 2017.

\bibitem{sphere}
J.~Miller and S.~Sheffield.
\newblock Liouville quantum gravity spheres as matings of finite-diameter
  trees.
\newblock {\em Ann. Inst. Henri Poincar\'{e} Probab. Stat.}, 55(3):1712--1750,
  2019.

\bibitem{MS15b}
J.~Miller and S.~Sheffield.
\newblock Liouville quantum gravity and the {B}rownian map {I}: the {${\rm
  QLE}(8/3,0)$} metric.
\newblock {\em Invent. Math.}, 219(1):75--152, 2020.

\bibitem{Mo96}
G.~M. Molchan.
\newblock Scaling exponents and multifractal dimensions for independent random
  cascades.
\newblock {\em Comm. Math. Phys.}, 179(3):681--702, 1996.

\bibitem{P81}
A.~M. Polyakov.
\newblock Quantum geometry of bosonic strings.
\newblock {\em Phys. Lett. B}, 103(3):207--210, 1981.

\bibitem{RV14}
R.~Rhodes and V.~Vargas.
\newblock Gaussian multiplicative chaos and applications: a review.
\newblock {\em Probab. Surv.}, 11:315--392, 2014.

\bibitem{RV17}
R.~Rhodes and V.~Vargas.
\newblock The tail expansion of {G}aussian multiplicative chaos and the
  {L}iouville reflection coefficient.
\newblock {\em Ann. Probab.}, 47(5):3082--3107, 2019.

\bibitem{Shamov16}
A.~Shamov.
\newblock On {G}aussian multiplicative chaos.
\newblock {\em J. Funct. Anal.}, 270(9):3224--3261, 2016.

\bibitem{W19b}
M.~D. {Wong}.
\newblock Tail universality of critical {G}aussian multiplicative chaos.
\newblock {\em arXiv:1912.02755}, 2019.

\bibitem{W19a}
M.~D. Wong.
\newblock Universal tail profile of {G}aussian multiplicative chaos.
\newblock {\em Probab. Theory Related Fields}, 177(3-4):711--746, 2020.

\end{thebibliography}

\end{document}